\documentclass[12pt,a4paper]{amsart}
\makeatletter
\renewcommand\small{%
    \@setfontsize\small{9.5}{12\p@ plus .2\p@ minus .2\p@}%
    \abovedisplayskip 8.5\p@ \@plus4\p@ \@minus1\p@
    \belowdisplayskip \abovedisplayskip
    \abovedisplayshortskip \abovedisplayskip
    \belowdisplayshortskip \abovedisplayskip}
\renewcommand\footnotesize{%
    \@setfontsize\footnotesize{8.5}{9.25\p@ plus .1pt minus .1pt}%%
    \abovedisplayskip 6\p@ \@plus4\p@ \@minus1\p@
    \belowdisplayskip \abovedisplayskip
    \abovedisplayshortskip \abovedisplayskip
    \belowdisplayshortskip \abovedisplayskip}
\ifdefined\pdfpagewidth
\else
\fi
\calclayout
\makeatother

\usepackage[T1]{fontenc}
\usepackage[utf8]{inputenc}
\usepackage{lmodern}
\usepackage{mathtools,booktabs,longtable,array,enumitem,microtype,xcolor}
\usepackage{amsmath,amssymb,amsfonts}%
\usepackage{amsthm}%
\usepackage{mathrsfs}%
\usepackage{tikz-cd}
\usetikzlibrary{arrows.meta,calc,positioning,decorations.pathreplacing,decorations.pathmorphing}
\tikzset{%
    symbol/.style={%
        draw=none,
        every to/.append style={%
            edge node={node [sloped, allow upside down, auto=false]{$#1$}}}
    }
}
\usepackage{adjustbox}
\usepackage[cmtip,all]{xy}
\newtheorem{theorem}{Theorem}[section]
\newtheorem{proposition}[theorem]{Proposition}% 

\newtheorem{lemma}[theorem]{Lemma}

\newtheorem{term}[theorem]{Terminology}
\newtheorem{conjecture}[theorem]{Conjecture}

\newtheorem{definition}[theorem]{Definition}%
\newtheorem{example}[theorem]{Example}%
\newtheorem{remark}[theorem]{Remark}%
\newtheorem{corollary}[theorem]{Corollary}

\newcommand{\CC}{\mathbb{C}}
\renewcommand{\P}{\mathbb{P}}

\renewcommand{\O}{\mathcal{O}}

\newcommand{\T}{\mathbb{T}}
\newcommand{\ttL}{\mathbb{L}}

\newcommand{\cP}{\mathcal{P}}

\newcommand{\supp}{\operatorname{supp}}
\newcommand{\M}{\mathcal{M}}

\newcommand{\Hom}{\operatorname{Hom}}

\newcommand{\ext}{\operatorname{Ext}}

\newcommand{\A}{\mathbb{A}}

\newcommand{\cX}{\mathcal{X}}

\usepackage{needspace}
\usepackage{hyperref, cleveref}
\numberwithin{equation}{section}

\let\IHESdoubleacute\H
\renewcommand{\P}{\mathbb{P}}

\renewcommand{\O}{\mathcal{O}}

\newcommand{\R}{\mathrm R}
\newcommand{\LLf}{\mathrm L}
\newcommand{\Perf}{\operatorname{Perf}}
\newcommand{\RPerf}{\mathfrak{Perf}}

\newcommand{\RH}{\operatorname{RHom}}
\newcommand{\cRHom}{\mathcal{R}\!\operatorname{Hom}}
\newcommand{\End}{\operatorname{End}}
\newcommand{\Tor}{\operatorname{Tor}}
\newcommand{\Spec}{\operatorname{Spec}}

\newcommand{\Sym}{\operatorname{Sym}}

\newcommand{\dCrit}{\operatorname{dCrit}}
\newcommand{\MF}{\operatorname{MF}}

\newcommand{\CyclicChains}{\operatorname{CC}}
\newcommand{\fib}{\operatorname{fib}}

\newcommand{\Cone}{\operatorname{Cone}}
\newcommand{\tr}{\operatorname{tr}}
\newcommand{\id}{\operatorname{id}}
\newcommand{\rk}{\operatorname{rk}}
\newcommand{\Map}{\operatorname{Map}}
\newcommand{\Quot}{\operatorname{Quot}}

\newcommand{\At}{\operatorname{At}}
\newcommand{\gr}{\operatorname{gr}}
\newcommand{\DR}{\operatorname{DR}}

\newcommand{\otL}{\mathbin{\otimes^{\mathbf L}}}
\newcommand{\timesh}{\mathbin{\times^{h}}}

\newcommand{\fM}{\mathfrak M}
\newcommand{\fG}{\mathfrak G}

\newcommand{\cH}{\mathcal H}
\newcommand{\CE}{\operatorname{CE}}

\newcommand{\Gm}{\mathbb G_m}

\newcommand{\an}{\mathrm{an}}

\newcommand{\eff}{\mathrm{eff}}

\newcommand{\RR}{\mathbb R}

\newcommand{\Pic}{\operatorname{Pic}}

\newcommand{\Bl}{\operatorname{Bl}}
\newcommand{\GL}{\operatorname{GL}}
\newcommand{\PGL}{\operatorname{PGL}}

\theoremstyle{definition}
\newtheorem{hypothesis}[theorem]{Hypothesis}

\theoremstyle{remark}
\newtheorem{problem}[theorem]{Problem}
\theoremstyle{plain}
\DeclareFontFamily{U}{rsfs}{\skewchar\font127}
\DeclareFontShape{U}{rsfs}{m}{n}{<-6.5> rsfs5 <6.5-7.5> rsfs7 <7.5-> rsfs10}{}
\setlist[enumerate]{itemsep=3pt,topsep=5pt}
\hypersetup{hidelinks,bookmarksnumbered=true,
 pdftitle={Good degeneration of Derived Quot schemes and Categorification of Donaldson--Thomas Invariants},
 pdfauthor={Jacob Kryczka and Artan Sheshmani},
 pdfsubject={Geometric constructions and conjectures for derived Donaldson--Thomas degeneration}}
\makeatletter
\newcommand{\IHEScopylabel}[2]{%
  \@ifundefined{r@#2}{}{%
    \protected@write\@auxout{}{\string\newlabel{#1}{\csname r@#2\endcsname}}}}
\newcommand{\IHESaliaslabel}[2]{%
  \AtEndDocument{\IHEScopylabel{#1}{#2}\IHEScopylabel{#1@cref}{#2@cref}}}
\makeatother

\title[Degeneration of Derived Quot schemes and categorical DT]{Degeneration of Derived Quot schemes, Lagrangian Foliations, and categorical Donaldson-Thomas invariants}
\author[J. Kryczka]{Jacob Kryczka}
\address{Jacob Kryczka}
\address{Beijing Institute of Mathematical Sciences and Applications, A4, Room 209
No. 544, Hefangkou Village, Huaibei Town, Huairou District, Beijing 101408,
China}
\email{jkryczka@bimsa.cn}

\author[A. Sheshmani]{Artan Sheshmani}
\address{Artan Sheshmani}
\address{Beijing Institute of Mathematical Sciences and Applications, A6, Room 205
No. 544, Hefangkou Village, Huaibei Town, Huairou District, Beijing 101408,
China}
\address{Massachusetts Institute of Technology (MIT), IAiFi Institute, 182 Memorial Drive, Cambridge, MA 02139, USA}
\address{National Research University, Higher School of Economics, Russian Federation, Laboratory of Mirror Symmetry, NRU HSE, 6 Usacheva Street, Moscow, Russia, 119048}
\email{artan@mit.edu}

\date{October 20, 2026.}

\providecommand{\Supp}{\operatorname{Supp}}
\providecommand{\id}{\operatorname{id}}

\begin{document}
\begin{abstract}
We develop a derived-geometric framework for Donaldson--Thomas
invariants in semistable Calabi--Yau threefold degenerations. Derived
restriction from the two smooth components of the central fiber to
their common anticanonical divisor is Lagrangian. We give explicit
universal-family diagrams for this structure, compare the deformation
complexes with perfect complexes on the singular union, and construct
the fourfold duality using a canonical bundle pulled back from the
base. On specified Quot presentations we describe the real
Lagrangian foliation and its shifted potential.

We formulate proper-completion and cohomological-comparison
conjectures for sheaves arising from sections of an extending line
bundle. Their specialization is described by oriented two-sided bar
complexes and a finite Gauss--Manin system. 
\end{abstract}
\keywords{Tyurin degenerations, derived Quot schemes, derived Lagrangian intersections, cohomological Donaldson--Thomas invariants}
\maketitle
\tableofcontents
\section{Introduction}\label{sec:intro}
Let $\pi:\cX\to C$ be a projective flat degeneration of Calabi--Yau
threefolds over a smooth pointed complex curve $(C,0)$, with smooth
fourfold total space and special fiber
\begin{equation}\label{eq:I-family}
 X_0=X_1\cup_S X_2.
\end{equation}
The threefolds $X_1,X_2$ and the surface $S$ are smooth. We choose
trivializations of $\omega_{X_i}(S)$ with opposite residues on $S$.
Our purpose is to formulate a degeneration construction for
cohomological Donaldson--Thomas invariants using derived restriction
to these fixed components. The geometric constructions are proved
below; the global comparison statements are formulated separately as
conjectures with a specified moduli problem.

For a smooth fiber $X_b$, write $\mathfrak X_b=\RPerf(X_b)$, and put
$\mathfrak X_i=\RPerf(X_i)$ and $\mathfrak S=\RPerf(S)$.
The special-fiber construction is the derived intersection
\begin{equation}\label{eq:I-main-intersection}
 \mathfrak G=\mathfrak X_1\times^h_{\mathfrak S}\mathfrak X_2.
\end{equation}
An object is a triple $(E_1,E_2,\alpha)$, where
$E_i\in\Perf(X_i)$ and
$\alpha:\LLf\iota_1^*E_1\xrightarrow{\sim}\LLf\iota_2^*E_2$ for
$\iota_i:S\hookrightarrow X_i$. Morphisms and higher homotopies are
retained. In particular, this construction is a fiber product of
moduli stacks, not a tensor product of the categories $\Perf(X_i)$
over $\Perf(S)$.

\subsection{Derived restriction without expanded degenerations}
Applications include pure two-dimensional sheaves scheme-theoretically
supported on divisors belonging to a fixed linear system on $X_b$.
Such sheaves on a smooth fiber are perfect complexes concentrated in
one degree. Their derived restrictions on the central components need
not be sheaves. For a coherent sheaf $F_i$ on $X_i$, the Cartier
resolution of $S$ gives
\[
 \iota_{i*}\LLf\iota_i^*F_i
   \simeq[F_i(-S)\longrightarrow F_i],\qquad
 \cH^{-1}(\LLf\iota_i^*F_i)
   =\Tor^{\O_{X_i}}_1(F_i,\O_S).
\]
We retain this negative cohomology, rather than imposing a
transversality condition to remove it.

Li--Wu construct degenerations of Quot schemes and coherent systems,
and obtain the ideal-sheaf DT and Pandharipande--Thomas degeneration
formulas~\cite{LiWu2015}. Those formulas concern their specified
relative moduli functors. They do not by themselves give a
degeneration formula for arbitrary Gieseker-stable pure
two-dimensional sheaves. An expansion changes the target and
introduces ruled components; admissibility and stability on that
target need not preserve the chosen sheaf-moduli problem. In the
construction here, no expanded degeneration or bubble is introduced.
Instead the full derived restrictions are matched over
$\mathfrak S$. Boundedness, stability and properness are additional
questions for the selected loci of~\eqref{eq:I-main-intersection}.

\subsection{Geometric results}
The first result constructs the Lagrangian structures used in the
central term. All assertions are made on compatible locally geometric
finite-presentation loci with perfect cotangent complexes.

\begin{theorem}\label{thm:I-intro-restriction}
Let $X_1,X_2$ be smooth projective threefolds with a common smooth
anticanonical divisor $S$ and compatible residue orientations. The
restriction morphisms
\[
 r_i:\mathfrak X_i\longrightarrow\mathfrak S
\]
are Lagrangian for a common target orientation, and
$\mathfrak G$ is $(-1)$-shifted symplectic. At $E_i$,
\[
 \mathbb T_{r_i,E_i}
 \simeq\RH_{X_i}(E_i,E_i(-S))[1]
 \simeq\mathbb L_{\mathfrak X_i,E_i}[-1].
\]
For $E,F\in\Perf(X_0)$, derived restriction gives the exact triangle
\[
 \RH_{X_0}(E,F)\longrightarrow
 \bigoplus_{i=1}^2\RH_{X_i}(E_i,F_i)
 \longrightarrow\RH_S(E_S,F_S)\xrightarrow{+1}.
\]
Restriction is fully faithful. A matching triple represented locally
by bounded vector-bundle complexes with a degreewise isomorphism on
$S$ is induced by a perfect complex on $X_0$.
\end{theorem}

The boundary-integration construction of Calaque and the intersection
theorem of Pantev--To\"en--Vaqui\'e--Vezzosi supply the shifted-symplectic
inputs~\cite{Calaque15,PTVV13}. Section~\ref{lagrangian-moduli} writes
the residue homotopy, the universal endomorphism complexes and the
Serre-duality diagrams explicitly. Section~\ref{sec:CY} proves the
mapping-complex and effectivity statements. The second author would
like to express his gratitude for invaluable correspondence he had
in 2015 with Oren Ben-Bassat, who taught him this point of view, and
communicated with him in private some major parts of the proof of
this theorem; see also~\cite{Amorim_2017}.

The comparison with deformations in the fourfold requires a different
pairing. The normal deformation directions are retained in this
step, and the canonical bundle must be trivialized near the support.

\begin{theorem}\label{thm:I-intro-fourfold}
Suppose $\omega_{\cX}\simeq\pi^*\mathcal B$, where $\mathcal B$ is a
line bundle on $C$. Let $Z\subset C$ be finite and choose a frame of
$\mathcal B$ on a neighborhood of $Z$. Perfect complexes whose
cohomology sheaves are supported in $\pi^{-1}Z$ have natural perfect
pairings
\[
 \RH_{\cX}(E,F)\otimes\RH_{\cX}(F,E)\longrightarrow\CC[-4].
\]
The induced volume form gives the proper-support moduli stack its
$(-2)$-shifted symplectic form. On the stable Quot presentations
satisfying Hypothesis~\ref{hyp:fourfold-realization}, the associated
real quotient has the BKSY Lagrangian foliation and a degree-$-1$
section--cosection potential.
\end{theorem}

Proposition~\ref{prop:supported-serre} and
Corollary~\ref{cor:local-CY4} prove the duality and
closed-form assertions. Section~\ref{sec:local} gives the Quot tangent
triangle and the strictification and gluing argument. The foliation
is used in its real derived setting~\cite{BKSY2025}; it does not
supply a holomorphic product of functions without a separate
comparison. In the resolved quintic example,
\[
 \cX=\Bl_{(h,t)}\{hq_4=tq_5\}\longrightarrow C,
 \qquad \omega_{\cX/C}\simeq\O_{\cX},\quad
 \omega_{\cX}\simeq\pi^*\omega_C.
\]
The modified component is $\Bl_{(4,5)}\P^3$, with its strict-transform
anticanonical divisor. No Fano property of the total space is used.

\subsection{The comparison conjectures and their geometric scope}
We specify a class of objects before formulating a regular
specialization statement. Let $L$ be a line bundle on the family and
fix $m>0$. Assume that $L_b$ is ample and
$\Pic(X_b)=\mathbb Z[L_b]$ for $b\ne0$, and that
\begin{equation}\label{eq:I-section-assumptions}
\begin{gathered}
 H^1(X_b,\O_{X_b})=H^2(X_b,\O_{X_b})=0,\\
 H^j(X_b,L_b^m)=0\quad(j>0),\\
 \mathcal V_m=\pi_*L^m\text{ is locally free of rank }r_m>0
 \text{ and commutes with base change}.
\end{gathered}
\end{equation}
We also choose a relative volume form. These assumptions include the
central fiber. Write $M_m=\P_C(\mathcal V_m)$, with the convention
that this projective bundle parametrizes one-dimensional subspaces,
and $d_m=r_m-1$. For an extending twist $L^k$, the universal
section gives the perfect complex
\[
 \mathcal E^{m,k}
   =[L^{k-m}\otimes\O_{M_m}(-1)\longrightarrow L^k]
       \quad\text{on }\cX\times_C M_m.
\]

\begin{conjecture}\label{conj:I-completion}
Under~\eqref{eq:I-section-assumptions}, the family $\mathcal E^{m,k}$
induces an open immersion of $M_m$ into the classical
scalar-rigidified moduli space of simple perfect complexes on the
fibers of $\pi$. It is a smooth proper relative moduli locus. On a
smooth fiber its objects are the stable pure two-dimensional sheaves
$\O_D\otimes L_b^k$, with $D\in|L_b^m|$. Its central objects are
the displayed two-term perfect complexes and their entire derived
restriction triples in $\mathfrak G$.
\end{conjecture}

The conjecture concerns this specified completion. It does not
assert that every central point is a sheaf, that every restriction to
$X_i$ is simple, or that a polynomial stability condition on all
perfect complexes has been constructed. Those stronger requirements
would define a different moduli problem.

\begin{conjecture}\label{conj:I-regular-comparison}
For the moduli loci of Conjecture~\ref{conj:I-completion}, the
relative orientation $K_{M_m/C}$ identifies the fiberwise d-critical
DT sheaves with $\CC_{M_{m,b}}[d_m]$. The oriented two-sided bar
complex of their central derived restriction locus has a global
comparison with
\[
 (\Omega_{M_m/C}^{\bullet}((u)),u\,d)[d_m].
\]
These comparisons commute with specialization. Their finite
cohomology has the horizontal basis $1,H,\ldots,H^{d_m}$, where
$H=c_1(\O_{M_m}(1))$, and a Gauss--Manin connection without a pole
at $0$. In particular
\[
 \mathrm{DT}_{\mathrm{coh}}^j(M_{m,b})
   \simeq H^{j+d_m}(\P^{d_m},\CC)
   \simeq\mathrm{DT}_{\mathrm{coh}}^j(M_{m,0}).
\]
The normalized cyclic realization is
$\CyclicChains_{\mathrm{per}}(\MF(M_{m,b},0))[d_m]$.
\end{conjecture}

The companion paper~\cite[Theorems 1.13 and 7.12]{CompanionII}
proves these conjectures for the resolved quintic family, for every
$m>0$ and $k\in\mathbb Z$. There
\[
 r_m=\binom{m+4}{4}-\binom{m-1}{4},\qquad d_m=r_m-1,
\]
where the second binomial coefficient is zero for $m<5$. The
assertion includes the defining-equation relations from degree five
onward and the parity shift when $d_m$ is odd. We formulate the
conjectures in the broader cohomological setting above; the
companion gives their detailed geometric verification in that
family. Its proofs do not require these conjectures as inputs.

The regularity condition cannot be extended to every pure sheaf in
the same linear system. The companion also treats
$F_{x,D}=j_{D*}I_{x/D}$ on a smooth quintic. The incidence moduli are
a projective bundle over $X_b$ and retain four copies of $H^3(X_b)$.
The computed odd monodromy logarithm has rank $84$
\cite[Theorems 1.14 and 7.14]{CompanionII}. Thus a pole-free
extension is excluded for this moduli problem. Section~\ref{sec:cyclic}
formulates the remaining logarithmic comparison question, retaining
its residue and nearby-cycle fiber. We do not replace this nearby
fiber by an ordinary central DT complex without constructing the
corresponding map.

\subsection{Organization and conventions}
Section~\ref{sec:geometry} defines the family and the derived
restriction maps, and proves their Lagrangian structures with the
universal trace diagrams. Section~\ref{sec:CY} compares the
endomorphism complexes on $X_0$, the relative family and the
fourfold. Section~\ref{sec:local} develops the Quot presentation,
real foliation and local deformation algebras. Section~\ref{sec:cyclic}
formulates the finite regular and logarithmic comparison problems.
Section~\ref{sec:quantization} constructs the bar and quantization
models entering those questions, with their complete-lattice
hypotheses. Appendix~\ref{app:checks} records the universal
Atiyah-class maps.

Complexes are cohomologically graded. The symbols $X_b,X_i,S$
denote varieties, and fraktur symbols denote derived moduli stacks.
We use $\RH$ for global derived Hom, $\cRHom$ for its sheaf-valued
version, $\otL$ for derived tensor products and $\times^h$ for
homotopy fiber products. The ordinary critical potentials have degree
zero; the fourfold section--cosection potential has degree $-1$.
The variables $u$, $\hbar$ and the geometric coordinate $t\in C$
have distinct roles.
\subsection*{Acknowledgements}
The authors are grateful to Maxim Kontsevich, Ludmil Katzarkov, and Vladimir Baranovsky who were initially joined as coauthors in the earlier stages of this project, for many valuable discussions and comments that helped shape the work in the article in later stages. Throughout the stages of this work, J.K. was supported by the Postdoctoral International Exchange Program of Beijing Municipal Human Resources and Social Security Bureau and subsequently by the Simons Foundation, grant SFI-MPS-T-Institutes-
00007697, and the Ministry of Education and Science of the Republic of Bulgaria, grant DO1-
239/10.12.2024. A.S. is supported by grants from Beijing Institute of Mathematical Sciences and Applications (BIMSA), the Beijing NSF BJNSF-IS24005, and the China National Science Foundation (NSFC) NSFC-RFIS program W2432008. He would like to thank China's National Program of Overseas High Level Talent for generous support. He would like to thank Sheldon Katz, Charles Doran, Tony Pantev, Shing-Tung Yau, Martijn Kool, Marco Robalo and Yukinobu Toda for many valuable discussions. Special thanks go to Bertrand To\"{e}n for invaluable comments on shifted symplectic structures. He would also like to thank Simons Center for Geometry and Physics for hospitality.

\section{Shifted symplectic structures over degenerating families}\label{sec:geometry}\IHESaliaslabel{sec: shifted-structure}{sec:geometry}

We first construct the central moduli stack
$\mathfrak G=\mathfrak X_1\times^h_{\mathfrak S}\mathfrak X_2$.
The Cartier restriction calculation retains the Tor sheaves, the
quintic model supplies the smooth fourfold and its orientation, and
the final universal-family diagrams prove
Theorem~\ref{thm:I-intro-restriction}.

\subsection{Derived restriction from \texorpdfstring{$X_1$ and $X_2$ to $S$}{X1 and X2 to S}}\label{ssec: RelSheavesAndDegens}
Let $\pi:\cX\to C$ be a simple degeneration, with central fiber $X_0=X_1\cup_S X_2$, and write $\iota_i:S\hookrightarrow X_i$. We work throughout on these two components. The stability conditions used to select the objects $E_i$ are conditions on the fixed threefolds $X_i$, for $i=1,2$. Matching is imposed on their derived restrictions to $S$. In particular, a nonzero Tor sheaf in a restriction is part of the moduli datum.

Near $S$, a local equation for the total space is $xy=t$, where $t$ is a local parameter on the base curve $C$ at $0$. The normal line bundles satisfy
\[
 N_{S/X_1}\otimes N_{S/X_2}
       \simeq\O_S\otimes_{\CC}T_{C,0}.
\]
Choosing $t$ identifies $T_{C,0}$ with $\CC$ and identifies these two normal line bundles as dual to one another.

\subsubsection{Normality to a divisor and derived restriction}
The following calculation identifies the Tor sheaf that must be retained in the matching equivalence. It is also used at the central hyperplane-section complex in \cite[Section~7]{CompanionII}.

\begin{definition}\label{def:admissibility}\IHESaliaslabel{def:rel}{def:admissibility}\IHESaliaslabel{term: Normal}{def:admissibility}
Let $Z$ be a finite-type complex scheme, let $i:D\hookrightarrow Z$ be a closed subscheme, and let $\mathcal I_D=\ker(\O_Z\to i_*\O_D)$ be its ideal sheaf. A coherent sheaf $F$ on $Z$ is normal to $D$ if
\[
 \Tor^{\O_Z}_1(F,i_*\O_D)=0.
\]
For an effective Cartier divisor this is equivalent to injectivity of the global map
\[
 \mu_D:F(-D)=F\otimes\O_Z(-D)\longrightarrow F.
\]
\end{definition}
This terminology is recalled from the relative-sheaf theory \cite{LiWu2015}. Normality to $S$ is not imposed on the sheaves $F_i$ on $X_i$ in the derived fiber product.

Indeed, the locally free resolution
\[
 [\O_Z(-D)\longrightarrow\O_Z]\longrightarrow i_*\O_D
\]
puts its two terms in degrees $-1$ and $0$. Tensoring it with $F$ gives, in the derived category of $Z$,
\begin{equation}\label{eq:restriction-two-term}
 i_*\LLf i^*F\simeq[F(-D)\xrightarrow{\mu_D}F].
\end{equation}
Consequently the cohomology sheaves of the derived restriction on $D$ are
\[
 \cH^{-1}(\LLf i^*F)=\Tor^{\O_Z}_1(F,\O_D),
 \qquad\cH^0(\LLf i^*F)=F\otimes_{\O_Z}\O_D.
\]
There is no other cohomology for a sheaf restricted to a Cartier divisor. The displayed Tor sheaf and tensor product are viewed as sheaves on $D$; their pushforwards appear in \eqref{eq:restriction-two-term}.

\begin{lemma}\label{lem:Cartier-normal}
For an effective Cartier divisor $D\subset Z$, the derived restriction $\LLf i^*F$ is concentrated in degree zero if and only if $\mu_D:F(-D)\to F$ is injective. In that case
\[
 0\longrightarrow F(-D)\longrightarrow F
   \longrightarrow i_*(F|_D)\longrightarrow0
\]
is exact. In the general case, the kernel of $\mu_D$ is retained as the negative cohomology of the derived restriction.
\end{lemma}
\begin{proof}
The two cohomology sheaves of \eqref{eq:restriction-two-term} are the kernel and cokernel of $\mu_D$. Tensoring $0\to\O_Z(-D)\to\O_Z\to i_*\O_D\to0$ with $F$ identifies the kernel with $\Tor_1^{\O_Z}(F,i_*\O_D)$ and the cokernel with $i_*(F|_D)$. This gives all the assertions.
\end{proof}

\begin{example}\label{ex:nontransverse-restriction}
Use the local model $Z=\Spec\CC[x,y,z]$ for a smooth component $X_i$, with its boundary divisor $D=(x=0)$. Take $F=\O_D$ as a sheaf on $Z$. Multiplication by $x$ on $F$ is zero. Thus $F$ is not normal to $D$. Nevertheless $F$ is a perfect complex on the smooth $Z$, and
\[
 \LLf i^*F\simeq[\O_D\xrightarrow{0}\O_D]
                \simeq\O_D[1]\oplus\O_D
\]
is a perfect complex on $D$. It is an allowed restriction in the present moduli problem. The degree $-1$ summand is retained in every matching equivalence and deformation calculation.
\end{example}

\subsubsection{Restriction morphisms between the moduli stacks of perfect complexes}\label{ssec:expanded-background}
Choose open derived moduli substacks $\fM_i\subset\RPerf(X_i)$ containing the objects $E_i$ on $X_i$ under consideration. For a bounded family of stable sheaves these are the specified stable loci; when complexes are used the stability datum and boundedness are fixed on $X_i$. The common target remains
\[
 \fM_S=\RPerf(S).
\]
The morphism $r_i:\fM_i\to\fM_S$ sends a family to its derived restriction. No simplicity, sheaf-concentration, or stability condition on this restriction is needed to define the map to the full target.

More precisely, on a test derived scheme $T$ the restriction is
\[
 E_i\in\Perf(X_i\times^h T)
 \quad\longmapsto\quad
 \LLf(\iota_i\times\id_T)^*E_i\in\Perf(S\times^h T).
\]
Derived pullback preserves perfect complexes: locally it tensors a bounded vector-bundle representative with the target algebra, giving a bounded complex of finite projective modules. The same formula is compatible with a further derived base change $T'\to T$. Hence these restrictions define morphisms of derived moduli stacks, not just maps on closed points.

\begin{definition}\label{def:moduli-fiber-product}
Let $X_1$ and $X_2$ be the two irreducible components of $X_0=X_1\cup_S X_2$, and let $\fM_i\subset\RPerf(X_i)$ and $\fM_S=\RPerf(S)$ be the moduli stacks just specified. The moduli stack used on the special fiber is
\begin{equation}\label{eq:unexpanded-gluing}
 \fG_{\fM_1,\fM_2}=\fM_1\timesh_{\fM_S}\fM_2.
\end{equation}
Its objects over $T$ are triples $(E_1,E_2,\alpha)$ with $E_i\in\fM_i(T)$ and an equivalence
\[
 \alpha:\LLf(\iota_1\times\id_T)^*E_1
             \xrightarrow{\sim}\LLf(\iota_2\times\id_T)^*E_2.
\]
Morphisms are pairs of morphisms of $E_i$ together with their specified compatibility homotopy after restriction. Higher homotopies are retained in this homotopy fiber product.
\end{definition}

For sheaves $F_i$ this includes the possibly nonzero $\Tor_1$ contributions just calculated. Matching only the ordinary tensor products $F_i\otimes\O_S$ would discard that information. We use the full complexes in \eqref{eq:unexpanded-gluing}.

The logarithmic volume forms on the two anticanonical pairs give Lagrangian structures on $r_i$, by the construction in Subsection~\ref{lagrangian-moduli}. Restricting the source of a Lagrangian morphism to an open derived substack preserves its Lagrangian structure. The derived intersection \eqref{eq:unexpanded-gluing} is therefore $(-1)$-shifted symplectic. At a matching triple, with $E_S=\LLf\iota_i^*E_i$ identified by $\alpha$, its tangent complex is the homotopy fiber of
\[
 \RH_{X_1}(E_1,E_1)[1]\oplus\RH_{X_2}(E_2,E_2)[1]
    \longrightarrow\RH_S(E_S,E_S)[1],
\]
where the arrow is the difference of the two restriction maps. Its target uses the derived object $E_S$ and thus contains the boundary Tor information.

\subsubsection{The two-sided bar resolution of the derived intersection}
Take a compatible affine derived chart $\Spec A$ of the target and affine charts $\Spec A_i$ of the two pulled-back sources, giving morphisms of commutative dg algebras $A\to A_i$. On these charts the derived intersection is
\[
 \Spec(A_1\otL_A A_2).
\]
After suitable flat or cofibrant replacements, the normalized two-sided bar resolution represents this derived tensor product. Before normalization its simplicial terms are
\begin{equation}\label{eq:component-bar}
 \operatorname{Bar}_n(A_1,A,A_2)
      =A_1\otimes_{\CC}A^{\otimes_{\CC}n}\otimes_{\CC}A_2,
       \qquad n\ge0.
\end{equation}
The first and last face maps are the two $A$-actions, and the intermediate face maps multiply neighboring $A$-factors. For dg algebras the total differential includes the internal differentials and the usual graded signs. The simplicial degree $n$ contributes cohomological degree $-n$. This is the bar construction used to calculate the derived intersection, independently of whether the intersection or the restrictions of the sheaves are transverse.

For example, if $A=\CC[p]$ and both $A_i=\CC$ are given by $p\mapsto0$, the derived tensor product is represented by $\CC[\theta]$ with $|\theta|=-1$ and zero differential. The nonzero degree $-1$ term is the excess intersection direction. Passing to the ordinary tensor product would retain only $\CC$. This elementary calculation explains why the two-sided bar resolution, rather than the ordinary intersection, is used.

The structure-algebra complex in \eqref{eq:component-bar} and the complex computing DT cohomology have different roles. The former records the derived geometry. On a critical chart $\operatorname{dCrit}(f)$ it has the Koszul model $(K_f,\delta_f)$. The local DT calculation uses the mixed BV model
\[
 (K_f((u)),-\delta_f+u\Delta),
\]
with the operators constructed in Proposition~\ref{prop:local-BV}. In the smooth oriented Lagrangian model, the quantized two-sided bar resolution gives the global complex of Section~\ref{sec:quantization}. These constructions specify the differential as well as the underlying derived tensor product.

\subsubsection{Cohomology of the derived intersection of the two moduli stacks}
Let $M_0$ be the classical truncation of the chosen derived intersection locus, with the finite-type and orientation data of Hypothesis~\ref{hyp:DT}. Its oriented DT complex is denoted $\cP_0$. The cohomological DT groups assigned to this locus are
\[
 \mathrm{DT}_{\mathrm{coh}}^k(M_0)
       =\mathbb H_c^k(M_0^{\an},\cP_0).
\]
On the specified oriented loci, Section~\ref{sec:quantization} constructs the
periodic bar/BV representative of these groups. On a proper locus compactly
supported and ordinary hypercohomology agree. The comparison theorem applies
when this central complex and the smooth-fiber cyclic complexes have the
fiber identifications in \cite[Hypothesis~5.11]{CompanionII}. We now construct
the underlying geometric family and its orientation.

\subsection{Calabi--Yau hypersurfaces and good degenerations}
The foliation construction requires a volume form on the fourfold total
space, while the restriction theorem uses volume forms on its fibers and
boundary. We obtain both from adjunction. Let $Y$ be the fixed smooth
projective ambient fourfold and $\cX'\subset Y\times C$ the hypersurface
family. The following normal-bundle calculation distinguishes an embedding
$X_b'\hookrightarrow Y$ from the fiber embedding $X_b\hookrightarrow\cX$.

\begin{lemma}\label{lem:normal-bundles}
For a smooth anticanonical hypersurface $X\subset Y$, one has
\[
 N_{X/Y}\simeq\omega_Y^{-1}|_X,\qquad \omega_X\simeq\O_X.
\]
For a fiber $i_b:X_b\hookrightarrow\cX$ of a flat morphism to a smooth curve,
\[
 N_{X_b/\cX}\simeq\O_{X_b}\otimes T_{C,b}.
\]
If $\cX$ is smooth and $X_b$ is a positive-dimensional smooth projective Calabi--Yau fiber, then $\cX$ is not Fano.
\end{lemma}
\begin{proof}
The first identities are adjunction for an anticanonical effective Cartier divisor. The conormal of a fiber is the pullback of $\mathfrak m_b/\mathfrak m_b^2$, giving the second identity, without a preferred trivialization of the one-dimensional vector space $T_{C,b}$. Adjunction then gives $\omega_{\cX}|_{X_b}\simeq\omega_{X_b}\otimes T_{C,b}^{\vee}$, a trivial line bundle up to a constant one-dimensional factor. An ample anticanonical line bundle on $\cX$ would restrict to an ample trivial line bundle on $X_b$, which is impossible.
\end{proof}

\begin{definition}\label{def:good}\IHESaliaslabel{term: Good degeneration}{def:good}
An oriented semistable Calabi--Yau degeneration is a projective flat morphism $\pi:\cX\to C$ from a smooth fourfold to a smooth pointed curve, with smooth projective Calabi--Yau fibers over $C^*=C\setminus\{0\}$, a reduced simple normal-crossing fiber $X_0=X_1\cup_S X_2$, and a specified isomorphism $\vartheta:\O_{\cX}\simeq\omega_{\cX/C}$. The components are smooth and $S$ is connected.
\end{definition}
To obtain the chosen orientation from fiberwise dualizing sheaves, we identify the base line bundle whose pullback is $\omega_{\cX/C}$. Its absolute version supplies the fourfold pairing in Proposition~\ref{prop:supported-serre}.

\begin{proposition}\label{prop:base-line}\IHESaliaslabel{lem: Pull-back divisor}{prop:base-line}
Let $\pi:\cX\to C$ be projective, flat, and Gorenstein of relative dimension three, with geometrically connected and geometrically reduced fibers. Suppose $\omega_{X_b}\simeq\O_{X_b}$ for every geometric point $b$, and cohomology and base change hold for $\pi_*\omega_{\cX/C}$ in degree zero. Then
\[
 \lambda:=\pi_*\omega_{\cX/C}
 \quad\text{is a line bundle,}\qquad
 \pi^*\lambda\xrightarrow{\sim}\omega_{\cX/C}.
\]
In particular, when $\cX$ is smooth,
\begin{equation}\label{eq:canonical-total}\IHESaliaslabel{eqn: Global isom}{eq:canonical-total}
 \omega_{\cX}\simeq\pi^*(\omega_C\otimes\lambda).
\end{equation}
A frame of $\lambda$ gives a relative orientation. A frame of $\omega_C\otimes\lambda$ is additional data needed to trivialize the absolute canonical bundle.
\end{proposition}
\begin{proof}
On each fiber, $h^0(\omega_{X_b})=h^0(\O_{X_b})=1$. The base-change assumption makes $\lambda$ invertible. The evaluation map restricts to the evaluation of the one-dimensional space of sections of a trivial line bundle on a connected fiber; that map is an isomorphism. Its kernel and cokernel vanish by Nakayama's lemma, locally on $\cX$. Formula~\eqref{eq:canonical-total} follows from the definition of the relative dualizing bundle.
\end{proof}
\begin{remark}
In the projective flat setting over a smooth complex curve, the constant-dimension cohomology-and-base-change theorem supplies the hypothesis in Proposition~\ref{prop:base-line}. The line bundle $\lambda$ records the transition between local relative volume forms.
\end{remark}

\subsubsection{Anticanonical pairs and degeneration of the quintic threefold}
Suppose $\omega_Y^{-1}=L_1\otimes L_2$ and $Y_i=(s_i=0)$ for $s_i\in H^0(Y,L_i)$. Then
\begin{equation}\label{eq:cross-adjunction}
 \omega_{Y_1}^{-1}=L_2|_{Y_1},\qquad
 \omega_{Y_2}^{-1}=L_1|_{Y_2}.
\end{equation}
If the sections meet transversally, $S=Y_1\cap Y_2$ is anticanonical in both. This alone does not give a semistable total space.

\begin{lemma}\label{lem:dsemistable}
For a simple normal-crossing central fiber of a smooth total space,
\[
 N_{S/X_1}\otimes N_{S/X_2}\simeq\O_S.
\]
If $S$ is positive dimensional and anticanonical in two projective Fano varieties, the two normal bundles are ample and this identity cannot hold. Semistable models therefore generally have components that are anticanonical pairs but not both Fano.
\end{lemma}
\begin{proof}
Locally the family has equation $xy=t$. Transition functions for the two normal directions multiply to the transition function for the pullback of the tangent space of the base at $0$, which is a constant line bundle on $S$. Equivalently, restricting $\O_{\cX}(X_1+X_2)=\pi^*\O_C(0)$ to $S$ gives the formula. For an anticanonical divisor in a Fano variety, $N_{S/X_i}=\omega_{X_i}^{-1}|_S$ is ample. A tensor product of ample line bundles cannot be trivial on a positive-dimensional projective variety.
\end{proof}

\begin{proposition}\label{prop:quintic-model}\IHESaliaslabel{exceptional-divisor}{prop:quintic-model}
Take general homogeneous forms $h,q_4,q_5$ on $\P^4$ of degrees $1,4,5$. Over a sufficiently small neighborhood of $0$, consider
\[
 \cX'=\{h q_4=t q_5\}\subset\P^4\times C.
\]
Let $S=(h=q_4=0)$ and $\Gamma=(h=q_4=q_5=0)$, with $S$ smooth and $\Gamma$ a smooth curve. The small resolution obtained by blowing up the ideal $(h,t)$ has smooth total space and central fiber
\[
 X_1=(q_4=0),\qquad X_2=\operatorname{Bl}_\Gamma\P^3,
\]
meeting along the strict transform of $S$, naturally isomorphic to $S$. Both pairs satisfy $\omega_{X_i}(S)\simeq\O_{X_i}$, and their normal bundles along $S$ are inverse. The relative dualizing bundle of the resolved family is trivial.
\end{proposition}
\begin{proof}
The possible total-space singularities on the central fiber occur where $h=q_4=q_5=0$. Near such a point, transversality supplies local coordinates $a=h$, $b=q_4$, $c=q_5$, and a fourth coordinate along $\Gamma$. The equation becomes $ab=tc$. On the chart $t=a\tau$ of the blowup, the strict transform is $b=\tau c$ and the map to the base is $t=a\tau$. On the chart $a=t\alpha$, it is $c=\alpha b$. Both charts are smooth. Away from $\Gamma$ the modification is an isomorphism. The central components in the first chart are $a=0$ and $\tau=0$; the second chart supplies the other standard blowup chart for the component $a=0$. This identifies the components as stated.

Write $H$ for the hyperplane class on $\P^3$ and $E$ for the exceptional divisor of $X_2$. The strict transform of $S$ has class $4H-E$, whereas
\[
 K_{X_2}=-4H+E.
\]
For $X_1$, adjunction gives $K_{X_1}=-H|_{X_1}$ and $S=H|_{X_1}$. Because $\Gamma$ is a Cartier divisor of class $5H$ on $S$, blowing up $S$ along $\Gamma$ leaves $S$ unchanged, and
\[
 N_{S/X_1}=\O_S(1),\qquad
 N_{S/X_2}=\O_S(4H-\Gamma)=\O_S(-1).
\]
Finally, $\cX'$ is anticanonical relative to $C$ inside $\P^4\times C$, so its relative dualizing bundle is trivial. The exceptional locus of the resolution has dimension two in a fourfold; the resolution has no exceptional divisors. The trivialization of the dualizing bundle away from this codimension-two locus extends across it on the smooth resolution. Thus the resolution is crepant and $\omega_{\cX/C}\simeq\O_{\cX}$.
\end{proof}

\begin{remark}
We compare the family over the affine base curve with its projective
pencil. Write $[s_0:s_1]$ for homogeneous coordinates on $\P^1$.
The equation
\[
 \overline{\cX}'=\{s_0h q_4=s_1q_5\}
       \subset\P^1\times\P^4
\]
defines a hypersurface of bidegree $(1,5)$ and a projective morphism
$\overline\pi':\overline{\cX}'\to\P^1$.
On the chart $s_0\ne0$, the coordinate $t=s_1/s_0$ recovers the equation
$h q_4=tq_5$. Thus the construction extends the family from an open
subset of $\A^1$ to the projective base curve $\P^1$.
Adjunction gives
\[
 \omega_{\overline{\cX}'}
   \simeq\O_{\overline{\cX}'}(-1,0),\qquad
 \omega_{\overline{\cX}'/\P^1}
   \simeq\O_{\overline{\cX}'}(1,0).
\]
Accordingly $\lambda=\O_{\P^1}(1)$ in
Proposition~\ref{prop:base-line}, when its hypotheses hold.
A crepant resolution preserves the absolute canonical-bundle formula.
Additional singular fibers may occur over $\P^1$; the semistable
construction of Proposition~\ref{prop:quintic-model} concerns a
neighborhood of the distinguished point $0$.
\end{remark}

\subsubsection{The boundary after the blowup}
The charts in Proposition~\ref{prop:quintic-model} identify the central
components of the resolved family. We will also need the canonical and
normal line bundles of the modified component when applying the
Lagrangian restriction theorem. The following blowup formulas compute
these bundles using the strict transform of $S$.

\begin{lemma}\label{lem:blowup-boundary}
Let $\rho:\widetilde Z=\Bl_{\Gamma}Z\to Z$ be the blowup of a smooth threefold along a smooth curve $\Gamma$ contained in a smooth divisor $S$, and let $E$ be its exceptional divisor. Then
\[
 K_{\widetilde Z}=\rho^*K_Z+E,\qquad
 \widetilde S=\rho^*S-E.
\]
In particular $K_Z+S=0$ implies $K_{\widetilde Z}+\widetilde S=0$. Moreover
\[
 N_{\widetilde S/\widetilde Z}=N_{S/Z}\otimes\O_S(-\Gamma).
\]
\end{lemma}
\begin{proof}
The center has codimension two, so the canonical discrepancy is one. The total transform of $S$ contains the exceptional divisor with multiplicity one because $S$ is smooth along $\Gamma$. Restricting the strict-transform line bundle to $\widetilde S\simeq S$ gives the normal-bundle identity. Adding the first two equalities proves the anticanonical statement.
\end{proof}
In the quintic example, $N_{S/X_1}=\O_S(1)$ and the initial normal line bundle in $\P^3$ is $\O_S(4)$. Since $\Gamma$ has class $\O_S(5)$, the transformed normal line bundle is $\O_S(-1)$, giving the required product $\O_S$. This checks semistable compatibility without declaring both modified components Fano. The blowup is an anticanonical pair; ampleness of its anticanonical line bundle would be an additional condition and is not used.

The compactified-pencil calculation above gives
$\omega_{\overline{\cX}}=\overline\pi^*\O_{\P^1}(-1)$, preserved under
the crepant modification. On an affine neighborhood of $0$, choose its
trivialization and work with objects having proper support in $\cX$.
The two resulting moduli constructions enter the comparison through
different tangent complexes:

\begin{remark}\label{rem:two-degrees}
Let $i_b:X_b\hookrightarrow\cX$ and $E\in\Perf(X_b)$.
The two moduli functors have tangent complexes
\[
 \begin{aligned}
 \T_{\RPerf(\cX/C)/C,E}
     &\simeq\RH_{X_b}(E,E)[1],\\
 \T_{\RPerf_c(\cX),i_{b*}E}
     &\simeq\RH_{\cX}(i_{b*}E,i_{b*}E)[1].
 \end{aligned}
\]
Here $\RPerf_c(\cX)$ denotes the moduli stack of perfect complexes
with proper support on $\cX$. If $N_b=N_{X_b/\cX}$, their comparison
is the exact triangle
\[
 \begin{aligned}
 \RH_{X_b}(E,E)[1]&\longrightarrow
 \RH_{\cX}(i_{b*}E,i_{b*}E)[1]\\
 &\longrightarrow\RH_{X_b}(E,E\otimes N_b)
       \xrightarrow{+1}.
 \end{aligned}
\]
The relative threefold orientation gives the first tangent complex
its relative $(-1)$-shifted pairing. A chosen fourfold volume form
gives the second tangent complex its absolute $(-2)$-shifted pairing.
Proposition~\ref{prop:normal-obstruction} identifies the connecting
morphism in this triangle.
\end{remark}

\subsection{\texorpdfstring{$\O$}{O}-compactness over good degenerations}\label{ssec: O-compactness}

The residue form must give a pairing compatible with families of objects, not a separate choice at each point. We formulate that requirement as a relative trace on perfect pushforwards. This is the orientation used in the universal-family proof of the restriction theorem.

\begin{definition}\label{def:relative-orientation}
Let $\pi:Z\to T$ be a proper perfect morphism. In the applications below it is projective, flat, and Gorenstein of pure relative dimension $d$. An $\O$-orientation of degree $d$ is a morphism
\[
 [Z/T]:\R\pi_*\O_Z\longrightarrow\O_T[-d]
\]
whose induced pairing identifies $\R\pi_*E$ with the shifted dual of $\R\pi_*E^\vee$ for every perfect $E$, compatibly with derived base change. The relevant $\O$-compactness includes perfectness of these pushforwards and their base-change compatibility.
\end{definition}
Under the geometric hypotheses, the relative dualizing complex is $\omega_{Z/T}[d]$ and Grothendieck duality gives the trace
\[
 \R\pi_*\omega_{Z/T}[d]\longrightarrow\O_T.
\]
A specified trivialization $\O_Z\simeq\omega_{Z/T}$ supplies the orientation. Projective perfect pushforward sends perfect complexes to perfect complexes, and the proper perfect base-change theorem gives the required compatibility. These facts apply to a singular Gorenstein fiber: smoothness of the diagonal is not an additional hypothesis for this geometric duality.

\begin{proposition}\label{prop:orientation-pairing}
Let $\pi:Z\to T$ satisfy the hypotheses above and choose $\vartheta:\O_Z\simeq\omega_{Z/T}$. For $E,F\in\Perf(Z)$, composition and trace induce a pairing
\[
 \R\pi_*\cRHom(E,F)\otL\R\pi_*\cRHom(F,E)
   \longrightarrow\O_T[-d].
\]
Its adjoint is an equivalence
\[
 \R\pi_*\cRHom(E,F)\simeq
 \bigl(\R\pi_*\cRHom(F,E)\bigr)^\vee[-d].
\]
\end{proposition}
\begin{proof}
Set $G=E^\vee\otimes F$. Relative duality gives
\[
 (\R\pi_*G)^\vee\simeq
 \R\pi_*(G^\vee\otimes\omega_{Z/T})[d].
\]
Use $G^\vee=F^\vee\otimes E$ and the chosen $\vartheta$. Dualizing again and using perfectness yields the stated equivalence. Evaluation identifies it with the composition-trace pairing. Since each operation is formed from a perfect complex and the relative trace, the construction commutes with derived base change.
\end{proof}
For the central union the required orientation is supplied by logarithmic
volume forms with opposite residues on $S$:
\[
 \omega_{X_0}|_{X_i}=\omega_{X_i}(S).
\]
Proposition~\ref{prop:singular-duality} constructs their common dualizing
section and proves the resulting threefold duality. Before trivializing
$\lambda$ in Proposition~\ref{prop:base-line}, the corresponding family
pairing has target $\lambda^{-1}[-3]$. Thus the residue construction and
the base line bundle specify the trace that is integrated below.

\subsubsection{Integration of closed forms}
The derived stack $\RPerf$ has a universal $2$-shifted symplectic form. Evaluating it on $Z\times_T\Map_T(Z,\RPerf)$ and integrating against $[Z/T]$ gives a closed $(2-d)$-shifted two-form on the mapping stack. This is the mapping-stack construction of \cite{PTVV13}, used relatively. Closedness is part of this construction, not a consequence of pointwise Serre duality alone. Nondegeneracy is checked by Proposition~\ref{prop:orientation-pairing} applied to the universal endomorphism complex. Section~\ref{sec:CY} spells out that verification for the relative moduli stack.

\subsection{Dg schemes of rigidified simple perfect complexes}\label{ssec: Dg schemes of rigidified simple perfect complexes}
\subsubsection{Relative objects and support conditions}
A sheaf application fixes a numerical class and, when appropriate, a
linear system of divisors on the projective fibers. Its objects on a
smooth fiber are coherent sheaves $G$, for instance pure two-dimensional
sheaves supported on hyperplane sections. The pushforward $i_{b*}G$ is
annihilated by the ideal of $X_b$ in the smooth total space. For a test
scheme $T\to C$, a family of such sheaves is required to be flat over
$T$. A perfect-complex family instead uses a perfect complex on the
fiber product; its terms and its cohomology sheaves are not confused
with a flat family of single sheaves. The latter description also
contains the derived restrictions to $S$ calculated in
Lemma~\ref{lem:Cartier-normal}.

For a projective flat morphism $\pi$, define
\begin{equation}\label{eq:relative-moduli}\IHESaliaslabel{MvertIntro}{eq:relative-moduli}
 \RPerf(\cX/C)(T)=\Perf(\cX_T)^{\simeq},
 \qquad\cX_T=\cX\times_C^h T.
\end{equation}
The superscript denotes the maximal $\infty$-groupoid. This is a relative moduli functor, retaining the map $T\to C$. Its fiber at $b$ is the perfect-complex functor on $X_b$ by definition. For representability and shifted structures we work on derived open Artin loci locally of finite presentation; their existence is part of the usual mapping-stack hypotheses \cite{TV,PTVV13}. Boundedness and the chosen stability condition are stated separately in any enumerative application.

The functor \eqref{eq:relative-moduli} differs from the full subcategory of $\Perf(\cX)$ whose cohomology is supported on finitely many fibers. The latter has morphisms and extensions computed in the total space. Even for a smooth fiber, these morphisms include normal deformation directions.

\begin{definition}\label{def:supported}
For a finite set $Z\subset C$ of closed points, let
$\Perf_{\pi^{-1}Z}(\cX)$ be the full subcategory of $\Perf(\cX)$
consisting of the perfect complexes $E$ for which
\[
 \supp\mathcal H^j(E)\subset\pi^{-1}(Z)
 \qquad\text{for every }j.
\]
Thus the condition concerns the supports of all cohomology sheaves of $E$.
These sheaves may be annihilated only by a power of the ideal of a fiber;
annihilation by the ideal itself is not part of this definition. Set
\[
 \Perf^{\mathrm{vert}}(\cX)=\bigcup_{Z\subset C\text{ finite}}
 \Perf_{\pi^{-1}Z}(\cX).
\]
These are small idempotent-complete stable categories. Their Ind-completions are different, presentable categories.
\end{definition}

\begin{lemma}\label{lem:simple-support}\IHESaliaslabel{stability}{lem:simple-support}
Let $\pi:\cX\to C$ be the family under consideration, and let $F$ be a
nonzero coherent sheaf on $\cX$ with proper support. Assume that $F$ is
simple, that is,
\[
 \Hom_{\cX}(F,F)=\CC\,\id_F,
\]
and that there are distinct closed points $b_1,\ldots,b_r\in C$ such that
\[
 \supp(F)\subset X_{b_1}\cup\cdots\cup X_{b_r}.
\]
Then there is a unique $b\in\{b_1,\ldots,b_r\}$ such that
$\supp(F)\subset X_b$. Moreover, if $\mathcal I_b$ is the ideal sheaf
of $X_b$ in $\cX$, then $\mathcal I_bF=0$. Equivalently,
$F=i_{b*}G$ for a unique coherent sheaf $G$ on $X_b$, where
$i_b:X_b\hookrightarrow\cX$ is the inclusion.
\end{lemma}
\begin{proof}
There are two assertions: the support meets only one of the distinct
fibers, and the ideal of that fiber, rather than only one of its powers,
annihilates $F$.

For the first assertion, let $\mathcal I_j$ be the ideal of $X_{b_j}$.
Since the fibers over distinct points are disjoint,
\[
 \mathcal I_j+\mathcal I_k=\O_{\cX}\quad(j\ne k).
\]
Put $\mathcal J=\prod_{j=1}^r\mathcal I_j$. The support assumption
and coherence imply that $\mathcal J^N F=0$ for some $N\geq1$.
Indeed, on a noetherian affine chart the ideal $\mathcal J$ is contained
in the radical of the annihilator of the finite module representing $F$.
A power is therefore contained in that annihilator. A finite affine
cover of a neighborhood of the proper support gives one common $N$;
see also \cite[Tag 01Y7]{Stacks}.
The ideals $\mathcal I_j^N$ are again pairwise comaximal, so the
Chinese remainder theorem gives
\[
 \O_{\cX}/\mathcal J^N
  \simeq\prod_{j=1}^r\O_{\cX}/\mathcal I_j^N.
\]
This is an isomorphism of sheaves of rings, obtained from the affine
Chinese remainder theorem \cite[Tag 00DT]{Stacks}.
Let $e_j$ be its orthogonal idempotents. Their action on $F$ gives
coherent direct summands
\[
 F_j=e_jF,\qquad
 F=\bigoplus_{j=1}^rF_j,\qquad
 \supp(F_j)\subset X_{b_j}.
\]
The projection onto $F_j$ is an idempotent endomorphism of $F$.
Simplicity says that this endomorphism is a scalar multiple of
$\id_F$; its scalar is either zero or one. Because $F\ne0$ and
$\sum_j e_j=1$, precisely one of these endomorphisms is the identity.
All other summands vanish. This proves the first assertion and
uniqueness of $b$.

For the second assertion, choose an open neighborhood $V\subset C$
of $b$ and a regular function $z\in\Gamma(V,\O_C)$ generating the
ideal of $b$ on $V$. We may shrink $V$ so that $z$ has no other zero
there. The pullback $\pi^*z$ generates $\mathcal I_b$ on
$\pi^{-1}(V)$. Multiplication by $\pi^*z$ is an endomorphism of
$F|_{\pi^{-1}(V)}$. It defines an endomorphism of $F$ on all of
$\cX$: glue it with the zero map on $\cX\setminus\supp(F)$; the
restrictions agree because $F$ vanishes on the overlap.
Again, a power of $\mathcal I_b$ annihilates $F$, so for some $m\geq1$
\[
 (\pi^*z)^m F=0.
\]
Write multiplication by $\pi^*z$ as $c\,\id_F$ using simplicity.
Then $c^m\id_F=0$. Since $F\ne0$ and $c\in\CC$, we get $c=0$.
Thus $\mathcal I_bF=0$ near the support, and hence everywhere.
The equivalence between coherent $\O_{X_b}$-modules and coherent
$\O_{\cX}$-modules annihilated by $\mathcal I_b$ gives the last statement.
\end{proof}
The simplicity hypothesis is essential. For nonzero sheaves $G_1,G_2$
on two distinct fibers,
$F=i_{b_1*}G_1\oplus i_{b_2*}G_2$ is supported on both fibers, but
it has two nonzero projection endomorphisms and is not simple.
The lemma concerns support in the fibers of $\pi$. It does not assert
that the scheme-theoretic support inside $X_b$ is reduced, or that
$G$ is perfect on a singular fiber.

\begin{example}\label{ex:extensions}
For a regular ring $A$ and a non-zero-divisor $t$, the sequence
\[
 0\longrightarrow A/(t)\xrightarrow{\ t\ }A/(t^2)
 \longrightarrow A/(t)\longrightarrow0
\]
consists of perfect $A$-modules. The two end terms are annihilated by $t$, but the middle term is not. Thus a category defined by annihilation by the reduced fiber ideal is not extension-closed. Definition~\ref{def:supported} avoids this problem by allowing infinitesimal thickenings.
\end{example}

\begin{example}\label{ex:node-perfect}
Near the double locus of the semistable family, use the local total-space ring $A=\CC[x,y,z,w]_{(x,y,z,w)}$ and the base equation $t=xy$. The ring of its central fiber is $R=A/(xy)$. The module $F=A/(x)$ has a length-one free resolution over $A$ and is annihilated by $t$. As an $R$-module, however, $F=R/(x)$ has the periodic resolution
\[
 \cdots\xrightarrow{x}R\xrightarrow{y}R\xrightarrow{x}R
 \longrightarrow R/(x)\longrightarrow0.
\]
Tensoring with the residue field leaves a nonzero term in every degree and zero differentials. Hence $F$ has infinite Tor dimension over $R$. Thus $F$ is perfect as an $A$-module and is annihilated by the fiber equation $xy$, but it is not perfect as an $R=A/(xy)$-module. This is the distinction between being perfect on the smooth total space and being perfect on its singular fiber.
\end{example}

\subsubsection{Trace and determinant}
For a perfect complex $E$ on a connected scheme, the scalar map and the trace map have opposite directions:
\[
 \O\xrightarrow{\mathrm{unit}}\cRHom(E,E)
 \xrightarrow{\tr}\O.
\]
Their composite is multiplication by the virtual rank. Simplicity concerns the scalar map on $\ext^0$, not the trace map.

\begin{proposition}\label{prop:rankzero}
If $E$ is simple, $\ext^{<0}(E,E)=0$, and $\rk(E)=0$, then
\[
 \CC\xrightarrow{\sim}\End(E),\qquad
 \tr:\End(E)\longrightarrow H^0(\O)
\]
have respectively the scalar isomorphism and the zero map. A scalar $a\in\Gm$ acts on $\det(E)$ by $a^{\rk(E)}=1$.
\end{proposition}
\begin{proof}
Represent $E$ locally by a bounded complex of vector bundles. The trace of its identity is the alternating sum of their ranks. The same alternating sum is the exponent of the scalar action on the determinant line bundle. These identities are independent of the resolution and give the assertions.
\end{proof}
The Lagrangian constructions use the full derived stacks. The hyperplane-section application subsequently uses scalar rigidification on its simple classical locus. We define the separate determinant constraint below so that the universal family and the tangent calculation specify the same moduli problem.

For applications of ordinary critical-chart and categorical descent theorems, we specify an oriented d-critical space, or an oriented quasi-smooth derived Deligne--Mumford enhancement, as part of the input. Local reduced deformation complexes are discussed in Section~\ref{sec:local}. On a Calabi--Yau threefold, scalar $\ext^0$ is dual to $\ext^3$, so removing the scalar direction alone does not preserve the shifted symplectic pairing.

\subsubsection{Complexes with a specified determinant line bundle}
Fix a line bundle $\Lambda$ on $X$. The moduli problem consists of
pairs $(E,\alpha)$ with $\alpha:\det E\xrightarrow{\sim}\Lambda$,
where
\[
 \det E=\bigotimes_j(\det E^j)^{(-1)^j}
\]
for a locally free representative. Its tangent complex is the
homotopy fiber of
\[
 \RH_X(E,E)[1]\xrightarrow{\tr[1]}R\Gamma(X,\O_X)[1].
\]
This is the derived determinant construction of~\cite{STV}.
For a rank-zero simple object the degree-zero trace is zero and
scalar automorphisms preserve $\alpha$. It is therefore necessary
to specify the classical rigidification and critical presentation
separately. For the section complexes,
$\det[L^{-m}\to\O_X]=L^m$.

\subsection{The derived quasi-Fano Lagrangian moduli spaces}\label{lagrangian-moduli}
The preceding constructions supply the varieties, moduli functors, and
trace orientations. We now prove that $r_i:\mathfrak X_i\to\mathfrak S$
is Lagrangian. The proof has two steps: the residue sequence gives a
nullhomotopy of the pulled-back \emph{closed} two-form, and Serre duality
identifies the resulting map
$\mathbb T_{r_i}\to\mathbb L_{\mathfrak X_i}[-1]$.
The final diagrams then produce the form on the derived intersection.

For the two anticanonical pairs, put $\M(X_i)=\RPerf(X_i)$ and $\M(S)=\RPerf(S)$ in this subsection. This notation denotes the full derived stacks, not their stable coarse moduli spaces. The restriction morphisms are
\begin{equation}\label{eqn: Restriction of moduli}
 r_i:\M(X_i)\longrightarrow\M(S),
\end{equation}
with objectwise rule
\begin{equation}\label{derived-restrict}
 E_i\longmapsto\LLf\iota_i^*E_i
       =E_i\otimes^{\mathbf L}_{\O_{X_i}}\O_S.
\end{equation}
A chosen open simple or stable locus in either source still maps to the full target $\RPerf(S)$. In particular, the derived restriction is not required to be a stable sheaf or to be concentrated in degree zero. When a smaller target chart is used for a local calculation, the sources are replaced by their derived pullbacks to that chart.

We work here with the full derived moduli stacks of perfect complexes. A universal perfect complex exists on the universal stack; on a coarse moduli space it may exist only on a gerbe. Endomorphism complexes and pairings can often descend across that gerbe, but the universal sheaf itself should not be asserted on the coarse space without a separate argument.

\subsubsection{Shifted symplectic geometry}
Let $\mathfrak{S}$ be a derived stack, and assuming it exists, denote by $\mathbb{L}_\mathfrak{S}$ its cotangent complex and by $\mathsf{DR}(\mathfrak{S})$ the de Rham complex, understood as a graded mixed complex over $\mathbb{C}.$ In \cite{PTVV13} the authors define a space $\mathcal{A}^{p}(\mathfrak{S},n)$ of $n$-shifted $p$-forms on $\mathfrak{S}$ and space of $n$-shifted closed $p$-forms $\mathcal{A}^{p,cl}(\mathfrak{S},n)$. They are defined in the affine case as follows. Letting $A$ be a commutative non-positively graded dg-algebra, one may define the simplicial sets
\begin{eqnarray*}
    \mathcal{A}^{p}(A, n) &=& |(\Lambda^{p}\mathbb{L}_{A}[n], d)|,
    \\
\mathcal{A}^{p,cl}(A, n) &=& |(\prod_{i\geq 0} \Lambda^{p+i}\mathbb{L}_{A}[-i + n], d + d_{dR})|.
\end{eqnarray*}
Here $d$ is the differential induced by the differential in $\mathbb{L}_{A}$, $d_{dR}$ is the de Rham differential and $|-|$ denotes the associated mapping-space realization (using the nonpositive truncation and the standard totalization signs). The $n$-shifted $p$-forms and closed $p$-forms define $\infty$-functors on the $\infty$-category of commutative dg-algebras which satisfy derived \'etale descent. They may then be extended to two $\infty$-functors from the opposite category of the $\infty$-category of derived stacks $\mathsf{dSt}$, to the $\infty$-category of simplicial sets  $$\mathcal{A}^{p}(-, n), \mathcal{A}^{p,cl}(-, n) : \mathsf{dSt}^{op} \to \mathbb{S}.$$

To be explicit, a $k$-shifted $p$-form on $\mathfrak{S}$ will be denoted by $\omega_{\mathfrak{S}}\in \big(\Lambda^p\mathbb{L}_{\mathfrak{S}}\big)^k$ such that $d\omega_\mathfrak{S}=0$ in $(\Lambda^p\mathbb{L}_{\mathfrak{S}})^{k+1}$. In particular, it defines a cohomology class $[\omega_{\mathfrak{S}}]\in H^k(\Lambda^p\mathbb{L}_\mathfrak{S}).$

A $k$-shifted closed $p$-form is a series $\omega_{\mathfrak{S}}^{\bullet}=(\omega_{\mathfrak{S}}^0,\omega_{\mathfrak{S}}^1,\ldots)$ with $\omega_\mathfrak{S}^m\in (\Lambda^{p+m}\mathbb{L}_{\mathfrak{S}})^{k-m}$ for $m\geq 0,$ satisfying 
$d\omega_\mathfrak{S}^0=0$ and $d\omega_\mathfrak{S}^{1+m}+d_{dR}\omega_\mathfrak{S}^m=0$ in $(\Lambda^{p+m+1}\mathbb{L}_{\mathfrak{S}})^{k-m}$ for all $m\geq 0.$

\begin{term}
\label{term: Underlying form}
\normalfont 
For $\mathfrak{S}=Spec(A)$, the natural morphism given by projection to the first component $\prod_{i\geq0} \Lambda^{p+i}\mathbb{L}_{A}[-i+n] \to \Lambda^{p}\mathbb{L}_{A}[n]$ induces a morphism $\mathcal{A}^{p,cl}(\mathfrak{S},n)\to \mathcal{A}^{p}(\mathfrak{S},n)$ of simplicial sets and given
a closed $p$-form $\omega_\mathfrak{S}^{\bullet}$ we call its image under this map the \emph{underlying $p$-form} and denote it by $\underline{\omega}_{\mathfrak{S}}:=\omega_\mathfrak{S}^0.$ 
\end{term}
If $\mathfrak{S}$ is a derived Artin stack the space of ($n$-shifted) $p$-forms is described as the mapping space
$$\mathcal{A}^{p}(\mathfrak{S},n) = Map_{\mathbf{L}_{qcoh}(\mathfrak{S})}(\mathcal{O}_{\mathfrak{S}},\Lambda^{p}\mathbb{L}_{\mathfrak{S}}[n]),$$
where $\mathbf{L}_{qcoh}(\mathfrak{S})$ is the $\infty$-category of quasi-coherent sheaves on $\mathfrak{S}$. 
\begin{definition}
\label{defn: n-shifted form}\label{def:shifted-forms}
Let $\mathfrak{S}$ be a derived Artin stack. An element $\omega\in \mathcal{A}^{2,cl}(\mathfrak{S},n)$ is called an \emph{$n$-shifted symplectic form} if its underlying 2-form $\omega^{0}$ is non-degenerate. That is, the map induced by $\omega^{0}$,$$\Theta_{\omega}: \mathbb{T}_{\mathfrak{S}}\to \mathbb{L}_{\mathfrak{S}}[n]$$is a quasi-isomorphism. We will denote by $Symp(\mathfrak{S},n)$ the space of all symplectic forms on $\mathfrak{S}$. A given $n$-symplectic derived stack is denoted by a pair $(\mathfrak{S},\omega)$.
\end{definition}
In the relative setting we use the relative cotangent and de Rham complexes, retaining the chosen base orientation and its twists, as specified in Section~\ref{sec:CY}. We use the definitions of \cite[Definitions 1.18, 2.7, 2.8]{PTVV13}.
\begin{definition}
Let $f:\mathfrak{X}\to \mathfrak{S}$ be a morphism of (derived) Artin stacks, where $\mathfrak{S}$ has an $n$-shifted symplectic 2-form $\omega\in Symp(\mathfrak{S},n)$. Let $f^*\omega\in \mathcal{A}^{2,cl}(\mathfrak{X},n)$ denote the pullback to a two form on $\mathfrak{X}$.
An \emph{isotropic structure} on $f$ (relative to $\omega$) is a path between $0$ and $f^*(\omega)$ in the space $\mathcal{A}^{2,cl}(\mathfrak{X},n)$. The space of isotropic structures on $f$ relative to $\omega$ is $$Isot(f, \omega)=Path_{0, f^*\omega}(\mathcal{A}^{2,cl}(\mathfrak{X},n)).$$
\end{definition}

\begin{definition}\label{def:lagrangian-original}
Let $r:\mathfrak L\to\mathfrak M$ and let $h$ be an isotropic structure for an $n$-shifted symplectic form $\omega$ on $\mathfrak M$. The structure $h$ is Lagrangian if the induced map
\[
 \Theta_h:\T_{\mathfrak L/\mathfrak M}\longrightarrow\ttL_{\mathfrak L}[n-1]
\]
is an equivalence. The map is obtained by factoring the composite determined by $h$ through the relative tangent complex.
\end{definition}
A closed form pulls back along any morphism, whereas its nondegeneracy need not pull back along a smooth atlas. The relative tangent complex of an atlas must be kept, as in Section~\ref{sec:local}.

\subsubsection{The boundary map and the isotropic structure}
Let $\iota:S\hookrightarrow V$ be a smooth anticanonical divisor in a
smooth projective variety of dimension $m$. Choose a nowhere-vanishing
section $\Omega_V$ of $\omega_V(S)$, and write
$\Omega_S=\operatorname{Res}_S\Omega_V$. Multiplication by these sections
identifies the two exact sequences in the following diagram:
\begin{equation}\label{eq:residue-divisor-diagram}
\begin{tikzcd}[column sep=small]
0\arrow[r]&\O_V(-S)\arrow[r]\arrow[d,"{\Omega_V}"',"{\sim}"]
 &\O_V\arrow[r]\arrow[d,"{\Omega_V}"',"{\sim}"]
 &\iota_*\O_S\arrow[r]\arrow[d,"{\Omega_S}"',"{\sim}"]&0\\
0\arrow[r]&\omega_V\arrow[r]
 &\omega_V(S)\arrow[r,"{\operatorname{Res}}"]
 &\iota_*\omega_S\arrow[r]&0.
\end{tikzcd}
\end{equation}
In particular, the residue is a nowhere-vanishing section of $\omega_S$.
The choice of $\Omega_V$ specifies both this section and the comparison
$\O_V(-S)\simeq\omega_V$; the two choices are used together throughout the proof.

For a proper variety $Z$, write $C_Z=\R\Gamma(Z,\O_Z)$. Derived global
sections of the top row of \eqref{eq:residue-divisor-diagram} give
\begin{equation}\label{eq:boundary-triangle-explicit}
 A_V\longrightarrow C_V\xrightarrow{\mathrm{res}}C_S
       \xrightarrow{\partial}A_V[1],
 \qquad A_V=\R\Gamma(V,\omega_V).
\end{equation}
Let $\operatorname{Tr}_V:A_V\to\CC[-m]$ and
$\operatorname{Tr}_S:\R\Gamma(S,\omega_S)\to\CC[1-m]$ be the Serre traces.
The orientation $[S]=\operatorname{Tr}_S\circ\Omega_S$ satisfies
\begin{equation}\label{eq:boundary-trace-square}
\begin{tikzcd}[column sep=large]
C_S\arrow[r,"{\partial}"]\arrow[d,"{[S]}"']
 & A_V[1]\arrow[d,"{\operatorname{Tr}_V[1]}"]\\
\CC[1-m]\arrow[r,equal]&\CC[1-m].
\end{tikzcd}
\end{equation}
Here the connecting map for the residue sequence is dual, under Serre
duality, to restriction of functions $H^0(V,\O_V)\to H^0(S,\O_S)$.
Evaluating the two sides on $1\in H^0(V,\O_V)$ proves the trace identity.
Equivalently, it is the compatibility of the counit of divisor duality
with the counit of proper duality. This explanation also fixes the
sign of $\partial$ in \eqref{eq:boundary-trace-square}.

The consecutive maps in \eqref{eq:boundary-triangle-explicit} have a
specified nullhomotopy. Composing that nullhomotopy with
$\operatorname{Tr}_V[1]$ gives
\begin{equation}\label{eq:boundary-nullhomotopy}
 [S]\circ\mathrm{res}\simeq0:
                  C_V\longrightarrow\CC[1-m].
\end{equation}
We spell out its chain-level meaning. Choose complexes and a chain map
$u:A_V\to C_V$ representing the first arrow of
\eqref{eq:boundary-triangle-explicit}. Its cone has terms and differential
\[
 \Cone(u)^n=C_V^n\oplus A_V^{n+1},\qquad
 d(b,a)=(d b+u(a),-d a).
\]
Represent $C_S$ by this cone. Then restriction is $b\mapsto(b,0)$,
and the boundary is $(b,a)\mapsto a$. Their composite is zero as a
chain map. Transporting this diagram through the chosen quasi-isomorphism
to $C_S$ gives \eqref{eq:boundary-nullhomotopy}, together with its
compatibility with \eqref{eq:boundary-trace-square}. In a universal
family the same construction is tensored with the structure sheaf of
the base. Thus it is a homotopy of complexes in families, not just the
vanishing of a map on an individual cohomology group.

We next express this homotopy on mapping stacks.
It is useful because an isotropic structure is a path of \emph{closed}
two-forms, including their higher de Rham components. Set
\[
 \mathfrak M_V=\Map(V,\RPerf),\qquad
 \mathfrak M_S=\Map(S,\RPerf),\qquad r:\mathfrak M_V\to\mathfrak M_S.
\]
All constructions are made on locally geometric finite-presentation
loci on which the displayed perfect cotangent complexes exist.
The evaluation maps fit into the diagrams
\begin{equation}\label{eq:restriction-evaluation-diagram}
\begin{tikzcd}[column sep=large]
S\times\mathfrak M_V\arrow[r,"{1\times r}"]\arrow[d,"{\iota\times1}"']
 &S\times\mathfrak M_S\arrow[d,"{\mathrm{ev}_S}"]\\
V\times\mathfrak M_V\arrow[r,"{\mathrm{ev}_V}"']&\RPerf
\end{tikzcd}
\qquad
\begin{tikzcd}[column sep=small]
S\times\mathfrak M_V\arrow[r,"{1\times r}"]\arrow[d]
 &S\times\mathfrak M_S\arrow[d]\\
\mathfrak M_V\arrow[r,"{r}"']&\mathfrak M_S.
\end{tikzcd}
\end{equation}
The square on the right is Cartesian. The left square is commutative
up to the universal equivalence of restricted perfect complexes.

Here is the closed-form version of the evaluation diagram. Write
$\mathscr N(M)=NC^w(M)$ for the weighted negative-cyclic model of forms
used in \cite{PTVV13}. In the following display
\[
 \mathscr N_Z(M)=NC^w(Z\times M/Z)
                \simeq C_Z\otimes\mathscr N(M).
\]
The notation $Z\times M/Z$ means that only the differentials in the
$M$ direction are retained. Evaluation followed by this projection
induces the maps denoted by $\mathrm{ev}_{Z,\mathrm{rel}}^*$ below.
Since $C_Z$ is perfect, tensoring with it commutes with the products
in the negative-cyclic totalization. Naturality of evaluation and of
integration gives
\begin{equation}\label{eq:closed-integration-diagram}
\begin{tikzcd}[column sep=small,row sep=large]
NC^w(\RPerf)\arrow[r,"{\mathrm{ev}_{S,\mathrm{rel}}^*}"]
 \arrow[d,"{\mathrm{ev}_{V,\mathrm{rel}}^*}"']
 &\mathscr N_S(\mathfrak M_S)\arrow[r,"{\int_{[S]}}"]
 \arrow[d,"{(1\times r)^*}"]
 &\mathscr N(\mathfrak M_S)[1-m]\arrow[d,"{r^*}"]\\
\mathscr N_V(\mathfrak M_V)\arrow[r,"{\iota^*}"]
 &\mathscr N_S(\mathfrak M_V)\arrow[r,"{\int_{[S]}}"']
 &\mathscr N(\mathfrak M_V)[1-m].
\end{tikzcd}
\end{equation}
The composite of the bottom arrows is induced by
$[S]\circ\mathrm{res}$. Its nullhomotopy is therefore
\eqref{eq:boundary-nullhomotopy} tensored with $\mathscr N(\mathfrak M_V)$.
Applying \eqref{eq:closed-integration-diagram} to the universal closed
$2$-shifted symplectic form $\omega_{\RPerf}$ gives
\[
 \omega_{\mathfrak M_S}=\int_{[S]}\mathrm{ev}_{S,\mathrm{rel}}^*
                          \omega_{\RPerf},
 \qquad
 h_V:0\simeq r^*\omega_{\mathfrak M_S}
          \quad\text{in }\mathcal A^{2,\mathrm{cl}}(\mathfrak M_V,3-m).
\]
This is Calaque's boundary-integration construction, with the boundary
homotopy and its evaluation maps written explicitly
\cite[Theorem 2.9 and Section 3.2]{Calaque15}.

\subsubsection{Universal complexes and the tangent triangle}
We now return to the threefolds $X_i$. Put
$\mathfrak M_i=\RPerf(X_i)$ and $\mathfrak M_S=\RPerf(S)$, and denote the
universal perfect complexes by $\widetilde{\mathcal F}_i$ and
$\widetilde{\mathcal F}_S$. Let
\[
 \begin{split}
 p_i:X_i\times\mathfrak M_i\longrightarrow\mathfrak M_i,&\qquad
 q_i:X_i\times\mathfrak M_i\longrightarrow X_i,\\
 p_{S,i}:S\times\mathfrak M_i\longrightarrow\mathfrak M_i,&\qquad
 j_i=\iota_i\times1:S\times\mathfrak M_i\longrightarrow X_i\times\mathfrak M_i.
 \end{split}
\]
Write $p_S:S\times\mathfrak M_S\to\mathfrak M_S$. The universal
restriction equivalence and the Cartesian projection square are
\begin{equation}\label{eq:universal-restriction-family}
 \LLf j_i^*\widetilde{\mathcal F}_i
  \simeq(1\times r_i)^*\widetilde{\mathcal F}_S,
 \qquad
 \begin{tikzcd}[column sep=large]
 S\times\mathfrak M_i\arrow[r,"{1\times r_i}"]\arrow[d,"{p_{S,i}}"']
 &S\times\mathfrak M_S\arrow[d,"{p_S}"]\\
 \mathfrak M_i\arrow[r,"{r_i}"']&\mathfrak M_S.
 \end{tikzcd}
\end{equation}
In every displayed tensor product of complexes we use the derived
product. This convention is essential when the sheaf has nonzero
$\Tor_1$ along $S$.

For readability introduce three perfect complexes on $\mathfrak M_i$:
\begin{align}\label{eq:universal-ABC}
 A_i&=\R p_{i*}\cRHom(\widetilde{\mathcal F}_i,
                                  \widetilde{\mathcal F}_i),\notag\\
 A_i^-&=\R p_{i*}\cRHom(\widetilde{\mathcal F}_i,
               \widetilde{\mathcal F}_i\otimes q_i^*\O_{X_i}(-S)),\\
 B_i&=\R p_{S,i*}\cRHom(\LLf j_i^*\widetilde{\mathcal F}_i,
                                \LLf j_i^*\widetilde{\mathcal F}_i).\notag
\end{align}
The projection formula, perfectness of the universal complex, and
\eqref{eq:universal-restriction-family} identify
\[
 B_i\simeq r_i^*\R p_{S*}\cRHom(\widetilde{\mathcal F}_S,
                                      \widetilde{\mathcal F}_S).
\]
Tensoring the divisor sequence by
$\cRHom(\widetilde{\mathcal F}_i,\widetilde{\mathcal F}_i)$ and
applying $\R p_{i*}$ gives the distinguished triangle
\begin{equation}\label{eq:universal-tangent-triangle}
 A_i^-\xrightarrow{u_i}A_i\xrightarrow{\rho_i}B_i
                              \xrightarrow{\partial_i}A_i^-[1].
\end{equation}
Consequently the tangent triangle of the restriction morphism is
identified term by term as follows:
\begin{equation}\label{eq:original-tangent-diagram-restored}
\begin{tikzcd}[column sep=large]
\mathbb T_{r_i}\arrow[r]\arrow[d,"{\sim}"']
 &\mathbb T_{\mathfrak M_i}\arrow[r]\arrow[d,"{\sim}"']
 &r_i^*\mathbb T_{\mathfrak M_S}\arrow[d,"{\sim}"']\\
A_i^-[1]\arrow[r,"{u_i[1]}"']&A_i[1]\arrow[r,"{\rho_i[1]}"']&B_i[1].
\end{tikzcd}
\end{equation}
For example, at a point represented by $F_i$, the left term is
$\RH_{X_i}(F_i,F_i(-S))[1]$. The middle and right terms retain all
self-extensions of $F_i$ and of $\LLf\iota_i^*F_i$, respectively.

\subsubsection{Compatibility of the trace maps and nondegeneracy}
The endomorphism complex of a bounded complex of vector bundles
$P^\bullet$ has differential
\[
 d(a)=d_Pa-(-1)^{|a|}a d_P.
\]
Its trace is the supertrace: for a degree-zero endomorphism,
\[
 \operatorname{tr}(a)=\sum_k(-1)^k\operatorname{tr}(a|_{P^k}).
\]
The trace map is zero in the other degrees of $\O$. The identity
$\operatorname{tr}(ab)=(-1)^{|a||b|}\operatorname{tr}(ba)$ shows that
composition followed by trace is a morphism of complexes and is
compatible with restriction. We use the standard suspension signs
when shifting both factors by $[1]$.

There is also a description of the underlying isotropic
homotopy directly on universal endomorphism complexes. The trace
compatibility is expressed by the following square
\begin{equation}\label{eqn: Diag}
\begin{tikzcd}[column sep=large,row sep=large]
 A_i[1]\otimes A_i[1]\arrow[r,"{\mathrm{tr}\circ\mathrm{comp}}"]
 \arrow[d,"{\rho_i[1]\otimes\rho_i[1]}"']
 &\R p_{i*}\O_{X_i\times\mathfrak M_i}[2]\arrow[d,"{\mathrm{res}}"]\\
 B_i[1]\otimes B_i[1]\arrow[r,"{\mathrm{tr}\circ\mathrm{comp}}"']
 &\R p_{S,i*}\O_{S\times\mathfrak M_i}[2]
                          \arrow[d,"{[S][2]}"]\\
 &\O_{\mathfrak M_i}.
\end{tikzcd}
\end{equation}
This is a diagram of \emph{derived pushforwards}, rather than just
of their second cohomology sheaves. Its composite down the right
from the upper corner has the nullhomotopy
\eqref{eq:boundary-nullhomotopy}, after base change to $\mathfrak M_i$.
Thus \eqref{eqn: Diag} describes the underlying two-form of the closed
isotropic structure already constructed in
\eqref{eq:closed-integration-diagram}.

For nondegeneracy, the boundary factor of the pairing is retained
instead of discarded. Composition and the threefold trace give
\begin{equation}\label{eq:relative-Serre-pairing-explicit}
 A_i^-[1]\otimes A_i[1]
 \longrightarrow \R p_{i*}q_i^*\omega_{X_i}[2]
 \xrightarrow{\operatorname{Tr}_{X_i}[2]}\O_{\mathfrak M_i}[-1].
\end{equation}
Here the first arrow uses $\O_{X_i}(-S)\simeq\omega_{X_i}$ fixed by
$\Omega_{X_i}$. Relative Serre duality gives
\begin{align}\label{eq:relative-Serre-identifications}
 A_i^\vee
 &\simeq\R p_{i*}\cRHom(\widetilde{\mathcal F}_i,
             \widetilde{\mathcal F}_i\otimes q_i^*\omega_{X_i})[3]
   =A_i^-[3],\notag\\
 A_i^-[1]&\simeq A_i^\vee[-2]
       =\mathbb L_{\mathfrak M_i}[-1].
\end{align}
The cyclic trace identifies the dual of the endomorphism sheaf with
itself; this is why the same universal complex occurs in the first
and second slots after duality.

\begin{lemma}\label{lem:boundary-pairing}\IHESaliaslabel{isotropic}{lem:boundary-pairing}
The map $\Theta_{h_i}:\mathbb T_{r_i}\to\mathbb L_{\mathfrak M_i}[-1]$
induced by the closed isotropic structure $h_i$ is the adjoint of
\eqref{eq:relative-Serre-pairing-explicit}. In particular it is the
quasi-isomorphism \eqref{eq:relative-Serre-identifications}.
\end{lemma}
\begin{proof}
The construction of $\Theta_{h_i}$ starts from the nullhomotopy of
\[
 A_i[1]\longrightarrow B_i[1]\longrightarrow B_i^\vee[-1]
                    \longrightarrow A_i^\vee[-1].
\]
Factoring through the fiber of $A_i[1]\to B_i[1]$ gives a map
$A_i^-[1]\to A_i^\vee[-2]$. To identify this map, tensor the two rows
of \eqref{eq:residue-divisor-diagram} with the universal endomorphism
complex and use composition followed by supertrace. Compatibility
of the boundary with multiplication gives a square of pairings;
applying the trace square \eqref{eq:boundary-trace-square} gives
exactly \eqref{eq:relative-Serre-pairing-explicit}. At the level of
the cone model, the $A_i^-$ component of a relative cocycle is paired
with an endomorphism in $A_i$, multiplied into
$\cRHom(\widetilde{\mathcal F}_i,
\widetilde{\mathcal F}_i\otimes q_i^*\omega_{X_i})$, and then traced.
This is precisely the boundary factor used by the nullhomotopy.

One can express the compatibility independently of a cone model.
Dualizing the restriction arrow and comparing it with the boundary
arrow gives
\begin{equation}\label{eq:boundary-adjoint-diagram}
\begin{tikzcd}[column sep=large]
 B_i[1]\arrow[r,"{\partial_i[1]}"]\arrow[d,"{\omega_S^\flat}"',"{\sim}"]
 &A_i^-[2]\arrow[d,"{\Theta_{h_i}[1]}","{\sim}"']\\
 B_i^\vee[-1]\arrow[r,"{\rho_i^\vee[-1]}"']&A_i^\vee[-1].
\end{tikzcd}
\end{equation}
The arrow on the right is the Serre-duality map: the residue exact
sequence is dual to the divisor exact sequence, as in
\eqref{eq:residue-divisor-diagram}. Thus this diagram determines the
same pairing as the construction of $h_i$. Relative Serre duality
then proves it is perfect. All maps have been constructed on the
universal family and commute with base change, so the conclusion
holds as an equivalence of perfect complexes on $\mathfrak M_i$.
\end{proof}

\begin{theorem}\label{thm:restriction}\IHESaliaslabel{lem: Lagrangian P to X}{thm:restriction}\IHESaliaslabel{Lagrangian}{thm:restriction}
Let $V$ be a smooth projective variety of dimension $d+1$, and let
$D\subset V$ be a smooth divisor with a chosen isomorphism
$\omega_V(D)\simeq\O_V$. The derived restriction morphism
\[
 r:\RPerf(V)\longrightarrow\RPerf(D),\qquad E\longmapsto\LLf i^*E,
\]
is Lagrangian for the $(2-d)$-shifted symplectic structure determined
by the induced volume form on $D$. We work on compatible locally
geometric finite-presentation loci of these stacks. Its relative
tangent is
\begin{equation}\label{eq:restriction-tangent}\IHESaliaslabel{relative-quas}{eq:restriction-tangent}
 \mathbb T_{r,E}\simeq\RH_V(E,E(-D))[1]
             \simeq\mathbb L_{\RPerf(V),E}[1-d].
\end{equation}
In particular, for each anticanonical threefold pair $(X_i,S)$,
\[
 \Theta_{h_i}:\mathbb T_{r_i}
                 \xrightarrow{\sim}\mathbb L_{\mathfrak M_i}[-1].
\]
The restriction to any open derived substack of $\RPerf(X_i)$ has
the same Lagrangian structure with target $\RPerf(S)$.
\end{theorem}
\begin{proof}
The constructions above give all the maps in the threefold case.
We assemble them in a diagram of tangent and cotangent triangles,
with all relative shifts displayed:
\begin{equation}\label{square}
\begin{tikzcd}[column sep=large,row sep=large]
\mathbb T_{r_i}\arrow[r]\arrow[d,"{\Theta_{h_i}}"',"{\sim}"]
 &\mathbb T_{\mathfrak M_i}\arrow[r]\arrow[d,"{\alpha_i}","{\sim}"']
 &r_i^*\mathbb T_{\mathfrak M_S}\arrow[d,"{\omega_S^\flat}","{\sim}"']\\
\mathbb L_{\mathfrak M_i}[-1]\arrow[r]
 &\mathbb L_{r_i}[-1]\arrow[r]
 &r_i^*\mathbb L_{\mathfrak M_S}.
\end{tikzcd}
\end{equation}
The upper row is \eqref{eq:original-tangent-diagram-restored}.
The lower row is the cotangent triangle
$r_i^*\mathbb L_{\mathfrak M_S}\to\mathbb L_{\mathfrak M_i}
\to\mathbb L_{r_i}$ rotated once. Its arrows carry the usual
rotation signs. In terms of \eqref{eq:universal-ABC}, the rows are
\begin{equation}\label{eq:expanded-dual-triangle}
\begin{tikzcd}[column sep=large,row sep=large]
 A_i^-[1]\arrow[r]\arrow[d,"{\sim}"']
 &A_i[1]\arrow[r]\arrow[d,"{\sim}"']
 &B_i[1]\arrow[d,"{\sim}"']\\
 A_i^\vee[-2]\arrow[r]
 &(A_i^-)^\vee[-2]\arrow[r]
 &B_i^\vee[-1].
\end{tikzcd}
\end{equation}
The middle equivalence is dual to the first, and the third is Serre
duality on the surface. The commutativity of the diagram is
\eqref{eq:boundary-adjoint-diagram} and its dual. The first
vertical equivalence is therefore the map belonging to the
\emph{specified} isotropic homotopy, not just some isomorphism
between its source and target. This proves the Lagrangian condition.

For $\dim V=d+1$, the boundary trace has target $\CC[-d]$ and
\eqref{eq:relative-Serre-pairing-explicit} has target $\O[1-d]$.
The tangent fiber is still $\RH_V(E,E(-D))[1]$. Relative Serre
duality in dimension $d+1$ gives
\[
 \RH_V(E,E(-D))[1]
    \simeq\RH_V(E,E)^\vee[-d]
    =\mathbb L_{\RPerf(V),E}[1-d].
\]
The same trace and boundary diagrams give the required isotropic
path of degree $2-d$ and identify its induced map with this
quasi-isomorphism. This proves the general statement.
Finally an open immersion has zero relative cotangent complex;
restricting all the universal complexes and the isotropic path to
that open preserves the displayed equivalences.
\end{proof}

The assumptions involve the anticanonical divisor and its residue
volume form. In particular the threefold proof uses no vanishing
of $H^1(S,\O_S)$ and no vanishing of $H^2(X_i,\O_{X_i})$.
The target is the moduli stack of the entire restricted complex.
For a sheaf on $X_i$ with nonzero $\Tor_1$ along $S$, that negative
cohomology is included in $B_i$ and in every arrow above.

\subsubsection{The modified component in the quintic family}
For clarity we check the component to which the blowup was applied.
In Proposition~\ref{prop:quintic-model} it is
$X_2=\Bl_\Gamma\P^3$, with exceptional divisor $E$ and pulled-back
hyperplane class $H$. The strict transform of the boundary satisfies
\[
 S\sim4H-E,\qquad K_{X_2}\sim-4H+E=-S.
\]
Thus multiplication by the logarithmic volume form gives
$\O_{X_2}(-S)\simeq\omega_{X_2}$ exactly as it does on $X_1$.
The contribution $E$ in the canonical bundle formula is already
included in the strict transform of $S$. The relative complex and
its dual are therefore
\[
 \begin{split}
 \mathbb T_{r_2}
   &=\R p_{2*}\cRHom(\widetilde{\mathcal F}_2,
         \widetilde{\mathcal F}_2\otimes q_2^*\O_{X_2}(-S))[1],\\
 \mathbb T_{r_2}
   &\xrightarrow{\sim}A_2^\vee[-2]
      =\mathbb L_{\mathfrak M_2}[-1].
 \end{split}
\]
This is the first vertical arrow in \eqref{square} for $i=2$.
It applies equally to complexes whose support meets $E$.

The anticanonical restriction construction also underlies the
categorical correspondences of \cite{Amorim_2017}. On a smooth
sheaf-moduli locus over a symplectic surface, its target pairing is
the usual trace pairing of \cite{mukai1984symplectic}. The second
author is grateful to Oren Ben-Bassat for communicating this
restriction viewpoint and major parts of the explicit proof in 2015.

\subsubsection{The diagrams for the derived Lagrangian intersection}
Let
\[
 \mathfrak G=\mathfrak M_1\times^h_{\mathfrak M_S}\mathfrak M_2,
 \qquad p_i:\mathfrak G\to\mathfrak M_i,
 \qquad q:\mathfrak G\to\mathfrak M_S.
\]
Choose the isotropic structures with respect to one fixed target
form. When the logarithmic residues used for the normal-crossing
volume are opposite, change the sign of the second boundary
orientation in passing to this common-target convention. This
changes the signs in the trace pairings, not the restriction maps.
On $\mathfrak G$ put
\[
 V_i=p_i^*\mathbb T_{\mathfrak M_i},\qquad
 W=q^*\mathbb T_{\mathfrak M_S},\qquad
 K_i=p_i^*\mathbb T_{r_i}.
\]
Differentiating the matching equivalence gives
\begin{equation}\label{eq:gluing-tangent-fiber}
 \mathbb T_{\mathfrak G}\longrightarrow V_1\oplus V_2
       \xrightarrow{\rho_1-\rho_2}W\xrightarrow{+1}.
\end{equation}
Thus, at $(E_1,E_2,\alpha)$ with common restriction $E_S$, the
endomorphism complex of this triple is
\[
 \begin{split}
 A_{\mathrm{gl}}
  &=\fib\bigl(\RH_{X_1}(E_1,E_1)\oplus\RH_{X_2}(E_2,E_2)
          \xrightarrow{r_1-r_2}\RH_S(E_S,E_S)\bigr),\\
 \mathbb T_{\mathfrak G}&=A_{\mathrm{gl}}[1].
 \end{split}
\]
The map from the second summand is conjugated by $\alpha$ before
subtraction. This specifies the map even when the two restrictions
are represented by different complexes.

\begin{corollary}\label{cor:gluing-symplectic}
The stack $\mathfrak G$ has a $(-1)$-shifted symplectic structure
whose underlying map is
$\mathbb T_{\mathfrak G}\xrightarrow{\sim}\mathbb L_{\mathfrak G}[-1]$.
The same statement holds for the specified open derived substacks
of $\RPerf(X_1)$ and $\RPerf(X_2)$. Restriction from a smooth fourfold to a smooth
anticanonical threefold is Lagrangian in a $(-1)$-shifted
symplectic target.
\end{corollary}
\begin{proof}
We give the triangle calculation, as well as the closed-form
construction. Pulling the two paths $h_i$ to $\mathfrak G$ and
using the universal homotopy $r_1p_1\simeq r_2p_2$ produces a loop
at zero in $\mathcal A^{2,\mathrm{cl}}(\mathfrak G,0)$.
A loop in this space is a closed two-form of degree $-1$;
denote it by $\omega_{\mathfrak G}$. This is the loop used in the
PTVV intersection theorem \cite[Theorem 2.9]{PTVV13}.

The tangent triangle and its dual-intersection triangle form
the following nondegeneracy diagram:
\begin{equation}\label{Bdiagram}
\begin{tikzcd}[column sep=large,row sep=large]
 \mathbb T_{\mathfrak G}\arrow[r]
              \arrow[d,"{\omega_{\mathfrak G}^{\flat}}"']
 &V_1\oplus V_2\arrow[r,"{\rho_1-\rho_2}"]
              \arrow[d,"{\alpha_1\oplus\alpha_2}","{\sim}"']
 &W\arrow[d,"{\omega_S^\flat}","{\sim}"']\\
 \mathbb L_{\mathfrak G}[-1]\arrow[r]
 &K_1^\vee[-1]\oplus K_2^\vee[-1]
               \arrow[r,"{\delta_1-\delta_2}"']&W^\vee.
\end{tikzcd}
\end{equation}
The upper row is \eqref{eq:gluing-tangent-fiber}. For the lower row,
first dualize that triangle to identify
\[
 \mathbb L_{\mathfrak G}[-1]
 \simeq\Cone\bigl(W^\vee[-1]
              \xrightarrow{(\rho_1^\vee,-\rho_2^\vee)}
                  V_1^\vee[-1]\oplus V_2^\vee[-1]\bigr).
\]
The maps $\delta_i:K_i^\vee[-1]\to W^\vee$ are the boundary
maps in the duals of $K_i\to V_i\to W$.
The following two squares express the corresponding cone calculation:
\begin{equation}\label{distinguished}
\begin{tikzcd}[column sep=large,row sep=large]
 W^\vee[-1]\arrow[r,"{\rho_1^\vee}"]\arrow[d,"{\rho_2^\vee}"']
 &V_1^\vee[-1]\arrow[d]\\
 V_2^\vee[-1]\arrow[r]&\mathbb L_{\mathfrak G}[-1]
\end{tikzcd}
\qquad
\begin{tikzcd}[column sep=large,row sep=large]
 \mathbb L_{\mathfrak G}[-1]\arrow[r]\arrow[d]
 &K_1^\vee[-1]\arrow[d,"{\delta_1}"]\\
 K_2^\vee[-1]\arrow[r,"{\delta_2}"']&W^\vee.
\end{tikzcd}
\end{equation}
The first is a homotopy pushout. Taking the two cofibers and using
the standard rotation signs gives the homotopy pullback on the
right. Explicitly,
$K_i^\vee[-1]=\Cone(W^\vee[-1]\to V_i^\vee[-1])$.
In the fiber of $\delta_1-\delta_2$, the difference of the two
copies of $W^\vee$ cancels its cone partner; the remaining terms
are $V_1^\vee[-1]\oplus V_2^\vee[-1]\oplus W^\vee$, with the
differential of the preceding displayed cone. This proves that
the lower row of \eqref{Bdiagram} is an exact triangle.

The right square of \eqref{Bdiagram} is the direct sum of the
trace-duality diagrams \eqref{square}, with the subtraction
corresponding to $\rho_1-\rho_2$. The loop defining
$\omega_{\mathfrak G}$ supplies the left map, so the left square
commutes with its specified homotopy. The middle and right
vertical maps are equivalences. Taking fibers in this commutative
diagram proves that the left vertical map is an equivalence as
well. This proves nondegeneracy of the actual closed form.
The argument restricts to derived open loci by base change.
For a fourfold pair the final assertion follows from
Theorem~\ref{thm:restriction} with $d=3$.
\end{proof}

For a triple induced by a perfect complex $E$ on $X_0$, the normalization sequence identifies $A_{\mathrm{gl}}$ with
$\RH_{X_0}(E,E)$. Its universal-family diagram is
\begin{equation}\label{Adiagram}
\begin{tikzcd}[column sep=large,row sep=large]
 \mathbb T_{\mathfrak G}\arrow[r]\arrow[d,"{\sim}"']
 &p_1^*\mathbb T_{\mathfrak M_1}\oplus p_2^*\mathbb T_{\mathfrak M_2}
      \arrow[r]\arrow[d,"{\sim}"']
 &q^*\mathbb T_{\mathfrak M_S}\arrow[d,"{\sim}"']\\
 A_{\mathrm{gl}}[1]\arrow[r]
 &p_1^*A_1[1]\oplus p_2^*A_2[1]\arrow[r]
 &q^*A_S[1].
\end{tikzcd}
\end{equation}
Here $A_S=\R p_{S*}\cRHom(\widetilde{\mathcal F}_S,
\widetilde{\mathcal F}_S)$; on the effective locus the left lower
term is the shifted universal endomorphism complex on $X_0$.
The resulting long exact sequence has terms
\[
 \begin{aligned}
 \cdots\longrightarrow\ext^k_{X_0}(E,E)
 &\longrightarrow\bigoplus_i\ext^k_{X_i}(E_i,E_i)\\
 &\longrightarrow\ext^k_S(E_S,E_S)
 \longrightarrow\ext^{k+1}_{X_0}(E,E)\longrightarrow\cdots.
 \end{aligned}
\]
It describes automorphisms, deformations, and obstructions using
the objects $E_i$ on $X_i$, for $i=1,2$, and their entire derived restrictions to $S$.
Propositions~\ref{prop:fully-faithful} and~\ref{prop:effective}
give the comparison and effectivity statements used here.

\section{Calabi--Yau category structures and vertical support}
\label{sec:CY}\IHESaliaslabel{CY-structure}{sec:CY}
The restriction theorem supplies the shifted form on the derived
intersection. To relate it to objects on the total space we need
threefold duality on the singular union, compatibility of mapping
complexes, and fourfold duality near the support. These constructions
are given separately so that their different shifts are retained.

\subsection{Categorical conventions}
A small stable idempotent-complete $\CC$-linear dg category is proper
when all its morphism complexes are perfect over $\CC$. A Serre
functor $\mathcal S$ represents the natural duality
\[
 \RH(E,F)^\vee\simeq\RH(F,\mathcal S E).
\]
A chosen identification $\mathcal S\simeq[d]$ specifies the
nondegenerate pairing. A cyclic Calabi--Yau structure additionally
specifies the cyclic lift of the trace. For a smooth category, a
negative cyclic class of dimension $d$ induces
$\mathcal C^!\simeq\mathcal C_\Delta[-d]$ and a closed
$(2-d)$-shifted form on its moduli of objects by
\cite[Theorem 5.6]{BD2}. In the geometric application below the
volume form supplies that lift through mixed HKR. The proper
Gorenstein family is treated directly by its relative duality trace.

\subsection{Calabi--Yau structure on the special fiber}
For the central term of the comparison we need the same duality on
$X_0$ and a description of its deformations through $X_1,X_2,S$.
The residue generator supplies the orientation; the normalization
sequence supplies the endomorphism and descent calculations.

Let $X_0=X_1\cup_S X_2$ be as in Definition~\ref{def:good}. The normalization sequence is
\begin{equation}\label{eq:normalization}
 0\longrightarrow\O_{X_0}\longrightarrow
 j_{1*}\O_{X_1}\oplus j_{2*}\O_{X_2}
 \xrightarrow{\mathrm{res}_1-\mathrm{res}_2}j_{S*}\O_S
 \longrightarrow0.
\end{equation}
The dualizing bundle restricts to $\omega_{X_i}(S)$, not to $\omega_{X_i}$. Its sections are pairs of logarithmic volume forms whose residues sum to zero.

\begin{proposition}\label{prop:singular-duality}\IHESaliaslabel{lem: CY special fiber}{prop:singular-duality}
Let $X_0=X_1\cup_S X_2$ be a proper simple normal-crossing
threefold with two smooth components and smooth double divisor.
Choose nowhere-vanishing sections
$\eta_i\in H^0(X_i,\omega_{X_i}(S))$ with
$\operatorname{Res}_1\eta_1+\operatorname{Res}_2\eta_2=0$.
Then $\omega_{X_0}\simeq\O_{X_0}$, and, for perfect $E,F$,
\begin{equation}\label{eq:singular-serre}
 \RH_{X_0}(E,F)^\vee\simeq\RH_{X_0}(F,E)[3]
\end{equation}
functorially in both arguments. The functor $[3]$ is the Serre
functor of $\Perf(X_0)$.
\end{proposition}
\begin{proof}
Work etale locally at $S$. The completed local equation is
$xy=0$ in smooth coordinates $(x,y,z_1,z_2)$. Put
$R=\CC[[x,y,z_1,z_2]]/(xy)$.
The length-one hypersurface resolution identifies its dualizing
module with the free rank-one module generated by the residue
\[
 \eta=\operatorname{Res}_{xy=0}
        \frac{dx\wedge dy\wedge dz_1\wedge dz_2}{xy}.
\]
Its restrictions to $x=0$ and $y=0$ are, with compatible residue
signs,
\[
 \frac{dy}{y}\wedge dz_1\wedge dz_2,
 \qquad
 -\frac{dx}{x}\wedge dz_1\wedge dz_2.
\]
Consequently there is an exact residue sequence
\begin{equation}\label{eq:v9-dualizing-normalization}
 0\longrightarrow\omega_{X_0}\longrightarrow
 j_{1*}\omega_{X_1}(S)\oplus j_{2*}\omega_{X_2}(S)
 \xrightarrow{\operatorname{Res}_1+\operatorname{Res}_2}
 j_{S*}\omega_S\longrightarrow0.
\end{equation}
Indeed in the local description a pair $(a_1,a_2)$ lies in the
kernel precisely when its coefficients have the same restriction
to $S$. Such pairs form
$\O_{X_1}\times_{\O_S}\O_{X_2}=\O_{X_0}$, multiplied by $\eta$.
The pair $(\eta_1,\eta_2)$ therefore defines a section $\eta_0$
of $\omega_{X_0}$. Its coefficient on each branch is a unit, hence
its common coefficient in the local ring of $X_0$ is a unit.
It follows that
$\vartheta_0:\O_{X_0}\to\omega_{X_0}$, $1\mapsto\eta_0$,
is an isomorphism.

Let $p:X_0\to\Spec\CC$ and put $G=E^\vee\otimes^{\mathbf L}F$.
Proper Gorenstein duality, with $p^!\CC=\omega_{X_0}[3]$, gives
\[
 \begin{aligned}
 \RH_{X_0}(E,F)^\vee
  &=\RH_\CC(Rp_*G,\CC)\\
  &\simeq Rp_*\cRHom(G,p^!\CC)\\
  &\simeq Rp_*(F^\vee\otimes^{\mathbf L}E
                                  \otimes\omega_{X_0})[3]\\
  &\xrightarrow{\vartheta_0^{-1}}\RH_{X_0}(F,E)[3].
 \end{aligned}
\]
The evaluation and trace maps defining the equivalence are
\[
 \RH(E,F)\otimes\RH(F,E\otimes\omega_{X_0})
  \xrightarrow{\circ}\RH(E,E\otimes\omega_{X_0})
  \xrightarrow{\tr_E}R\Gamma(\omega_{X_0})
  \xrightarrow{\operatorname{Tr}_p}\CC[-3].
\]
Every arrow is natural in $E,F$; the evaluation of a perfect
complex commutes with tensor product and base change.
The complexes have finite-dimensional bounded cohomology because
$X_0$ is proper. Tensoring by $\omega_{X_0}[3]$ is an
autoequivalence and becomes $[3]$ under $\vartheta_0$.
This proves the Serre-functor assertion.
\end{proof}

\begin{remark}\label{rem:CY-types}
For a smooth category $\mathcal C$, a left Calabi--Yau structure
is a class $\theta:k[d]\to\mathrm{CC}_\bullet(\mathcal C)^{hS^1}$
whose underlying bimodule map is
\[
 \mathcal C^!\xrightarrow{\sim}\mathcal C_\Delta[-d].
\]
For a proper category, a right Calabi--Yau structure is a cyclic trace
\[
 \mathrm{CC}_\bullet(\mathcal C)_{hS^1}\longrightarrow k[-d]
\]
whose adjoint pairing is nondegenerate. The geometric construction
on $X_0$ uses the trivialization of $\omega_{X_0}$ in
Proposition~\ref{prop:singular-duality} and its derived trace
$R\Gamma(X_0,\O_{X_0})\to\CC[-3]$.
Theorem~\ref{thm:relative-symplectic} integrates the universal closed
form against this trace. The categorical construction uses
\cite[Theorem 5.6 and Corollary 6.2]{BD2}.
\end{remark}

\begin{proposition}\label{prop:fully-faithful}
For $E,F\in\Perf(X_0)$, put $E_i=\LLf j_i^*E$,
$F_i=\LLf j_i^*F$, and similarly $E_S,F_S$. There is a natural
exact triangle
\begin{equation}\label{eq:Hom-MV}
 \RH_{X_0}(E,F)\longrightarrow
 \bigoplus_{i=1}^2\RH_{X_i}(E_i,F_i)
 \xrightarrow{r_1-r_2}\RH_S(E_S,F_S)\xrightarrow{+1}.
\end{equation}
Thus restriction into
$\Perf(X_1)\times^h_{\Perf(S)}\Perf(X_2)$ is fully faithful.
The triangle commutes with derived base change over $\CC$.
At a descended object, restriction is an equivalence of the
associated characteristic-zero formal deformation problems.
\end{proposition}
\begin{proof}
Let $G=E^\vee\otimes^{\mathbf L}F$.
Tensor the normalization sequence by $G$ in the derived category.
The projection formula gives the following diagram of exact triangles:
\begin{equation}\label{eq:v9-normalization-tensor}
\begin{tikzcd}[column sep=small]
 G\arrow[r]\arrow[d,equal]&
 G\otimes^{\mathbf L}(j_{1*}\O_{X_1}\oplus j_{2*}\O_{X_2})
       \arrow[r]\arrow[d,"\simeq"]&
 G\otimes^{\mathbf L}j_{S*}\O_S\arrow[d,"\simeq"]\\
 G\arrow[r]&\displaystyle\bigoplus_i j_{i*}\LLf j_i^*G
       \arrow[r,"r_1-r_2"']&j_{S*}\LLf j_S^*G.
\end{tikzcd}
\end{equation}
Perfectness of $E$ identifies
\[
 \LLf j_i^*G\simeq
 (\LLf j_i^*E)^\vee\otimes^{\mathbf L}\LLf j_i^*F
       \simeq\cRHom_{X_i}(E_i,F_i),
\]
and the same formula holds on $S$. Applying $R\Gamma(X_0,-)$
to the lower row of \eqref{eq:v9-normalization-tensor} proves
\eqref{eq:Hom-MV}.

To identify the mapping complex of triples, set
\[
 \begin{aligned}
 A^\bullet&=\RH_{X_1}(E_1,F_1)\oplus\RH_{X_2}(E_2,F_2),\\
 B^\bullet&=\RH_S(E_S,F_S),\qquad r=r_1-r_2.
 \end{aligned}
\]
After choosing cochain representatives, the homotopy fiber is
\[
 C^n=A^n\oplus B^{n-1},\qquad
 d_C(a,b)=(d_Aa,\,r(a)-d_Bb).
\]
Its square vanishes because $d_Br=rd_A$:
\[
 d_C^2(a,b)=(0,\,rd_Aa-d_Br(a)+d_B^2b)=(0,0).
\]
A closed degree-zero pair consists of two morphisms $a_i$ and
a homotopy $b$ between their restrictions. Closed negative-degree
pairs and boundaries give the higher homotopies of morphisms of
triples. This is the mapping complex in the homotopy fiber
product. Triangle \eqref{eq:Hom-MV} identifies it with
$\RH_{X_0}(E,F)$, proving full faithfulness.

For a derived test scheme $T$, pullback of the normalization
sequence is an exact triangle on $X_0\times^hT$.
Tensor with $E_T^\vee\otimes^{\mathbf L}F_T$ and repeat
\eqref{eq:v9-normalization-tensor}; all projection-formula maps
are the base changes of the displayed maps. This proves the
base-change assertion.

For $F=E$, restriction defines a morphism of endomorphism
$E_1$-algebras to the homotopy limit of the endomorphism algebras
of the two restrictions. Its underlying complex is the preceding
homotopy fiber, so the morphism is a quasi-isomorphism. Passing to
the commutator dg Lie algebras gives a quasi-isomorphism governing
deformations. For a local Artinian dg algebra $A$ with maximal
ideal $\mathfrak m_A$, the resulting equivalence is
\[
 \operatorname{MC}_\bullet(\mathfrak g_E\otimes\mathfrak m_A)
       \simeq
 \operatorname{MC}_\bullet(\mathfrak g_{(E_1,E_2,\alpha)}
                                      \otimes\mathfrak m_A).
\]
Here the Lie algebra on the right is a homotopy-limit model;
the displayed cone specifies its underlying complex. This is the
characteristic-zero dg Lie description of formal deformation
problems, and proves the final assertion.
\end{proof}

\begin{proposition}\label{prop:effective}
Let $(E_1,E_2,\alpha)\in\mathfrak G$.
Assume that locally on $X_0$ there are bounded vector-bundle
complexes $P_i^\bullet$ representing $E_i$ and a chain
isomorphism $\alpha:P_1^\bullet|_S\xrightarrow{\sim}P_2^\bullet|_S$
which is an isomorphism in every degree.
Then the triple is induced by a perfect complex on $X_0$.
In particular the conclusion holds for vector bundles with a
specified isomorphism on $S$.
\end{proposition}
\begin{proof}
At a point of $S$, use the local rings
\[
 R=A_1\times_D A_2
   =\{(a_1,a_2):\bar a_1=\bar a_2\text{ in }D\},
 \qquad A_i\twoheadrightarrow D.
\]
Trivialize $P_i^n=A_i^{r_n}$. The equality of the ranks follows
from the degreewise isomorphism on $S$. Its matrix
$g_n\in\GL_{r_n}(D)$ lifts to a matrix
$\widetilde g_n\in\GL_{r_n}(A_2)$: lift the entries and use that
$\det\widetilde g_n$ maps to a unit in the local quotient $D$.
Changing the $A_2$-basis makes $g_n=1$.
If $d_i^n$ denote the resulting differential matrices, the
chain-isomorphism condition reads
\[
 \overline{d_1^n}=\overline{d_2^n}\quad\text{in }\operatorname{Mat}_{r_{n+1}\times r_n}(D).
\]
Set
\[
 P^n=\{(v_1,v_2)\in A_1^{r_n}\oplus A_2^{r_n}:
                         \bar v_1=\bar v_2\}=R^{r_n},
 \qquad d^n=(d_1^n,d_2^n).
\]
Every entry of $d^n$ belongs to $R$, and
\[
 d^{n+1}d^n=(d_1^{n+1}d_1^n,d_2^{n+1}d_2^n)=(0,0).
\]
Thus $P^\bullet$ is a bounded finite-free $R$-complex and
\[
 P^\bullet\otimes_R A_i\simeq P_i^\bullet,
 \qquad
 (P^\bullet\otimes_R A_1)\otimes_{A_1}D
  \xrightarrow{1}
 (P^\bullet\otimes_R A_2)\otimes_{A_2}D
\]
recovers $\alpha$ after undoing the basis change. Since the terms
are free, these ordinary tensors compute the derived restrictions.
The local equalizer is displayed by
\[
 0\longrightarrow P^\bullet\longrightarrow
 P_1^\bullet\oplus P_2^\bullet
 \xrightarrow{\alpha\circ r_1-r_2}P_2^\bullet|_S\longrightarrow0.
\]
Surjectivity holds degreewise because $A_2\to D$ is surjective.

Choose such constructions on an open cover $\{U_\lambda\}$ of
$X_0$. On $U_{\lambda\mu}$ both constructed complexes restrict to
the prescribed triple. Proposition~\ref{prop:fully-faithful}
identifies the mapping spaces by
\[
 \Map(P_\lambda|_{U_{\lambda\mu}},P_\mu|_{U_{\lambda\mu}})
 \xrightarrow{\sim}
 \Map(rP_\lambda,rP_\mu).
\]
Lift the given overlap equivalences through this equivalence.
On triple intersections lift their specified homotopy
$\beta_{\mu\nu}\beta_{\lambda\mu}\simeq\beta_{\lambda\nu}$;
the same equivalence of mapping spaces lifts each higher simplex.
We obtain an object of
\[
 \operatorname*{holim}_{[n]\in\Delta}
       \Perf\!\left(\coprod_{\lambda_0,\ldots,\lambda_n}
                     U_{\lambda_0\cdots\lambda_n}\right).
\]
Descent for perfect complexes identifies this homotopy limit with
$\Perf(X_0)$. Its descended object has the required restrictions
and matching isomorphism, because these identifications were part
of the descent datum.
\end{proof}

\begin{definition}\label{def:effective-locus}
An object $(E_1,E_2,\alpha)$ of $\fG$ over a test scheme is \emph{effective} if it is induced by a perfect complex on $X_0$ over that test scheme. We write $\fG^{\eff}$ for the essential-image subfunctor. A chosen stable special-fiber moduli functor must specify a locus in this subfunctor, its numerical class, and its stability condition. Proposition~\ref{prop:effective} is a sufficient test for effectivity; it is not a properness theorem.
\end{definition}
The essential-image functor is
\[
 \Perf(X_0)\longrightarrow
 \Perf(X_1)\times^h_{\Perf(S)}\Perf(X_2),
 \qquad E\longmapsto(E_1,E_2,\alpha_E).
\]
Proposition~\ref{prop:fully-faithful} identifies its mapping complexes.
Proposition~\ref{prop:effective} constructs preimages for the triples
satisfying its degreewise vector-bundle hypothesis.

\subsubsection{Objects and morphisms in the derived fiber product}
An object of the gluing stack over a test derived scheme $T$ is a triple $(E_1,E_2,\alpha)$, where $E_i$ is perfect on $X_i\times^hT$ and $\alpha$ is an equivalence of their derived restrictions to $S\times^hT$. Morphisms include compatible morphisms of the two complexes and a homotopy between their restrictions. Higher homotopies are part of the datum. Replacing this stack by an ordinary fiber product of coarse moduli spaces loses these homotopies and, with them, part of the obstruction theory.

At a triple, the tangent complex is the homotopy fiber of the difference map
\begin{equation}\label{eq:gluing-tangent-full}
 A_1[1]\oplus A_2[1]\longrightarrow A_S[1],
 \qquad A_i=\RH_{X_i}(E_i,E_i),\quad
 A_S=\RH_S(E_S,E_S).
\end{equation}
Writing $\mathbb T=\mathbb T_{\mathfrak G,(E_1,E_2,\alpha)}$,
the associated long exact sequence contains
\[
 \begin{aligned}
 \bigoplus_i\End_{X_i}(E_i)&\longrightarrow\End_S(E_S)
       \longrightarrow H^0(\mathbb T)\\
 &\longrightarrow\bigoplus_i\ext^1_{X_i}(E_i,E_i)
       \longrightarrow\ext^1_S(E_S,E_S)\\
 &\longrightarrow H^1(\mathbb T)
       \longrightarrow\bigoplus_i\ext^2_{X_i}(E_i,E_i).
 \end{aligned}
\]
The first connecting map changes the matching equivalence $\alpha$.
The next one is the obstruction to matching the two first-order
component deformations. For a perfect complex $E$ on $X_0$,
Proposition~\ref{prop:fully-faithful} identifies this exact sequence
with the corresponding sequence for $\RH_{X_0}(E,E)[1]$.

\begin{remark}\label{rem:gluing-form-comparison}
Let $\rho:\mathfrak X_0=\RPerf(X_0)\to\mathfrak G$ be restriction,
and let $\omega_0,\omega_{\mathfrak G}$ be the closed forms determined
by the respective constructions. Compatibility of these forms is a
homotopy
\[
 \rho^*\omega_{\mathfrak G}\simeq\omega_0
      \quad\text{in }\mathcal A^{2,cl}(\mathfrak X_0,-1).
\]
Its underlying pairing gives the commutative diagram
\[
\begin{tikzcd}[column sep=large,row sep=large]
 \mathbb T_{\mathfrak X_0}\arrow[r,"d\rho","\sim"']
      \arrow[d,"\Theta_{\omega_0}"']&
 \rho^*\mathbb T_{\mathfrak G}\arrow[d,"\rho^*\Theta_{\omega_{\mathfrak G}}"]\\
 \mathbb L_{\mathfrak X_0}[-1]&
 \rho^*\mathbb L_{\mathfrak G}[-1]\arrow[l,"{(d\rho)^\vee[-1]}","\sim"'].
\end{tikzcd}
\]
Here $d\rho$ is the tangent equivalence of
Proposition~\ref{prop:fully-faithful}. When comparing the oriented
DT complexes, this closed-form homotopy and the isomorphism of their
chosen orientation bundles are included in the comparison datum.
\cite[Section~7]{CompanionII} gives these comparisons on the stated
smooth locus of hyperplane-section objects.
\end{remark}

The restriction calculation applies to the entire complexes
$\LLf\iota_i^*F_i$, including their Tor sheaves. We now place this central
moduli problem in the relative family and compute the additional
fourfold duality needed for the foliation construction.

\subsection{Calabi--Yau structures when the canonical bundle is pulled back from the base}\label{sssec: Pull-back divisor}
We now compare the relative threefold pairing with the proper-support
fourfold pairing. The relative dualizing sheaf gives the former;
the pullback expression for $\omega_{\cX}$ and a frame near the
support give the latter. Keeping their base twists is necessary for
applying the foliation construction in families.

\subsubsection{Relative dimension three}

Put $\fM=\RPerf(\cX/C)$ on a derived open Artin finite-presentation locus. On $\cX\times_C^h\fM$ let $\mathcal E$ be the universal perfect complex and let $q$ denote projection to $\fM$.

\begin{theorem}\label{thm:relative-symplectic}
Under Definition~\ref{def:good}, the stack $\fM\to C$ has a relative $(-1)$-shifted symplectic structure. Its tangent and cotangent complexes are
\begin{equation}\label{eq:relative-tangent}
 \T_{\fM/C}\simeq\R q_*\cRHom(\mathcal E,\mathcal E)[1],
 \qquad
 \ttL_{\fM/C}\simeq
 \bigl(\R q_*\cRHom(\mathcal E,\mathcal E)\bigr)^{\vee}[-1].
\end{equation}
The form and these identities commute with derived base change. In the untrivialized situation of Proposition~\ref{prop:base-line}, the induced duality is
\begin{equation}\label{eq:twisted-relative}\IHESaliaslabel{obs: Canonicicty}{eq:twisted-relative}\IHESaliaslabel{eqn: Decomp}{eq:twisted-relative}
 \T_{\fM/C}\simeq\ttL_{\fM/C}[-1]\otimes p^*\lambda^{\vee},
\end{equation}
where $p:\fM\to C$. A frame of $\lambda$ turns this into the untwisted pairing.
\end{theorem}
\begin{proof}
A projective flat Gorenstein morphism is proper and perfect. Pushforward of a perfect complex is perfect, and derived base change holds. Thus it satisfies the relative $\O$-compactness conditions. The chosen relative volume form and duality trace give
\[
 \R\pi_*\O_{\cX}\xrightarrow{\vartheta}
 \R\pi_*\omega_{\cX/C}\longrightarrow\O_C[-3].
\]
For every perfect complex $G$ after a derived base change, relative duality identifies the pairing against $G^{\vee}$ with its perfect dual. This is the nondegeneracy required in an $\O$-orientation.

The stack of perfect complexes has its universal $2$-shifted symplectic form. Applying the mapping-stack construction over each affine derived base and integrating that closed form against the preceding relative trace gives a closed two-form of degree $2-3=-1$. The constructions in the proof of \cite[Theorem 2.5]{PTVV13} are compatible with base change and therefore descend over $C$. This is a construction of a closed form, not only of a pairing on cohomology.

Write $A=\R q_*\cRHom(\mathcal E,\mathcal E)$. Its underlying form is the shifted composition--trace pairing
\[
 A[1]\otimes A[1]\longrightarrow
 \R q_*\O[2]\longrightarrow\O_{\fM}[-1].
\]
The trace pairing identifies $\cRHom(\mathcal E,\mathcal E)$ with its perfect dual. Relative duality gives $A^{\vee}\simeq A[3]$, so $A[1]\simeq A^{\vee}[-2]=\ttL_{\fM/C}[-1]$. This verifies nondegeneracy. Without a frame of $\lambda$, duality instead gives $A^{\vee}\simeq A\otimes p^*\lambda[3]$, which is exactly \eqref{eq:twisted-relative}. All maps used are relative evaluation, trace, and pushforward maps, so commute with derived base change.
\end{proof}

\begin{remark}
There is still a triangle
\[
 \T_{\fM/C}\longrightarrow\T_{\fM}
 \longrightarrow p^*T_C\xrightarrow{+1}.
\]
The relative form of Theorem~\ref{thm:relative-symplectic} is not an absolute $(-2)$-shifted symplectic form on $\fM$. In particular, adjoining a base direction to a $(-1)$-shifted tangent pairing does not change its shift by itself. A separate construction would be needed for that conclusion.
\end{remark}

\begin{proposition}\label{prop:supported-serre}\IHESaliaslabel{CY4-category}{prop:supported-serre}
Let $\pi:\cX\to C$ be projective, with $\cX$ smooth of dimension
four, and fix an isomorphism
\[
 \varphi:\omega_{\cX}\xrightarrow{\sim}\pi^*\mathcal B,
 \qquad \mathcal B=\omega_C\otimes\lambda.
\]
For a finite set $Z\subset C$ of closed points, put
$W=\pi^{-1}Z$ and let $\Perf_W(\cX)$ consist of perfect complexes
whose cohomology sheaves have set-theoretic support in $W$.
This category is proper and has Serre functor
\[
 \mathcal S(E)=E\otimes\omega_{\cX}[4].
\]
If $\epsilon:\O_U\xrightarrow{\sim}\mathcal B|_U$ is a
trivialization on an open neighborhood $U$ of $Z$, then
$\epsilon$ and $\varphi$ give a natural equivalence
$\mathcal S|_{\Perf_W(\cX)}\simeq[4]$.
For $E=i_{b*}G$, $G\in\Perf(X_b)$, the untrivialized pairing is
\begin{equation}\label{eq:v9-point-twisted-fourfold}
 \RH_{\cX}(E,E)\otimes\RH_{\cX}(E,E)
          \longrightarrow\mathcal B_b^\vee[-4].
\end{equation}
The corresponding ambient tangent pairing has target
$\mathcal B_b^\vee[-2]$.
\end{proposition}
\begin{proof}
For $E,F\in\Perf_W(\cX)$, the complex
$E^\vee\otimes^{\mathbf L}F$ is perfect and its cohomology is
supported in the proper scheme $W$. On a smooth fourfold its
cohomology is bounded. Properness of the support therefore gives
\[
 \dim_\CC\ext^j_{\cX}(E,F)<\infty,
 \qquad\ext^j_{\cX}(E,F)=0\quad\text{for }|j|\gg0.
\]
Serre duality for complexes with proper support on the smooth
fourfold gives the natural isomorphism
\begin{equation}\label{eq:v9-support-duality}
 \RH_{\cX}(E,F)^\vee
   \simeq\RH_{\cX}(F,E\otimes\omega_{\cX})[4].
\end{equation}
Its adjoint pairing is evaluation, sheaf trace, and the
proper-support duality trace. Tensoring by an invertible sheaf
preserves $W$, and its inverse is tensoring by the dual bundle.
Thus the right side represents a Serre autoequivalence of
$\Perf_W(\cX)$.

We now use the pullback formula. On $\cX_U$ the composite
\begin{equation}\label{eq:v9-total-volume}
 \eta_U:\O_{\cX_U}\xrightarrow{\pi^*\epsilon}
 \pi^*(\mathcal B|_U)\xrightarrow{\varphi^{-1}}
 \omega_{\cX_U}
\end{equation}
is an isomorphism. For each $E$ it defines the morphism
$\mu_{E,U}=\id_E\otimes\eta_U$ from $E$ to
$E\otimes\omega_{\cX}$ on this open set.
Restriction to an open neighborhood of $W$ is fully faithful on
these complexes: the sheaf derived Hom has cohomology supported
in $W$, so its derived sections on $\cX$ and $\cX_U$ agree.
It follows that $\mu_{E,U}$ is a global morphism on $\cX$.
Its cone restricts to zero on $\cX_U$ and has zero cohomology
outside $W$, hence is zero. Therefore $\mu_{E,U}$ is an
isomorphism. For every morphism $a:E\to F$ the square
\begin{equation}\label{eq:v9-volume-natural}
\begin{tikzcd}[column sep=large]
 E\arrow[r,"\mu_{E,U}"]\arrow[d,"a"']&
 E\otimes\omega_{\cX}\arrow[d,"a\otimes1"]\\
 F\arrow[r,"\mu_{F,U}"']&F\otimes\omega_{\cX}
\end{tikzcd}
\end{equation}
commutes, since both composites are $a\otimes\eta_U$.
The same equality is an equality of the maps of sheaf Hom
complexes. Thus the equivalence is natural, including its
homotopy coherence. Substituting it in
\eqref{eq:v9-support-duality} gives
\[
 \RH_{\cX}(E,F)^\vee\simeq\RH_{\cX}(F,E)[4].
\]
The use of an open neighborhood, rather than only the reduced
fiber, means that $\mu_{E,U}$ acts on complexes whose cohomology
is annihilated by any power of the fiber ideal.

For the last assertion, the projection formula and
$\pi i_b=b$ give
\[
 \begin{aligned}
 i_{b*}G\otimes\omega_{\cX}
 &\xrightarrow{\varphi}i_{b*}(G\otimes i_b^*\pi^*\mathcal B)\\
 &\simeq i_{b*}G\otimes_\CC\mathcal B_b.
 \end{aligned}
\]
Writing $A_E=\RH_{\cX}(E,E)$, equation
\eqref{eq:v9-support-duality} becomes
\[
 A_E^\vee\simeq A_E\otimes\mathcal B_b[4].
\]
Its adjoint is \eqref{eq:v9-point-twisted-fourfold}.
Shifting both factors by $[1]$ gives the tangent pairing with
target $\mathcal B_b^\vee[-2]$.
\end{proof}

For a family of complexes on $\cX\times T$ of the form
$E=j_{\beta*}G$, where
$j_\beta:\cX\times_C T\hookrightarrow\cX\times T$ and
$\beta:T\to C$, projection formula gives
\begin{equation}\label{eq:v9-graph-canonical-twist}
 E\otimes\operatorname{pr}_{\cX}^*\omega_{\cX}
     \simeq E\otimes q^*\beta^*\mathcal B,
 \qquad q:\cX\times T\to T.
\end{equation}
Assume the family and its support satisfy proper-support perfect
pushforward and base change. Relative fourfold duality then gives
\[
 \mathcal A_E^\vee\simeq
       \mathcal A_E\otimes\beta^*\mathcal B[4],
 \qquad\mathcal A_E=Rq_*\cRHom(E,E).
\]
These are the ambient deformation complexes restricted to this
family. The deformation complex within $X_b$ is compared to them
by \eqref{eq:normal-deformation}.

Let $\epsilon_\alpha$ be frames of $\mathcal B$ on a base cover,
with $\epsilon_\beta=g_{\alpha\beta}\epsilon_\alpha$.
The induced scalar trace pairings satisfy
\begin{equation}\label{eq:v9-pairing-transition}
 \eta_\beta=\pi^*g_{\alpha\beta}\,\eta_\alpha,
 \qquad
 \langle-,-\rangle_\beta
        =\beta^*g_{\alpha\beta}\langle-,-\rangle_\alpha.
\end{equation}
This follows by inserting $\eta_\beta$ into the trace composition.
The untrivialized pairing has target $\beta^*\mathcal B^\vee$;
evaluation of that factor against $\epsilon_\beta$ gives the
second equality. Hence a single frame on the base gives the
compatible pairings on every fiber and in every infinitesimal
base direction.

\subsubsection{Perfect complexes supported on fibers and their infinitesimal thickenings}
The trivialization used in Proposition~\ref{prop:supported-serre} must act
on the whole deformation of an object, including its infinitesimal
thickenings. The following annihilation lemma makes that requirement
precise. It is then used to restrict the fourfold volume form to the
proper-support moduli problem.

\begin{lemma}\label{lem:formal-support}
Let $W$ be a noetherian scheme, let
$\mathcal I\subset\O_W$ be a coherent ideal sheaf, and define
\[
 V(\mathcal I)=
   \{w\in W:\mathcal I_w\subseteq\mathfrak m_w\},
\]
with its closed-subscheme structure $\O_W/\mathcal I$.
Let $E\in D^b_{\mathrm{coh}}(W)$ satisfy
$\Supp\mathcal H^j(E)\subseteq V(\mathcal I)$ for every $j$.
Then there exists $N$ with
$\mathcal I^N\mathcal H^j(E)=0$ for all $j$.
A trivialization of a line bundle on an open neighborhood of
$V(\mathcal I)$ induces a natural tensor trivialization on these
complexes.
\end{lemma}
\begin{proof}
On $\Spec A\subset W$, write $I=(a_1,\ldots,a_r)$ for the ideal
of $\mathcal I$ and $M=H^j(E|_{\Spec A})$.
Since $M$ is finite,
\[
 \Supp M=V(\operatorname{Ann}_A M)\subseteq V(I)
 \quad\Longrightarrow\quad
 I\subseteq\sqrt{\operatorname{Ann}_A M}.
\]
Choose $n_i\ge1$ with $a_i^{n_i}M=0$ and put
$N_j=1+\sum_i(n_i-1)$.
A monomial of total degree $N_j$ in the $a_i$ has an exponent
at least $n_i$ for some $i$; hence $I^{N_j}M=0$.
There are finitely many nonzero $\mathcal H^j(E)$ and a finite
affine cover of $W$. Taking the maximum of their exponents gives
one $N$ on all of $W$.

Let $j:U\hookrightarrow W$ contain $V(\mathcal I)$ and let
$\epsilon:\O_U\xrightarrow{\sim}\mathcal L|_U$.
The morphism
$\id_{E|_U}\otimes\epsilon:E|_U\to(E\otimes\mathcal L)|_U$
extends uniquely through the full-faithfulness of restriction on
complexes with cohomology supported in $V(\mathcal I)$.
The cone has zero restriction to $U$ and zero cohomology outside
$V(\mathcal I)$, so it is zero. Naturality is the tensor identity
$(a\otimes1)(1\otimes\epsilon)=(1\otimes\epsilon)a$.
\end{proof}
For the application to Proposition~\ref{prop:supported-serre},
$W=\cX$ and $\mathcal I$ is the ideal sheaf of $\pi^{-1}Z$.
The exponent controls the cohomology sheaves; the natural tensor
map is defined on the open neighborhood in
\eqref{eq:v9-total-volume}.

\begin{corollary}\label{cor:local-CY4}
Under Proposition~\ref{prop:supported-serre}, the volume form
$\eta_U$ gives $\Perf(\cX_U)$ its geometric left Calabi--Yau
structure of dimension four. The moduli stack
$\mathfrak M_c(\cX_U)$ of perfect complexes with proper support
has its induced $(-2)$-shifted symplectic structure on locally
geometric finite-presentation loci.
\end{corollary}
\begin{proof}
The variety $\cX_U$ is smooth, separated and of finite type.
Its diagonal is a regular immersion of codimension four, so the
inverse dualizing bimodule of $\Perf(\cX_U)$ is represented by
\[
 \Delta_*(\omega_{\cX_U}^{-1})[-4].
\]
The inverse of \eqref{eq:v9-total-volume} identifies this kernel
with $\Delta_*\O_{\cX_U}[-4]$.
The top form $\eta_U\in\Gamma(\cX_U,\Omega^4)$ is closed:
$d\eta_U\in\Gamma(\cX_U,\Omega^5)=0$.
Under the mixed HKR description its Hochschild class therefore
has a negative cyclic lift with representative $\eta_U$ and
zero higher coefficients. Its underlying bimodule map is the
preceding diagonal equivalence. This is the geometric CY class
used in \cite[Section~6.1]{BD2}.

For a universal proper-support family $\mathcal E$ set
$\mathcal A=Rq_*\cRHom(\mathcal E,\mathcal E)$.
The geometric moduli theorem sends the negative cyclic class to
a closed two-form. Its underlying pairing is
\[
 \mathcal A[1]\otimes\mathcal A[1]
       \xrightarrow{\circ}\mathcal A[2]
       \xrightarrow{\tr_{\eta_U}}\O[-2].
\]
Proper-support duality identifies $\mathcal A^\vee\simeq\mathcal A[4]$,
so the adjoint is
$\mathcal A[1]\xrightarrow{\sim}\mathcal A^\vee[-3]
 =\mathbb L_{\mathfrak M_c}[-2]$.
This proves nondegeneracy. The negative cyclic representative
supplies closedness and the universal construction supplies base
change, as in the geometric construction of \cite[Section 6.1]{BD2}.
\end{proof}
On the affine resolved quintic family,
$\mathcal B=\omega_C$ and the coordinate $t$ gives the frame
$dt$. In particular, the class and pairing above are defined on
the entire fourfold over that affine base. On the projective pencil,
$\mathcal B=\O_{\P^1}(-1)$ and the transition rule is
\eqref{eq:v9-pairing-transition}.

The chosen frame of $\omega_C\otimes\lambda$ has now supplied the
fourfold volume form and the proper-support shifted symplectic
structure. Section~\ref{sec:local} uses this form on the specified
Quot charts to construct the real Lagrangian foliation and to write
its shifted potential. The base-frame transition formulas above
remain part of this construction when more than one base open is used.

\section{Global Lagrangian foliation}\label{sec:local}\IHESaliaslabel{ssec: Relative lag}{sec:local}
We apply the fourfold pairing of Section~\ref{sec:CY} to a derived Quot
presentation of objects with proper support. The goal is to construct
the real Lagrangian foliation and its degree-$-1$ potential, with the
transition data needed to compare charts. After that construction,
we compute the relation between ambient and fiberwise deformation
complexes and give the graded factorization criterion. This last
criterion identifies the additional information needed to extract an
ordinary family of potentials for the Gauss--Manin construction.

\subsection{Distributions on derived quotient stacks}
The definitions below specify the mixed differential, isotropic
homotopy, and equivariance retained during gluing. Each of these enters
the strictification argument in
Corollary~\ref{cor:ambient-foliation}.

\subsubsection{Integrable distributions and Lagrangian foliations}
Let $A$ be a connective commutative dg algebra with internal
differential $\delta$ of degree $1$. Its graded mixed de Rham
algebra has
\[
 \DR(A)^{p,*}=\Sym_A^p(\mathbb L_A[1]),\qquad
 \epsilon:\DR(A)^{p,q}\to\DR(A)^{p+1,q-1},
\]
where
\[
 \delta^2=\epsilon^2=0,\qquad\delta\epsilon+\epsilon\delta=0.
\]
The weight is $p$ and the cohomological degree is $q$.

\begin{definition}\label{def:distribution}\IHESaliaslabel{defn: Integrable distribution}{def:distribution}
An integrable distribution on $\mathfrak U=\Spec A$ is a
morphism of graded mixed algebras
$\alpha:\DR(A)\to\Lambda$ with
\[
 \Lambda^0\simeq A,\qquad \Lambda^1\in\Perf(A),\qquad
 \Lambda^{\mathrm{gr}}\simeq\Sym_A(\Lambda^1).
\]
The last isomorphism is in graded dg algebras after forgetting
$\epsilon$. A morphism of distributions commutes with $\alpha$,
$\delta$ and $\epsilon$; equivalences are quasi-isomorphisms of
these graded mixed algebras.
\end{definition}
Dualizing the weight-one map gives
\begin{equation}\label{eq:anchor-correct}\IHESaliaslabel{eqn: Anchor}{eq:anchor-correct}
 \rho:L=(\Lambda^1)^\vee[1]\longrightarrow\mathbb T_{\mathfrak U}.
\end{equation}
In a strict dg Lie algebroid presentation, $\epsilon^2=0$ expresses
\[
 \begin{aligned}
 [x,ay]&=\rho(x)(a)y+(-1)^{|x||a|}a[x,y],\\
 \rho([x,y])&=[\rho(x),\rho(y)],\\
 (-1)^{|x||z|}[x,[y,z]]&+
 (-1)^{|y||x|}[y,[z,x]]+
 (-1)^{|z||y|}[z,[x,y]]=0.
 \end{aligned}
\]
The identity $\delta\epsilon+\epsilon\delta=0$ gives
$\delta[x,y]=[\delta x,y]+(-1)^{|x|}[x,\delta y]$
and $\rho\delta=\delta\rho$. For a nonstrict presentation the
higher operations are the components of the corresponding
square-zero Chevalley--Eilenberg coderivation; the graded mixed
algebra definition retains these operations.

\begin{definition}\label{def:purely-derived}\IHESaliaslabel{defn: Purely derived foliation}{def:purely-derived}
In a well-presented derived $C^\infty$ chart, a distribution is
purely derived if $H^i(L)=0$ for $i\le0$ and locally on minimal
charts it admits a representative with injective anchor.
A strict representative has an inclusion of complexes
$\rho:L^\bullet\hookrightarrow\mathbb T_{\mathfrak U}^\bullet$
together with the specified compatible mixed differential.
\end{definition}
A well-presented model has $A^0$ smooth of finite type and
\[
 A^{\mathrm{gr}}=\Sym_{A^0}\left(\bigoplus_{j<0}P^j\right),
\]
where each $P^j$ is finite projective over $A^0$ and only finitely
many $P^j$ are nonzero.
In a strict $(-2)$-shifted Darboux chart, use coordinates
$|x_i|=0$, $|y_a|=-1$, $|z_i|=-2$. Its underlying form is
\begin{equation}\label{eq:minus-two-Darboux}\IHESaliaslabel{ex: Well-presented}{eq:minus-two-Darboux}
 \omega=\sum_i d_{\mathrm{dR}}x_i\,d_{\mathrm{dR}}z_i
       +\sum_a d_{\mathrm{dR}}y_a\,d_{\mathrm{dR}}y_a.
\end{equation}
The symmetric degree-one tangent block is the dual of the
$y_a$ block. The differential obeys
$\delta\omega=0$ and $\delta^2=0$. The closed form is the full
graded mixed representative, with
$d_{\mathrm{dR}}\omega=0$ in the strict chart.

\subsubsection{Isotropic and Lagrangian structures on a distribution}\label{notate: Lagrangian foliations}
Let $\omega$ be an $n$-shifted symplectic form on
$\mathfrak M$ and $\rho:L\to\mathbb T_{\mathfrak M}$ a
distribution. Its restricted pairing is
\[
 L\otimes L\xrightarrow{\rho\otimes\rho}
 \mathbb T_{\mathfrak M}\otimes\mathbb T_{\mathfrak M}
 \xrightarrow{\omega}\O_{\mathfrak M}[n].
\]
An isotropic structure specifies a nullhomotopy of this pairing
in the compatible closed-form complex. On chain representatives
it is a map $h$ of degree $n-1$ satisfying
\[
 \delta h-(-1)^{n-1}h\delta=\rho^*\omega.
\]
The nullhomotopy extends the map
$\mathbb T_{\mathfrak M}\xrightarrow{\Theta_\omega}
 \mathbb L_{\mathfrak M}[n]\xrightarrow{\rho^\vee[n]}L^\vee[n]$
over the cone of $\rho$, yielding
\[
 \Theta_h:\operatorname{cofib}(L\xrightarrow{\rho}\mathbb T_{\mathfrak M})
                \longrightarrow L^\vee[n].
\]
The Lagrangian condition is that $\Theta_h$ be a quasi-isomorphism.
For $n=-2$ this gives
\[
 \operatorname{cofib}(\rho)\simeq L^\vee[-2],\qquad
 \operatorname{fib}(\rho)\simeq L^\vee[-3].
\]
For the fourfold form, the condition is imposed on
$\omega_{\mathrm{im}}$; the restriction of
$\omega_{\mathrm{re}}$ to $H^1(L)$ is negative definite.

\subsubsection{Equivariance and the quotient stack}
For a right action of $G=\GL_n(\CC)$ on $Q$, put
$Q_r=Q\times G^r$. The face maps are
\[
 \begin{aligned}
 d_0(q,g_1,\ldots,g_r)&=(qg_1,g_2,\ldots,g_r),\\
 d_i(q,g_1,\ldots,g_r)&=(q,g_1,\ldots,g_ig_{i+1},\ldots,g_r)
                                     &&(0<i<r),\\
 d_r(q,g_1,\ldots,g_r)&=(q,g_1,\ldots,g_{r-1}).
 \end{aligned}
\]
Degeneracies insert the identity of $G$. The quotient is
$[Q/G]\simeq\operatorname*{hocolim}_{\Delta^{op}}Q_r$.
For the groupoid of distributions this gives
\begin{equation}\label{eq:distribution-descent}\IHESaliaslabel{eqn: L(X)}{eq:distribution-descent}
 \operatorname{Dist}([Q/G])\simeq
       \operatorname*{holim}_{[r]\in\Delta}\operatorname{Dist}(Q_r).
\end{equation}
In degrees zero, one and two an object is
\[
 \Lambda_0,\qquad
 \gamma:d_0^*\Lambda_0\xrightarrow{\sim}d_1^*\Lambda_0,
 \qquad
 d_1^*\gamma\simeq d_2^*\gamma\circ d_0^*\gamma,
\]
with all maps in graded mixed algebras. On $Q_3$ the two
composites of this homotopy agree up to a specified next
homotopy. Higher simplices give the remaining compatibilities.
The BKSY strictification chooses representatives in which the
invariant subcomplex and these mixed structures define the same
object of \eqref{eq:distribution-descent}.

\subsubsection{Dividing by the distribution}
For a morphism of derived stacks $q:\mathfrak M\to\mathfrak N$,
define its formal quotient by the tangent distribution as
\begin{equation}\label{eq:formal-quotient}\IHESaliaslabel{lem: Quotient}{eq:formal-quotient}
 \mathfrak N\times^h_{\mathfrak N_{\mathrm{dR}}}
                          \mathfrak M_{\mathrm{dR}}.
\end{equation}
Here $\mathfrak Z_{\mathrm{dR}}(T)=\mathfrak Z((t_0T)_{\mathrm{red}})$.
The canonical morphism, with its target written explicitly, is
\[
 \mathfrak M\xrightarrow{(q,\mathrm{red})}
 \mathfrak N\times^h_{\mathfrak N_{\mathrm{dR}}}
                          \mathfrak M_{\mathrm{dR}}.
\]
On a test scheme, it sends $m:T\to\mathfrak M$ to
$(q m,m|_{(t_0T)_{\mathrm{red}}},\id)$.
A point of the target is $(n,m_0,\alpha)$, where
$n:T\to\mathfrak N$, $m_0:(t_0T)_{\mathrm{red}}\to\mathfrak M$,
and $\alpha:q m_0\simeq n|_{(t_0T)_{\mathrm{red}}}$.
If $\mathfrak M_{\mathrm{dR}}\to\mathfrak N_{\mathrm{dR}}$
is an equivalence, projection to $n$ identifies
\eqref{eq:formal-quotient} with $\mathfrak N$.
For the BKSY quotient the bundle-and-cosection datum below
supplies the shifted potential in addition to this quotient
construction.

The real foliation uses both real and imaginary parts of the same
complex shifted form. Lemma~\ref{lem:real-imag} constructs these forms
on the associated real complexes. On the stable proper-support Quot
locus, Proposition~\ref{prop:Quot-support-contained-in-Z} identifies
their complex source with the fourfold trace pairing. We first compute
the tangent complexes of the Quot presentation, so that the framing
directions are retained in this pullback.

\subsection{Derived Quot-schemes and their shifted symplectic structures}\label{ssec: Derived Quot-schemes}

Fix a relatively ample bundle $\O_{\cX}(1)$, a Hilbert polynomial $P$, an integer $m$, and a vector space $V$ of dimension $P(m)$. The ordinary relative Quot scheme parametrizes flat coherent quotients
\[
 V\otimes\O_{\cX_T}(-m)\twoheadrightarrow F.
\]
It is projective over $C$. In a bounded $m$-regular locus, require that the induced map $V\to H^0(F(m))$ be an isomorphism. This condition defines the open subscheme of the Quot scheme used to present the moduli stack; for derived Quot presentations see \cite{DGQuotI}. Requiring stability or perfectness on the fiber further restricts the Quot scheme; such an open is not thereby proper.

\begin{proposition}\label{prop:quot-tangent}
At a quotient on a fiber
\[
 0\longrightarrow K\longrightarrow V\otimes\O_{X_b}(-m)
 \longrightarrow F\longrightarrow0,
\]
the relative derived Quot tangent is $\RH_{X_b}(K,F)$. The forgetful map to the stack of objects is part of the exact triangle
\begin{equation}\label{eq:quot-triangle}
 \RH(F,F)\longrightarrow
 \RH(V\otimes\O(-m),F)\longrightarrow
 \RH(K,F)\longrightarrow\RH(F,F)[1].
\end{equation}
On the framed $m$-regular locus, the quotient stack $[Q/\mathrm{GL}(V)]$ presents the sheaf stack on that locus. A shifted symplectic structure on the latter is a structure on the quotient stack, not automatically a nondegenerate form on $Q$.
\end{proposition}
\begin{proof}
The first assertion is the deformation complex of a morphism with fixed source and varying target subject to the quotient condition; equivalently it follows by linearizing the homotopy fiber defining derived quotients. Applying $\RH(-,F)$ to the quotient triangle gives \eqref{eq:quot-triangle}. Under $m$-regularity, the middle term is the vector space of endomorphisms of the framing, in degree zero. Dividing by its change-of-basis action gives the object deformation complex. The quotient presentation is therefore compatible with \eqref{eq:quot-triangle}; it does not make the framing directions vanish on the atlas itself.
\end{proof}

\begin{lemma}\label{lem:atlas}\IHESaliaslabel{rmk: Inheriting -2}{lem:atlas}
For a smooth group $G$ acting on $Q$, the atlas $q:Q\to[Q/G]$ has
\[
 \T_{Q/[Q/G]}\simeq\mathfrak g\otimes\O_Q.
\]
The pullback of a shifted two-form from $[Q/G]$ annihilates the image of this relative tangent in $\T_Q$. On a stable framed locus with nonzero effective orbit directions, this prevents that pullback from being nondegenerate. Formal etaleness of a further map $[Q/G]\to\fM$ does not change this calculation.
\end{lemma}
\begin{proof}
The morphism $q$ is a principal $G$-bundle in the stack-theoretic sense. Its relative tangent is the Lie algebra bundle, and the tangent triangle gives the assertion. Pullback of the form factors through $q^*\T_{[Q/G]}$ in each argument, so the relative tangent maps to zero under the induced pairing. On the stated locus those orbit vectors are nonzero.
\end{proof}

\subsubsection{The tangent triangle for forgetting a framing}
A Quot presentation records more than an object: it records a quotient from a fixed source, and on a suitable open set it records a basis of global sections after twisting. Let $q:Q\to[Q/G]$ denote the atlas. Its tangent triangle is
\begin{equation}\label{eq:atlas-tangent}
 \mathfrak g\otimes\O_Q\longrightarrow\T_Q
 \longrightarrow q^*\T_{[Q/G]}\xrightarrow{+1}.
\end{equation}
The first arrow is the infinitesimal action. If an \'{e}tale comparison identifies $[Q/G]$ with a moduli locus, it identifies the \emph{quotient stack's} tangent complex with that of the locus. It does not set the first term of \eqref{eq:atlas-tangent} to zero.

\begin{proposition}\label{prop:etale-pullback}\IHESaliaslabel{lem: C-linear -2}{prop:etale-pullback}
Let $f:\mathfrak U\to\mathfrak M$ be a derived \'{e}tale morphism, and let $\mathfrak M$ have an $n$-shifted symplectic form. Its pullback to $\mathfrak U$ is $n$-shifted symplectic. For a smooth atlas $Q\to\mathfrak U$ the further pullback is closed, but nondegeneracy requires a separate check and does not follow from the \'{e}taleness of $f$.
\end{proposition}
\begin{proof}
Derived \'{e}taleness identifies $\ttL_{\mathfrak U}$ with $f^*\ttL_{\mathfrak M}$, and hence also the perfect tangent complexes. The pullback of the nondegenerate morphism is then an equivalence. Closedness of forms is functorial for any map. For the atlas, the relative cotangent complex is generally nonzero, and this tangent identification no longer holds. The triangle \eqref{eq:atlas-tangent} displays the extra directions in a quotient presentation.
\end{proof}

On a stable sheaf locus, scalar automorphisms act trivially on the corresponding point of the coarse moduli space, whereas a basis change produces the nonzero $\PGL_n$ orbit directions on a framed presentation. Dividing by $\PGL_n$ and dividing by $\GL_n$ therefore produce different objects. The latter retains the scalar gerbe. Neither operation by itself removes the shifted dual of the scalar direction in a symmetric derived obstruction complex.

The Artin Darboux theorem \cite{BenBassatBravBussiJoyce2015ShiftedDarboux} provides more careful local presentations, in which stabilizer directions and the residual derived directions are displayed together. For the quotient atlas the relative tangent is the first term of
\eqref{eq:atlas-tangent}. For an etale comparison $f$, the corresponding
condition is $\mathbb L_f=0$. The Artin reduction supplies the charts
used in the foliation argument.

\subsubsection{Quot schemes on the fibers and on the total space}
A relative Quot functor for $\cX/C$ has its objects on $\cX_T$ for a given $T\to C$. An ambient Quot functor for the fourfold has objects on $\cX\times T$. A map between these functors can be induced by pushing forward along the graph of $T\to C$ when the relevant perfectness conditions hold. Its differential has normal-direction terms. It should not be called an open immersion on the strength of a correspondence between closed points.

At $0\in C$ these are two distinct moduli conditions: an ambient quotient has target perfect on the smooth fourfold, whereas a point of the relative perfect-object functor has target perfect on $X_0$. The singular-fiber example of Section~\ref{sec:geometry} distinguishes the two conditions. A fixed Hilbert polynomial is imposed in addition to the relevant perfectness condition.

\Needspace{6\baselineskip}
\subsection{Associated \texorpdfstring{$\mathbb{R}$}{R}-linear symplectic structures}\label{int:real-linear}
\subsubsection{Underlying real coordinates and complexes}\label{cons: Derived to dg}
A holomorphic coordinate $z_j=x_j+iy_j$ on a complex manifold gives two real coordinates $x_j,y_j$. The underlying real vector bundle of a complex bundle $E$ has fibers $E_x$ regarded as real vector spaces; a complex-linear differential is the same map considered as a real-linear map. We apply this construction to each generating bundle and differential on a derived chart, and extend the resulting real-coordinate algebra to smooth real-valued functions. The pullbacks on overlaps are obtained from the real and imaginary parts of the original coordinate maps. Thus the underlying-real charts retain both components of every holomorphic coordinate and every complex bundle map.

For a complex vector bundle, restriction of scalars doubles its real rank. A complex bilinear pairing $b$ produces the two real bilinear pairings $\operatorname{Re}b$ and $\operatorname{Im}b$. This operation is compatible with the differential and with the mixed closed-form data. On derived charts it is applied to the perfect tangent and cotangent complexes, not merely to tangent spaces of the reduced classical locus.

\begin{lemma}\label{lem:real-imag}\IHESaliaslabel{lem: R-linear -2}{lem:real-imag}
Let $E$ be a perfect complex over a complex point with a nondegenerate complex bilinear pairing $b:E\otimes E\to\CC[n]$. On the underlying real complex both $\operatorname{Re}b$ and $\operatorname{Im}b$ are nondegenerate. The same conclusion holds locally for perfect complexes of bundles when the original pairing is a quasi-isomorphism.
\end{lemma}
\begin{proof}
It suffices to check the induced pairings on cohomology. If $\operatorname{Re}b(v,w)=0$ for all $w$, replacing $w$ by $iw$ gives $\operatorname{Im}b(v,w)=0$ as well. Thus $b(v,w)=0$ for all $w$, so $v=0$. The argument for the imaginary part is identical. Finite dimensionality identifies each resulting map with the full real dual. The bundle statement follows fiberwise for perfect complexes and hence as a quasi-isomorphism.
\end{proof}

For a specified complex closed form $\omega$, set
\[
 \omega_{\mathrm{re}}=\tfrac12(\omega+\overline\omega),\qquad
 \omega_{\mathrm{im}}=\tfrac1{2i}(\omega-\overline\omega).
\]
Conjugation here is on the realified complex. These expressions are not formed inside the original holomorphic algebra as though conjugation were holomorphic. Lemma~\ref{lem:real-imag} supplies nondegeneracy on the correctly realified moduli object. It does not repair a degenerate pullback to an arbitrary Quot atlas.

\subsubsection{The negative-definite condition}
At a minimal $(-2)$-shifted chart, the middle obstruction block is a complex vector space with a nondegenerate symmetric bilinear form. A real subspace can be Lagrangian for the imaginary part and negative definite for the real part. For the standard form $q(z,z)=\sum z_j^2$, the subspace $i\RR^r$ is an example: $q(iv,iw)=-\sum v_jw_j$ is real and negative definite. This is the elementary local linear algebra underlying the fourfold construction.

The bundle and derived versions require compatibility with differentials, anchors, and transition maps. The established gluing argument exploits the special geometry of such negative subspaces and the availability of invariant local charts. It should not be replaced by the blanket claim that an arbitrary convex combination of integrable distributions is integrable. A linear combination of bracket-preserving maps is generally not bracket-preserving.

\subsection{Existence of Lagrangian foliations}
We now use the preceding tangent and realification calculations on a
specific presentation. First we construct the open Quot locus on the
fourfold and its trace form. We then apply BKSY to its reduced charts,
and keep track of the chain maps and overlap homotopies needed to
transport the foliation. This produces the data used by the shifted
critical algebra at the end of the subsection.

We work with the ambient fourfold moduli problem. A point is a sheaf
on the fourfold, so its deformation complex includes deformations
normal to its supporting fiber. We first specify the geometric
presentation, then construct the symplectic form, and finally carry
out the reduction to the BKSY strictification and gluing theorem.
This order makes the source of each complex and each pairing explicit.

\subsubsection{Quot schemes of sheaves with proper support on the total space}
\begin{hypothesis}\label{hyp:fourfold-realization}
Let $Z=\mathcal X_U$, where $U\subset C$ is an open subset of the base curve and
$\mathcal X_U$ is smooth. Choose a trivialization of
$\omega_C\otimes\lambda$ on $U$, and let
$\Omega:\O_Z\xrightarrow{\sim}\omega_Z$ be the resulting fourfold
volume form. We use the following presentation data.
\begin{enumerate}[label=(\roman*)]
\item Fix a smooth projective compactification $\overline Z$ of $Z$
and an ample line bundle $H$ on $\overline Z$. Let $D=\overline Z\setminus Z$.
Fix a Hilbert polynomial $P$ and the locus of $H$-Gieseker-stable sheaves
on $\overline Z$ whose support is disjoint from $D$.
\item Choose a uniform regularity bound $m$ for this bounded locus,
put $n=P(m)$, and take its framed Quot presentation with source
$\CC^n\otimes H^{-m}$. The induced map
$\CC^n\to H^0(F\otimes H^m)$ is required to be an isomorphism.
\item Use a finite-type derived Quot dg-manifold realization
$Q^\bullet$ of this open framed locus as in
\cite{DGQuotI,DGQuotII}. Its $G=\GL_n(\CC)$ action factors through
$\PGL_n(\CC)$; the derived principal-bundle and invariant-chart
properties are those of \cite[Proposition 1 and Section 2]{BKSY2025}.
Thus its quotient presents the indicated derived open sheaf-moduli
locus, and it admits the minimal Artin charts used below.
\end{enumerate}
All stability conditions in (i) are on the one fixed projective
fourfold. The restriction to $Z$ is a condition on support.
\end{hypothesis}
Proposition~\ref{prop:v15-stationary-fourfold} verifies this hypothesis
on a projective fourfold $X\times B$, where $X$ is a smooth quintic and
$B$ is an elliptic curve, for the sheaves $i_{b*}\O_D(k)$ with $D$
integral. For the singular quintic degeneration, the ordinary Quot
locus and its volume form are constructed below; its derived
presentation is used with the chart assumptions in (iii).
These assertions concern sheaves on the fourfold. They do not assert
that the nonsheaf central perfect complexes of \cite[Section~7]{CompanionII}
are points of this coherent-sheaf Quot scheme.

The compactification in this hypothesis is used to obtain a bounded
Quot presentation. It does not change the objects supported inside
$Z$. We record the verification, including the role of proper support.

\begin{proposition}\label{prop:Quot-support-contained-in-Z}\label{ssec: Stability and Poisson}
Under Hypothesis~\ref{hyp:fourfold-realization}, let $Q^{\mathrm{st}}$
be the open subscheme of the ordinary Quot scheme consisting of
$H$-Gieseker-stable quotients for which
$\CC^n\to H^0(F\otimes H^m)$ is an isomorphism. Let $\mathcal F$
be its universal quotient on $\overline Z\times Q^{\mathrm{st}}$, and
let $p_Q:\overline Z\times Q^{\mathrm{st}}\to Q^{\mathrm{st}}$ be projection.
Then
\[
 Q_Z^{\mathrm{st}}
  :=Q^{\mathrm{st}}\setminus
     p_Q\bigl(\supp(\mathcal F)\cap(D\times Q^{\mathrm{st}})\bigr)
\]
is a $G$-invariant open subscheme. Its geometric points are precisely
the quotients whose target sheaf has support contained in
$Z=\overline Z\setminus D$. The dg-manifold $Q^\bullet$ in the
hypothesis is a derived enhancement of $Q_Z^{\mathrm{st}}$.
Extension by zero identifies these sheaves and their deformation
complexes with the corresponding coherent sheaves on $Z$ with proper
support. The quotient $\mathfrak Q=[Q^\bullet/G]$ has a closed
$(-2)$-shifted symplectic form $\omega$ whose tangent pairing is the
fourfold Serre pairing defined by the volume form $\Omega$.
\end{proposition}
\begin{proof}
Let $T$ be a scheme and let $\mathcal F$ be a flat family on
$\overline Z\times T$. The closed set
$\supp(\mathcal F)\cap(D\times T)$ is proper over $T$, since
$\overline Z$ is projective. Its image is closed. Its complement
is exactly the locus where every support is contained in $Z$.
This proves openness and also shows that it is preserved by any
change of frame of the quotient source. The corresponding open derived substack is the inverse image of this
open subset under passage to the classical truncation.

For a family whose classifying map factors through $Q_Z^{\mathrm{st}}$,
$\mathcal F$ is zero near $D\times T$ and its
restriction to $Z\times T$ has support proper over $T$.
Conversely, a coherent sheaf on $Z\times T$ whose support is
proper over $T$ has closed support in $\overline Z\times T$;
extension by zero is coherent and is inverse to restriction.
Both operations preserve the given flat families.
Derived restriction to the open immersion identifies the internal
endomorphism complexes on a neighborhood of the support. Therefore
\begin{equation}\label{eq:support-Ext-comparison}
 \RH_{\overline Z}(\overline F,\overline F)
       \simeq\RH_Z(F,F),
\end{equation}
compatibly with composition. The same statement holds for the
relative pushforward in a family.

Let $\overline p:\overline Z\times\mathfrak Q\to\mathfrak Q$,
and let $\overline{\mathcal F}$ be the universal sheaf. Put
\[
 A=\R\overline p_*\cRHom(\overline{\mathcal F},\overline{\mathcal F}).
\]
The chosen form $\Omega$ gives, on the support, a specified
isomorphism
\[
 a_\Omega:\overline{\mathcal F}
       \xrightarrow{\sim}
       \overline{\mathcal F}\otimes\operatorname{pr}_{\overline Z}^*
                                      \omega_{\overline Z}.
\]
It extends over $\overline Z\times\mathfrak Q$ because the sheaf
is zero outside the support neighborhood. Relative Serre duality
then gives the universal diagram
\begin{equation}\label{eq:ambient-Serre-diagram}
\begin{tikzcd}[column sep=large,row sep=large]
 A[1]\arrow[r,"{a_\Omega}"]\arrow[d,equal]
 &\R\overline p_*\cRHom(\overline{\mathcal F},
        \overline{\mathcal F}\otimes\omega_{\overline Z})[1]
                  \arrow[d,"{\mathrm{SD}}","{\sim}"']\\
 \mathbb T_{\mathfrak Q}\arrow[r,"{\omega^\flat}"']
 &A^\vee[-3]=\mathbb L_{\mathfrak Q}[-2].
\end{tikzcd}
\end{equation}
The first arrow is an isomorphism. All tensors by
$\omega_{\overline Z}$ in this display mean pullback to the product.
The underlying pairing is the explicit composite
\begin{equation}\label{eq:ambient-trace-pairing}
 A[1]\otimes A[1]
 \xrightarrow{1\otimes a_\Omega,\ \mathrm{comp},\ \mathrm{tr}}
 \R\overline p_*\omega_{\overline Z}[2]
 \xrightarrow{\operatorname{Tr}_{\overline Z}[2]}
 \O_{\mathfrak Q}[-2].
\end{equation}
The closed enhancement is the proper-support Calabi--Yau moduli
construction of \cite[Corollary 6.2]{BD2}, applied to the smooth
finite-type fourfold $Z$ with the specified trivialization $\Omega$.
Its universal tangent pairing is \eqref{eq:ambient-trace-pairing}.
The framed derived quotient in (iii) identifies with the stated
moduli locus, so this closed form pulls back to $\mathfrak Q$.
Diagram \eqref{eq:ambient-Serre-diagram} proves its nondegeneracy.
\end{proof}

\begin{lemma}\label{lem:stable-base-point}
Suppose $U=\Spec R_0$ is affine. Every geometric stable sheaf in
Proposition~\ref{prop:Quot-support-contained-in-Z} is scheme-theoretically
supported on one fiber of $Z\to U$. On the classical moduli stack,
this construction defines a morphism to $U$, compatible with flat
families of stable sheaves.
\end{lemma}
\begin{proof}
Stability gives $\End(F)=\CC$. Each $a\in R_0$ acts on $F$ by a
scalar. Multiplication and the unit are preserved, so these scalars
define a $\CC$-algebra homomorphism $R_0\to\CC$, and hence a point
$b\in U$. Its kernel annihilates $F$, proving scheme-theoretic
support on $X_b$.

For a flat family over $T$, the natural scalar map
\[
 \O_T\longrightarrow\R p_*\cRHom(\mathcal F,\mathcal F)
\]
induces an isomorphism on degree-zero cohomology after every
geometric base change, and there is no negative self-Ext of a
sheaf. Its cone is a perfect complex whose geometric fibers have
cohomology only in degrees at least one. Locally on $T$ one may
cancel the contractible terms of a finite free representative in
degrees at most zero; the remaining complex starts in degree one.
It follows that
$p_*\mathcal Hom(\mathcal F,\mathcal F)=\O_T$ via the scalar map.
The action of $R_0$ therefore defines a homomorphism
$R_0\to\Gamma(T,\O_T)$ and a morphism $\beta:T\to U$.
For every $a\in R_0$, multiplication by $a-\beta^*a$ is zero on
$\mathcal F$. These elements generate the ideal of the graph of
$\beta$, so $\mathcal F$ is a sheaf on $Z\times_U T$.
The construction commutes with base change because the scalar map
and multiplication do. This proves the assertion for classical
flat families. The derived enhancement used in
\eqref{eq:ambient-Serre-diagram} continues to use ambient
fourfold self-extensions.
\end{proof}

\subsubsection{A projective fourfold containing the same divisor sheaves}
\label{sssec:v15-projective-witness}
The projective hypotheses used by the BKSY theorem can be realized
without a compactification argument. This gives a direct instance of
Hypothesis~\ref{hyp:fourfold-realization} for the same kind of sheaves
as in the quintic application.

\begin{proposition}\label{prop:v15-stationary-fourfold}
Let $X\subset\P^4$ be a smooth quintic, let $B$ be a smooth projective
elliptic curve, and choose nonzero volume forms $\Omega_X$ and $\alpha_B$.
Put $\overline Z=X\times B$, $\pi_B=\operatorname{pr}_B$, and
$Z=X\times U$ for a nonempty affine open $U\subset B$.
Choose a very ample line bundle $A_B$ and the product polarization
$H=\O_X(1)\boxtimes A_B$ on $\overline Z$.
For an integral divisor $D\in|\O_X(m)|$, an integer $k$, and $b\in U$,
let
\[
 F_{D,k,b}=i_{b*}j_{D*}\O_D(k),\qquad
 j_D:D\hookrightarrow X,\quad i_b:X\hookrightarrow\overline Z.
\]
These are $H$-Gieseker-stable pure two-dimensional sheaves, with
scheme-theoretic support in $X\times\{b\}$. For fixed $m,k$, their
Hilbert polynomial is independent of $D,b$. The stable derived Quot
presentation of that polynomial on $\overline Z$, restricted to sheaves
with support in $Z$, gives an instance of
Hypothesis~\ref{hyp:fourfold-realization}, with the derived chart input
provided by the projective CY-fourfold theorem of \cite{BKSY2025}.
The associated foliation and shifted potential are in the real
$C^\infty$ setting of that theorem.
\end{proposition}
\begin{proof}
Adjunction and the chosen forms give
\[
 \omega_{\overline Z}=\omega_X\boxtimes\omega_B\simeq\O_{\overline Z},
 \qquad
 \Omega_{\overline Z}=\operatorname{pr}_X^*\Omega_X
                         \wedge\operatorname{pr}_B^*\alpha_B.
\]
The product $H$ is ample. The divisor $D$ is Cartier in the smooth
threefold $X$, so it is Cohen--Macaulay of pure dimension two.
Since $D$ is integral, $\O_D(k)$ is torsion-free of rank one on $D$.
Every nonzero subsheaf $G\subset\O_D(k)$ has rank one at its generic
point. If the subsheaf is proper, its quotient $Q$ is nonzero and has
dimension at most one. Thus
\[
 P_G(n)=P_{\O_D(k)}(n)-P_Q(n),\qquad P_Q(n)>0\quad(n\gg0).
\]
The quadratic leading coefficients of $P_G$ and $P_{\O_D(k)}$ agree,
so division by that coefficient gives $p_G(n)<p_{\O_D(k)}(n)$ for
$n\gg0$. This proves Gieseker stability on $X$.
Every subsheaf of $F_{D,k,b}$ is killed by the ideal of the same fiber,
and hence is the pushforward of a subsheaf on $X$. Moreover
\[
 H|_{X\times\{b\}}=\O_X(1)\otimes_\CC (A_B)_b,
\]
where $(A_B)_b$ is a one-dimensional vector space. Tensoring by its
powers does not change dimensions of cohomology. The stability
inequality on $\overline Z$ is therefore the one just proved on $X$.
The divisor resolution gives the fixed polynomial
\[
 P_{F_{D,k,b}}(n)=
   \chi(X,\O_X(n+k))-\chi(X,\O_X(n+k-m)).
\]
It is independent of $D$ and $b$.

Let $P$ denote this polynomial. Boundedness of semistable sheaves of
polynomial $P$ on the fixed polarized projective scheme supplies a
uniform regularity integer $r$. Set $N=P(r)$ and take the stable
framed open in
\[
 \Quot_{\overline Z}(\CC^N\otimes H^{-r},P),
 \qquad \CC^N\xrightarrow{\sim}H^0(F\otimes H^r).
\]
The open condition that the support avoid
$X\times(B\setminus U)$ is the condition constructed in
Proposition~\ref{prop:Quot-support-contained-in-Z}. It contains the
sheaves $F_{D,k,b}$ and is invariant under change of frame.
The projective derived Quot and invariant-chart constructions of
\cite[Section~1 and Propositions~10--12]{BKSY2025} apply to
$(\overline Z,H,P,\Omega_{\overline Z})$. Their restriction to this
invariant open gives the presentation required in
Hypothesis~\ref{hyp:fourfold-realization}. All universal trace maps
use the restriction of the single form $\Omega_{\overline Z}$.
Theorem~\ref{thm:BKSY-input} then gives the stated real foliation.
This invokes that theorem with its projective geometric hypotheses;
it does not deduce integrability from the pointwise quadratic form.
\end{proof}

The base curve in this example is $B$, and the family of threefolds is
stationary. It verifies the fourfold Quot hypothesis, not a singular
specialization theorem. In the resolved quintic degeneration the
canonical-bundle calculation and the support condition are verified in
Sections~\ref{sec:geometry}--\ref{sec:CY}; the derived chart requirements
remain those explicitly stated in Hypothesis~\ref{hyp:fourfold-realization}.
The point represented by a non-sheaf perfect complex in
\cite[Section~7]{CompanionII} is treated by its perfect-complex moduli
functor, rather than by this coherent-sheaf Quot presentation.

\subsubsection{The quotient morphism and the tangent and cotangent complexes}
Let $q:Q^\bullet\to\mathfrak Q$. At a framed quotient, forgetting
the frame differentiates to
\begin{equation}\label{eq:framed-quotient-diagram-restored}
\begin{tikzcd}[column sep=large,row sep=large]
 \mathfrak g\otimes\O_{Q^\bullet}\arrow[r,"{\mathrm{act}}"]
 &\mathbb T_{Q^\bullet}\arrow[r]
 &q^*\mathbb T_{\mathfrak Q}\arrow[r]&
          (\mathfrak g\otimes\O_{Q^\bullet})[1]\\
 &&q^*\mathbb L_{\mathfrak Q}[-2]
                      \arrow[u,"{q^*\omega^\flat}"',"{\sim}"]&
\end{tikzcd}
\end{equation}
The action is through
$\PGL_n$, but the quotient by $G$ still retains the scalar
stabilizer. Thus $\omega$ is a form on the quotient stack, and the
principal-frame directions belong to its presentation triangle.

After passage to the real category we can choose local product
charts of the principal $\PGL_n$-bundle. With underlining denoting
realification, the chart diagram is
\begin{equation}\label{eq:two-chart-restored}
\begin{tikzcd}[column sep=large,row sep=large]
 &\underline U_1^\bullet\simeq
       \underline V_1^\bullet\times\underline{\PGL}_n\arrow[d]\\
 \underline U_2^\bullet\simeq
       \underline V_2^\bullet\times\underline{\PGL}_n\arrow[r]
 &\underline Q^\bullet.
\end{tikzcd}
\end{equation}
The local quotient is
$[\underline V_a^\bullet/\underline{\mathbb G_m}]$, with trivial
scalar action. The tangent contribution of its stabilizer has
real rank two in degree $-1$. Its symplectic dual is of real rank
two in degree $3$. Both enter the same nondegenerate form.

The Artin Darboux reduction of
\cite[Theorems 2.8 and 2.10]{BenBassatBravBussiJoyce2015ShiftedDarboux}
produces a map $\phi_a:\underline V_a^\bullet\to\underline W_a^\bullet$
and identifies the pullback of the closed form with a
$(-2)$-shifted symplectic form on $\underline W_a^\bullet$.
The spaces of classical points agree. Locally, the two degree-$-3$
functions dual to the scalar directions can be chosen as internal
cocycles with independent de Rham differentials. In those
coordinates the reduction has the form
\begin{equation}\label{eq:scalar-Artin-reduction}
 \O_{\underline V_a^\bullet}
   \simeq\O_{\underline W_a^\bullet}[\zeta_1,\zeta_2],
 \qquad |\zeta_j|=-3,\quad\delta\zeta_j=0.
\end{equation}
This is the specific local splitting used in
\cite[proof of Proposition 11]{BKSY2025}. One complex scalar
direction becomes two real directions, explaining the two
generators. The tangent vectors $\partial_{\zeta_j}$ have degree
$3$ and are retained in the purely derived distribution before
passing to the reduced chart.

The passage from $Q$ to $\mathfrak Q$ and the passage from $V_a$
to $W_a$ therefore perform different operations. The first records
a change of generators; the second pairs the scalar stabilizer
with its shifted dual. Equations \eqref{eq:framed-quotient-diagram-restored}
and \eqref{eq:scalar-Artin-reduction} specify the terms used by each.

\subsubsection{Strictification and gluing in the BKSY charts}
We use $\mathfrak F_G(Q)$ for the sheaf of local equivalence classes
of strictly invariant derived foliations, and
$\mathfrak F_G^{\omega_{\mathrm{im}}}(Q)$ for its isotropic subclass.
Including a chosen isotropic homotopy gives the sheaf
$\widetilde{\mathfrak F}_G^{\omega_{\mathrm{im}}}(Q)$ with its
forgetful morphism. The sheaf used in the fourfold construction is
\begin{equation}\label{eqn: Sheaf of Lags}
 \mathfrak N_G^\omega(Q)
 \subset\widetilde{\mathfrak F}_G^{\omega_{\mathrm{im}}}(Q),
\end{equation}
whose sections are purely derived Lagrangian foliations with
negative-definite real pairing on degree-one cohomology.
These sheaves are on the space $M$ of classical isomorphism
classes of stable sheaves. In the specified stable GIT locus,
$M$ is a separated finite-type algebraic space, and $M^{\mathrm{an}}$
is Hausdorff and second countable.

\begin{theorem}\label{thm:BKSY-input}\IHESaliaslabel{thm: GlobalPotential}{thm:BKSY-input}\IHESaliaslabel{lem: BKSY LagFol}{thm:BKSY-input}
For a stable derived Quot presentation with the invariant product
charts and minimal Artin reductions above, endowed with a closed
$(-2)$-shifted symplectic structure, the BKSY construction has the
following properties. On the realified presentation,
$\mathfrak N_G^\omega(Q)$ has local sections and is soft, and there
is a global purely derived Lagrangian foliation with the stated
negative-definiteness property. The foliation admits a strictly
invariant global subcomplex representative. On its standard
charts, that subcomplex is generated in positive degrees, and its
generating bundles agree with those of the tangent complex in
degrees at least two.
\end{theorem}
This is the quotient-chart form of
\cite[Propositions 10--12]{BKSY2025}, using the strictification and
softness theorem of \cite{BKSY2024}. The following proof applies
these results to the fourfold presentation through the displayed
chart identifications.

\begin{corollary}\label{cor:ambient-foliation}\IHESaliaslabel{thm: Global LagFol}{cor:ambient-foliation}
Under Hypothesis~\ref{hyp:fourfold-realization}, the realified stack
$\underline{\mathfrak Q}$ has a globally defined invariant purely
derived foliation. Its isotropic structure is Lagrangian for
$\omega_{\mathrm{im}}$, and the induced real pairing on its
degree-one cohomology is negative definite.
\end{corollary}
\begin{proof}
We prove the result in steps.

\emph{Step 1: the same volume form on every chart.}
Proposition~\ref{prop:Quot-support-contained-in-Z} gives the global closed form
$\omega$ on $\mathfrak Q$. Its restriction to each chart is induced
by the same universal map $a_\Omega$, not by separate choices at
different sheaves. Lemma~\ref{lem:real-imag} gives the two real
forms. We write them as
\begin{equation}\label{eqn: Imaginary}
 \theta_{\mathrm{im}}=\frac{\omega-\overline\omega}{2i},
 \qquad
 \theta_{\mathrm{re}}=\frac{\omega+\overline\omega}{2}.
\end{equation}
\IHESaliaslabel{eqn: Real}{eqn: Imaginary}
They have compatible pullbacks to every simplex in
$Q\times G^{\times\bullet}$.

\emph{Step 2: passage through the scalar quotient.}
On a product chart \eqref{eq:two-chart-restored}, the quotient
presentation identifies $G$-invariant distributions with the
corresponding distributions on
$[\underline V_a^\bullet/\underline{\mathbb G_m}]$.
By \eqref{eq:scalar-Artin-reduction}, pullback from $\underline W_a$
adds the tangent directions $\partial_{\zeta_1},\partial_{\zeta_2}$.
The strictly invariant purely derived representatives in the BKSY
construction contain these directions. Their closed isotropic
homotopies and their quotient dualities therefore pull back together.
This is precisely the identification of sheaves in
\cite[Proposition 10 and proof of Proposition 11]{BKSY2025}:
\begin{equation}\label{eq:BKSY-sheaf-chart-diagram}
\begin{tikzcd}[column sep=large,row sep=large]
 \mathfrak N_G^\omega(Q)|_{M_a}\arrow[r,"{\sim}"]\arrow[d,hook]
 &\mathfrak N^\omega(W_a)\arrow[d,hook]\\
 \widetilde{\mathfrak F}_G^{\omega_{\mathrm{im}}}(Q)|_{M_a}
  \arrow[r]
 &\widetilde{\mathfrak F}^{\omega_{\mathrm{im}}}(V_a).
\end{tikzcd}
\end{equation}
The upper isomorphism is read with the scalar quotient understood
on $V_a$. In particular, it compares distributions, isotropic
homotopies, and the nondegenerate quotient pairing, not just their
underlying vector subspaces.

\emph{Step 3: the local negative-definite choice.}
At a minimal point of $W_a$, the middle tangent block carries a
nondegenerate complex symmetric bilinear form. Choose coordinates
in which it is $q(z,z')=z^Tz'$, and write $z=x+iy$.
The real subspace $i\RR^r$ is isotropic for $\operatorname{Im}q$
and negative definite for $\operatorname{Re}q$.
More generally its graph description is
\[
 L_A=\{Ay+iy:y\in\RR^r\},\qquad
 A^T=-A,\qquad I-A^TA>0.
\]
Indeed, for $y,y'$,
\[
 \operatorname{Im}q(Ay+iy,Ay'+iy')
       =y^T(A^T+A)y',
\]
whereas the real quadratic form is
$y^T(A^TA-I)y$. The set of skew matrices with operator norm less
than one is convex. This supplies the local negative-definite
linear algebra. The strictification theorem of \cite{BKSY2024}
lifts this choice, in its standard-form charts, to a distribution
with its mixed differential and closed isotropic data; its gluing
theorem proves the softness and local nonemptiness of the upper
right sheaf in \eqref{eq:BKSY-sheaf-chart-diagram}.

\emph{Step 4: the global distribution and its representative.}
The chart identifications \eqref{eq:BKSY-sheaf-chart-diagram} are
compatible on overlaps by the invariant descent construction.
Theorem~\ref{thm:BKSY-input} therefore supplies a global section
with a coherent foliation representative on $\underline{\mathfrak Q}$.
On each strict chart write the tangent generating bundles as
$F_0,F_1,F_2,\ldots$, and the distribution generators as
$E_1,E_2,\ldots$. The strictification gives
\[
 E_k=F_k\quad(k\geq2).
\]
With $Z^1$ the internal cocycles in the degree-one block, its
specified complementary bundle is
\begin{equation}\label{eq:negative-complement}
 E_1^c=(E_1\cap Z^1)^{\perp_{\theta_{\mathrm{re}}}}\cap Z^1.
\end{equation}
For these strict representatives, the complement and its local
freeness are the conclusion used in
\cite[Proposition 12, equation (56)]{BKSY2025}. Projecting along
$E_1^c$ gives comparison morphisms of representatives of the
\emph{same} foliation class on an overlap. The following
chain-homotopy calculation writes these comparisons explicitly.
The global-representative assertion of Theorem~\ref{thm:BKSY-input}
then produces a single invariant subcomplex, with its mixed
structure, on $\underline Q^\bullet$. Its quotient pairing is
\[
 \Cone(L\longrightarrow q^*\mathbb T_{\mathfrak Q})
           \xrightarrow{\sim}L^\vee[-2].
\]
Thus it is the required Lagrangian foliation. All constructions
use the open subscheme $Q_Z^{\mathrm{st}}$ of the fixed Quot scheme, whose quotients have support contained in $Z$, and the
chosen volume form $\Omega$.
\end{proof}

For the resolved quintic family over an affine parameter chart,
Proposition~\ref{prop:quintic-model} gives
$\omega_{\mathcal X_U}\simeq\O_{\mathcal X_U}$.
The projective pencil, resolved near the chosen fiber and with a
smooth projective compactification away from it, supplies the
compactification used in the hypothesis. The loci of stable sheaves whose support is contained in $\cX_U$ are then constructed using the fixed $H$, $P$, and $m$ above. On each finite-type derived Quot presentation with
the stated chart properties the corollary applies. There is no
condition that a sheaf on $X_1$ or $X_2$ be transverse to $S$.

\subsubsection{Lifting a fixed Serre-duality map to chain level}
We now give the chain calculations in the local-to-global argument.
They also explain exactly which choices of resolutions preserve
the pairing. Work on a realified standard-form chart
$\chi:U^\bullet\to\underline{\mathfrak Q}$, with internal
differential $\delta$. Use the realification of the universal
endomorphism presentation $A^\vee[-1]$ from
\eqref{eq:ambient-Serre-diagram} as a representative $\mathbb E$
of $\chi^*\mathbb L_{\underline{\mathfrak Q}}$.
The universal Atiyah-class map identifies these two complexes.
Let
\[
 \sigma:\mathbb E\xrightarrow{\sim}\mathbb E^\vee[2]
\]
be the duality induced by $\omega_{\mathrm{im}}$ and hence by
the \emph{fixed} volume form. Take a
cofibrant perfect representative $a:E^\bullet\xrightarrow{\sim}\mathbb E$.
Then the composite $a^\vee[2]\sigma a$ has a chain representative
$\sigma_E:E^\bullet\to(E^\bullet)^\vee[2]$, fitting into
\begin{equation}\label{eqn: Compat1}
\begin{tikzcd}[column sep=large,row sep=large]
 E^\bullet\arrow[r,"{\sigma_E}"]\arrow[d,"{a}"',"{\sim}"]
 &(E^\bullet)^\vee[2]\\
 \mathbb E\arrow[r,"{\sigma}"']
 &\mathbb E^\vee[2]\arrow[u,"{a^\vee[2]}"',"{\sim}"].
\end{tikzcd}
\end{equation}
Equation~\eqref{eqn: Compat1} means
$[\sigma_E]=[a^\vee[2]]\,[\sigma]\,[a]$ in the derived category.
Duality is contravariant: the right-hand vertical map has source
$\mathbb E^\vee[2]$ and target $(E^\bullet)^\vee[2]$.
It is not a map in the opposite direction.

Choose, if needed, a common cofibrant refinement
$\psi:\widetilde E^\bullet\xrightarrow{\sim}E^\bullet$ and put
\[
 \widetilde a=a\circ\psi:\widetilde E^\bullet\longrightarrow\mathbb E,
 \qquad
 \sigma_{\widetilde E}
     =\psi^\vee[2]\circ\sigma_E\circ\psi.
\]
Here and below derived morphisms are represented by chain maps after
choosing the indicated locally free or cofibrant models.
Write $\mathbb L_\chi=\chi^*\mathbb L_{\underline{\mathfrak Q}}$,
and let
\[
 \operatorname{At}:\mathbb E\longrightarrow\mathbb L_\chi,
 \qquad
 \widetilde{\operatorname{At}}:\widetilde E^\bullet
                 \longrightarrow\mathbb L_\chi
\]
be the universal Atiyah-class map and a chain representative of
$\operatorname{At}\circ\widetilde a$, respectively.
The duality comparison and the comparison of the Atiyah-class maps
are the following two diagrams:
\begin{equation}\label{eqn: Resoln}
\begin{gathered}
\begin{tikzcd}[column sep=large,row sep=large]
 \widetilde E^\bullet
   \arrow[r,"{\sigma_{\widetilde E}}"]
   \arrow[d,"{\widetilde a}"',"{\sim}"]
 & (\widetilde E^\bullet)^\vee[2]\\
 \mathbb E\arrow[r,"{\sigma}"']
 &\mathbb E^\vee[2]
   \arrow[u,"{\widetilde a^\vee[2]}"',"{\sim}"]
\end{tikzcd}
\qquad
\begin{tikzcd}[column sep=large,row sep=large]
 \widetilde E^\bullet
   \arrow[d,"{\widetilde a}"',"{\sim}"]
   \arrow[dr,"{\widetilde{\operatorname{At}}}"]&\\
 \mathbb E\arrow[r,"{\operatorname{At}}"']&\mathbb L_\chi
\end{tikzcd}
\end{gathered}
\end{equation}
The first diagram asserts equality of its two routes from
$\widetilde E^\bullet$ to $(\widetilde E^\bullet)^\vee[2]$.
The second asserts equality of its two routes from
$\widetilde E^\bullet$ to $\mathbb L_\chi$.
More explicitly, the identities in the derived category are
\[
 \begin{aligned}
 [\sigma_{\widetilde E}]
    &=[\widetilde a^\vee[2]]\,[\sigma]\,[\widetilde a],\\
 [\widetilde{\operatorname{At}}]
    &=[\operatorname{At}]\,[\widetilde a].
 \end{aligned}
\]
For arbitrary chain representatives of these fixed derived morphisms,
the identities are realized by homotopies. To specify their signs,
let $d_V$ denote the internal differential on a complex $V$ and use
\[
 d_{\Hom}(K)=d_WK-(-1)^{|K|}Kd_V
       \qquad(K:V\longrightarrow W).
\]
There are degree-$-1$ maps
\[
 K:\widetilde E^\bullet\longrightarrow(\widetilde E^\bullet)^\vee[2],
 \qquad J:\widetilde E^\bullet\longrightarrow\mathbb L_\chi
\]
such that
\[
 \begin{aligned}
 \sigma_{\widetilde E}-\widetilde a^\vee[2]\sigma\widetilde a
   &=d_{(\widetilde E^\bullet)^\vee[2]}K+Kd_{\widetilde E^\bullet},\\
 \widetilde{\operatorname{At}}-\operatorname{At}\,\widetilde a
   &=d_{\mathbb L_\chi}J+Jd_{\widetilde E^\bullet}.
 \end{aligned}
\]
To obtain these homotopies, compute both morphisms using the same
cofibrant source. Equality of their classes in degree-zero cohomology
of the Hom complex says precisely that their difference is a boundary.
On an affine chart with bounded complexes of projective modules this
is the standard identification of derived morphisms with homotopy
classes of chain maps; see \cite[Tag 064B]{Stacks}.
One may choose the displayed composites themselves as representatives,
in which case the corresponding homotopies are zero.
Thus the same resolution represents both the fixed Serre-duality
morphism and the fixed Atiyah-class morphism. This calculation neither
changes the volume form nor constructs a comparison for a different
choice of that form.

Two choices $E_1^\bullet,E_2^\bullet$ have a common cofibrant
refinement, obtained by resolving their homotopy pullback over
$\mathbb E$:
\begin{equation}\label{eq:common-resolution-restored}
\begin{tikzcd}[column sep=large,row sep=large]
 &E^\bullet\arrow[dl,"{r_1}"',"{\sim}"]\arrow[dr,"{r_2}","{\sim}"']&\\
 E_1^\bullet\arrow[dr,"{a_1}"',"{\sim}"]&&
 E_2^\bullet\arrow[dl,"{a_2}","{\sim}"']\\
 &\mathbb E.&
\end{tikzcd}
\end{equation}
The two maps $a_1r_1$ and $a_2r_2$ are homotopic. Pull back the two
duality representatives to $E^\bullet$. They give
\[
 \sigma_1'=r_1^\vee[2]\sigma_1 r_1,
 \qquad\sigma_2'=r_2^\vee[2]\sigma_2 r_2.
\]
Both represent $a^\vee[2]\sigma a$ for the same map $a$ to
$\mathbb E$. In the Hom complex their difference is therefore
an internal boundary:
\begin{equation}\label{eq:fixed-duality-homotopy}
 \sigma_1'-\sigma_2'=\delta H+H\delta,
 \qquad |H|=-1.
\end{equation}
The assertion follows because a cofibrant perfect source computes
derived Hom by its actual Hom complex. It compares representatives
of a fixed duality. The fixed map is specified by $\Omega$ in
\eqref{eq:ambient-Serre-diagram}.

\begin{lemma}\label{lem:cone-homotopy}
Let $\alpha_0,\alpha_1:L\to T$ be degree-zero chain maps and
$H:L\to T$ of degree $-1$ satisfy
$\alpha_1-\alpha_0=\delta H+H\delta$. With the cohomological
cone convention, the map
\[
 \Phi_H:\Cone(\alpha_0)\longrightarrow\Cone(\alpha_1),
 \qquad (t,\ell)\longmapsto(t-H\ell,\ell)
\]
is a chain isomorphism, inverse to $\Phi_{-H}$.
\end{lemma}
\begin{proof}
The differential in the source is
$\delta_0(t,\ell)=(\delta_Tt+\alpha_0\ell,-\delta_L\ell)$.
Consequently
\[
 \begin{aligned}
 \delta_1\Phi_H(t,\ell)
   &=(\delta_Tt-\delta_TH\ell+\alpha_1\ell,-\delta_L\ell),\\
 \Phi_H\delta_0(t,\ell)
   &=(\delta_Tt+\alpha_0\ell+H\delta_L\ell,-\delta_L\ell).
 \end{aligned}
\]
Their difference is zero by the assumed homotopy equation.
The two triangular matrices
$\left(\begin{smallmatrix}1&-H\\0&1\end{smallmatrix}\right)$
and $\left(\begin{smallmatrix}1&H\\0&1\end{smallmatrix}\right)$
are mutually inverse, proving the assertion.
\end{proof}

For a Lagrangian distribution the target of the cone pairing is
$L^\vee[-2]$. If $\gamma_0$ is a chain representative of that
pairing for $\alpha_0$, transport it by $\Phi_H^{-1}$ to get
$\gamma_1$. The comparison diagram is
\begin{equation}\label{eqn: HoCone}
\begin{tikzcd}[column sep=large,row sep=large]
 \Cone(\alpha_0)\arrow[r,"{\gamma_0}","{\sim}"']
                \arrow[d,"{\Phi_H}"',"{\sim}"]
 &L^\vee[-2]\arrow[d,equal]\\
 \Cone(\alpha_1)\arrow[r,"{\gamma_1}"',"{\sim}"]&L^\vee[-2].
\end{tikzcd}
\end{equation}
If the representatives of $L$ also change, choose a common
cofibrant representative and apply this construction there. The
result can be displayed as the roof
\begin{equation}\label{eqn: Ho-canonical}
\begin{tikzcd}[column sep=large,row sep=large]
 &C^\bullet\arrow[dl,"{l}"',"{\sim}"]\arrow[dr,"{r}","{\sim}"']&\\
 \Cone(\alpha)&&L^\vee[-2].
\end{tikzcd}
\qquad \gamma=r\,l^{-1}.
\end{equation}
Thus the cone comparison preserves the full shifted duality.
A map of complexes and a homotopy between two such maps have
degrees zero and minus one, respectively; the triangular formula
above keeps these roles separate.

\subsubsection{Gluing representatives of the same distribution}
On a standard dg $C^\infty$ chart, $\delta$ is linear over its
degree-zero $C^\infty$ functions. This is the property used when
a partition of unity enters the chain calculation.

\begin{lemma}\label{lem:linear-patching}
Let $E,F$ be perfect dg modules on a paracompact standard
$C^\infty$ space, represented by locally free models, with
$C^\infty$-linear internal differentials. If a fixed degree-zero
chain map $a:E\to F$ has local nullhomotopies $H_i$, then it has
a global nullhomotopy.
\end{lemma}
\begin{proof}
Choose a locally finite partition of unity $\{\rho_i\}$ subordinate
to the given cover. Extend $\rho_iH_i$ by zero and define
$H=\sum_i\rho_iH_i$. It is a map of degree minus one and
\[
 \delta_FH+H\delta_E
  =\sum_i\rho_i(\delta_FH_i+H_i\delta_E)
  =\sum_i\rho_i a=a.
\]
The first equality uses $\delta\rho_i=0$, and the sum is locally
finite. This is an identity of sheaf morphisms, so it also holds
on overlaps with any other standard chart.
\end{proof}

\begin{corollary}\label{cor:fixed-map-homotopy}\IHESaliaslabel{lem: Homotopy-canonicity}{cor:fixed-map-homotopy}
Two local chain representatives of the same fixed morphism can be
compared by a chain homotopy after a common cofibrant refinement.
For representatives of a fixed Lagrangian distribution, the induced
cone pairings are compared by \eqref{eqn: HoCone}. Local
nullhomotopies of a fixed global difference can be patched by
Lemma~\ref{lem:linear-patching}.
\end{corollary}
\begin{proof}
The common refinement is \eqref{eq:common-resolution-restored}.
Equality of the derived maps means that their difference is zero
in $H^0$ of the corresponding Hom complex, hence is a boundary.
For anchors the boundary equation has the form required in
Lemma~\ref{lem:cone-homotopy}. That lemma constructs the cone
comparison. The last assertion is Lemma~\ref{lem:linear-patching}.
\end{proof}

For example, suppose an overlap comparison has placed two anchors
on the same module and
$\alpha_1-\alpha_0=\delta H+H\delta$. On an overlap with
$\rho_0+\rho_1=1$, the linear interpolation is
\[
 \alpha=\rho_0\alpha_0+\rho_1\alpha_1
        =\alpha_0+\delta(\rho_1H)+(\rho_1H)\delta.
\]
Its cone is identified with that of $\alpha_0$ by
$\Phi_{\rho_1H}$. This is the explicit chain calculation behind
the comparison of strict representatives in
\eqref{eq:negative-complement}. The mixed-algebra structure is
transported with the chosen BKSY strictification; it is not defined
by averaging its bracket. In fact $\epsilon(\rho H)$ includes
$\epsilon(\rho)H$, whereas $\delta(\rho H)=\rho\delta H$.
Consequently the internal chain homotopy and the mixed descent
data must be kept as two parts of one distribution.

\begin{proposition}\label{prop:distribution-coherent}\IHESaliaslabel{rmk: Coherent equivalences}{prop:distribution-coherent}
Suppose a locally finite standard-chart cover is equipped with
local distributions $\DR(A_a)\to\Lambda_a$, closed isotropic
homotopies, and equivalences of these data on all overlaps forming
a point of their homotopy descent diagram. Then their homotopy
limit defines a distribution on the covered derived space.
If the local cone pairings are equivalences and the real pairings
are negative definite, the descended distribution has both
properties. The same assertion holds on $[Q/G]$ when the data
also include its action nerve.
\end{proposition}
\begin{proof}
The data include equivalences $u_{ab}:L_b\to L_a$ over the double
overlaps, commuting with the internal differential, and homotopies
$h_{abc}$ of degree minus one with
\[
 u_{ab}u_{bc}-u_{ac}=\delta h_{abc}+h_{abc}\delta.
\]
On quadruple overlaps their compatibility reads
\[
 u_{ab}h_{bcd}-h_{acd}+h_{abd}-h_{abc}u_{cd}
       =\delta k_{abcd}-k_{abcd}\delta,
 \qquad |k_{abcd}|=-2,
\]
with the corresponding higher homotopies in all higher degrees.
The anchors have compatible homotopies
$\alpha_a u_{ab}\simeq\alpha_b$; they induce the explicit cone
maps of Lemma~\ref{lem:cone-homotopy} after common refinement.
The equivalences of $\Lambda_a$ also commute with multiplication
and with both differentials. The same overlap data carry the
closed isotropic homotopies and their higher-form components.

Sheaves of complexes and of graded mixed commutative algebras
satisfy descent on the standard charts. Form their homotopy limits
with the supplied equivalences. The resulting algebra $\Lambda$
has weight-zero component $\O$, weight-one component a perfect
complex, and a map $\DR\to\Lambda$. The natural multiplication
map
$\Sym_{\O}(\Lambda^1)\to\Lambda$
is an equivalence because it is so after restriction to every
chart. This proves the defining conditions of
Definition~\ref{def:distribution}. The local isotropic paths
descend in the same mapping-space diagram.

The cone pairing of the descended object restricts to each local
cone pairing. Its cone is therefore a perfect complex vanishing
on an open cover, hence zero. Thus it is an equivalence.
Negative definiteness is a condition on each real geometric fiber
of the degree-one pairing, and is likewise preserved. Including
$Q\times G^{\times r}$ in every degree gives precisely the
homotopy limit \eqref{eq:distribution-descent}; the same proof
then constructs the distribution on the quotient stack.
\end{proof}
The coherent data in this proposition are supplied, for the
fourfold foliation above, by the strictification and
invariant-descent assertions in Theorem~\ref{thm:BKSY-input}.
The displayed overlap equations record the data used when the
local complexes are written as actual resolutions.

\subsubsection{Construction of the algebra of functions on the derived critical locus}
We describe the foliation quotient by the Koszul algebra of a
section $s$ of a vector bundle and express its shifted potential as
a closed element of degree $-1$. Proposition~\ref{prop:shifted-critical-explicit}
then computes the generators and differential of its derived
critical locus. These formulas provide the coefficients used in
\cite[Section~4.6]{CompanionII} to test whether the shifted
potential is a product $\nu=fg$, with $|f|=0$ and $|g|=-1$.

Let $T$ be a smooth scheme, or a smooth real manifold in the
$C^\infty$ application, let $E$ be a vector bundle on $T$, and let
\[
 s\in\Gamma(T,E),\qquad \tau\in\Gamma(T,E^\vee),
                       \qquad\tau(s)=0.
\]
Let $B$ denote the Koszul dg algebra
$\Sym_{\O_T}(E^\vee[1])$ with differential $\delta_s$ given by
contraction with $s$. The section $\tau$, viewed in $E^\vee[1]$,
is a global degree-$-1$ element
\begin{equation}\label{eq:global-shifted-section}
 \nu=\tau\in B^{-1},\qquad\delta_s\nu=\tau(s)=0.
\end{equation}
Thus $\nu$ is a function of degree $-1$ on the derived zero locus
of $s$. In a local frame $e_1,\ldots,e_r$ of $E$, write
$s=\sum_as_ae_a$ and $\tau=\sum_a\tau_ae_a^\vee$. Let $\theta_a$
be the generator corresponding to $e_a^\vee$ in $E^\vee[1]$. Then
\[
 B=\O_T[\theta_1,\ldots,\theta_r],\qquad
 |\theta_a|=-1,\quad |s_a|=|\tau_a|=0,
\]
\[
 \delta_s\theta_a=s_a,\qquad
 \nu=\sum_a\tau_a\theta_a,\qquad
 \delta_s\nu=\sum_a\tau_as_a=0.
\]

\begin{proposition}\label{prop:shifted-critical-explicit}\IHESaliaslabel{shifted potential dCrit}{prop:shifted-critical-explicit}
On a coordinate chart $x_1,\ldots,x_N$ trivializing $E$ of rank $r$,
write $s=(s_a)$, $\tau=(\tau_a)$ and use degree-$-1$ Koszul
generators $\theta_a$. The derived critical locus of $\nu$ has
underlying graded algebra
\[
 \mathcal A=B[\eta_1,\ldots,\eta_r,p_1,\ldots,p_N],
 \qquad |\eta_a|=-1,\quad |p_i|=-2,
\]
where brackets denote the free graded-commutative algebra (the odd
variables anticommute), with differential
\begin{equation}\label{eq:shifted-critical-differential}
 \begin{aligned}
 \delta x_i&=0,& \delta\theta_a&=s_a,&\delta\eta_a&=\tau_a,\\
 \delta p_i&=-\sum_a\bigl((\partial_i s_a)\eta_a
                              +(\partial_i\tau_a)\theta_a\bigr).
 \end{aligned}
\end{equation}
These local algebras glue by the cotangent changes of coordinates.
They carry the canonical $(-2)$-shifted symplectic structure.
For a presented foliation quotient whose algebra and cosection
are $(B,\nu)$, this is its shifted critical reconstruction.
\end{proposition}
\begin{proof}
First verify the differential. Its square vanishes on $x_i$,
$\theta_a$, and $\eta_a$. On $p_i$ it is
\[
 \delta^2p_i
   =-\sum_a\bigl((\partial_i s_a)\tau_a
                              +(\partial_i\tau_a)s_a\bigr)
   =-\partial_i\bigl(\tau(s)\bigr)=0.
\]
The local underlying graded space is $T^*[-2]\Spec B$.
Use its shifted Poisson bracket of degree two with
$\{p_j,x_i\}=\delta_{ij}$ and
$\{\theta_a,\eta_b\}=\{\eta_b,\theta_a\}=\delta_{ab}$.
The signs on $\{x_i,p_j\}$ are negative. The Hamiltonian
\[
 \mathcal H=\sum_a s_a\eta_a+\sum_a\tau_a\theta_a
\]
has degree $-1$ and satisfies
$\{\mathcal H,\mathcal H\}=2\tau(s)=0$.
The differential $\delta=\{\mathcal H,-\}$ is exactly
\eqref{eq:shifted-critical-differential}. The first summand is
the cotangent lift of $\delta_s$; the second is the lift of
$d\nu$. Hence this is the dg model of $\mathrm{dCrit}_B(\nu)$,
equivalently the derived intersection of the zero section and
$d\nu$ in $T^*[-1]\Spec B$.

The local symplectic form is
\[
 \sum_i d_{\mathrm{dR}}x_i\,d_{\mathrm{dR}}p_i
       +\sum_a d_{\mathrm{dR}}\theta_a\,d_{\mathrm{dR}}\eta_a,
\]
with the wedge and suspension signs determined by the displayed
bracket. It is closed, preserved by the Hamiltonian differential,
and pairs the generators in complementary degrees. Thus its
contraction map is an isomorphism of the free tangent and
shifted cotangent modules.

For gluing, $B$ is globally defined from $(E,s)$ and $\nu$ is the
global element \eqref{eq:global-shifted-section}. Changes of frame
in $E$ induce changes of the degree-$-1$ coordinates; their
cotangent lifts transform both $\eta$ and $p$. They preserve the
tautological one-form, and hence the bracket, the lifted
$\delta_s$, and $d\nu$. Therefore they intertwine the displayed
differential. On triple overlaps the cotangent lifts compose,
so the dg algebras and the forms glue. This proves the global
critical construction. Equivariant isomorphisms of $(E,s,\tau)$
induce equivariant isomorphisms of these algebras. This is the
calculation on the reduced derived-scheme charts $W_a$ of the
fourfold presentation. Their Artin presentations also carry the
scalar directions in \eqref{eq:scalar-Artin-reduction} and the
action nerve; those are retained in the full cotangent triangle
\eqref{eq:framed-quotient-diagram-restored}. Identification with
a chosen reduced foliation quotient uses its specified
bundle-and-section presentation.
\end{proof}

This computation gives an explicit meaning to the degree-$-1$
section in the fourfold construction: it is $\tau$ as an element
of the Koszul algebra, not an ordinary function on $T$.
The identity $\tau(s)=0$ is the equation making it closed.
In the remainder of the section, $f$ denotes instead the
ordinary degree-zero potential of a reduced threefold deformation
algebra. The degrees and the ambient dg algebras will be displayed
when the two types of construction are used.

\subsubsection{Compatibility with restriction to open subsets of the base curve}
If $\omega_{\mathcal X}=\pi^*(\omega_C\otimes\lambda)$, a frame of
$\omega_C\otimes\lambda$ on $U$ supplies the same $\Omega$ throughout
$\mathcal X_U$. On another frame $\Omega'=a\Omega$, the family
pairing \eqref{eq:ambient-trace-pairing} is multiplied by the
invertible base function $a$. A change of resolution compares
representatives of one of these fixed pairings as in
\eqref{eq:fixed-duality-homotopy}; it does not change the chosen
frame. For the affine quintic model the chosen $dt$ and relative
volume supply a global frame on that parameter open.

\subsection{Deformation complexes and degree-zero local potentials}\label{ssec: Local models and shifted symplectic structures}
The shifted quotient has been constructed on the specified fourfold
presentation. To relate it to the fiber invariants, we calculate the
full deformation complex of $i_*E$ and the map from the deformation
complex of $E$. Only after identifying its normal term do we pass to
a reduced cyclic model and its ordinary potential.

\subsubsection{Formal moduli and the complete Chevalley--Eilenberg algebra}\label{sssec: A local model for derived critical locus}
Let $Z$ be a proper scheme and $E\in\Perf(Z)$. The deformation Lie algebra is a dg enhancement of $\RH_Z(E,E)$, with commutator bracket. Cyclic homotopy-algebra constructions are developed in \cite{KontsevichSoibelman2006}. Its Maurer--Cartan formal moduli problem describes deformations of $E$ with an identification of the special fiber. For a minimal $L_\infty$ model $L$ on its cohomology, the algebra of formal functions is the completed Chevalley--Eilenberg algebra
\begin{equation}\label{eq:CE}\IHESaliaslabel{eqn: Koszul P}{eq:CE}
 \widehat\CE^{\bullet}(L)
 =\bigl(\widehat\Sym(L^{\vee}[-1]),d_{\CE}\bigr).
\end{equation}
The brackets, not only the graded dimensions, determine $d_{\CE}$. The completion is essential: the formal potential need not be polynomial or convergent. When $L$ has degree-zero automorphisms, the formal moduli problem is stacky; an ordinary affine two-term Koszul algebra is not a presentation of that whole stack.

At $E$, the tangent complex of the full moduli stack is
\[
 \mathbb T_{\RPerf(Z),E}\simeq\RH_Z(E,E)[1],\qquad
 H^{i-1}(\mathbb T_{\RPerf(Z),E})=\ext_Z^i(E,E).
\]
In \eqref{eq:CE}, the function generator dual to a class in
$L^i=\ext_Z^i(E,E)$ has degree $1-i$. In particular,
$\ext_Z^3(E,E)$ is the degree-two tangent cohomology, and its
dual gives degree-$-2$ generators in the completed algebra.
The degree-zero potential constructed below uses the reduced cyclic
deformation algebra specified in Hypothesis~\ref{hyp:cyclic-model},
whose only nonzero terms are $L_{\mathrm{red}}^1=U$ and
$L_{\mathrm{red}}^2=U^\vee$. The following proposition gives the
tangent-amplitude condition for a presentation as a derived critical
locus of a degree-zero function.

\begin{proposition}\label{prop:amplitude}
The derived critical locus of a degree-zero function on a smooth classical scheme has cotangent amplitude $[-1,0]$ and tangent amplitude $[0,1]$. Consequently a formal derived moduli problem with nonzero tangent cohomology in degree $2$ cannot be represented by such a critical locus without first changing the moduli problem. This applies in particular to a surviving $\ext^3$ term.
\end{proposition}
\begin{proof}
For a smooth $U$ and a function $f$, the critical algebra is the Koszul cdga on the section $df$ of $\Omega_U^1$. At a critical point its tangent is $[T_U\xrightarrow{\mathrm{Hess}(f)}\Omega_U^1]$ in degrees $0,1$. Tangent complexes are invariant under equivalence of formal moduli problems.
\end{proof}

\subsubsection{The exact triangle comparing deformations on a divisor and on the ambient variety}
Let $i:X\hookrightarrow Y$ be an effective Cartier divisor with normal line bundle $N$. For a perfect complex $E$ on $X$, the two-step filtration of the derived self-intersection kernel gives a natural triangle
\begin{equation}\label{eq:self-intersection}\IHESaliaslabel{eqn: VirDecomp}{eq:self-intersection}
 E\otimes N^{\vee}[1]\longrightarrow\LLf i^* i_*E
 \longrightarrow E\xrightarrow{\alpha_E}E\otimes N^{\vee}[2].
\end{equation}
This is induced by the truncation triangle of the self-intersection kernel as a bimodule; it is not an assertion that a general complex $E$ has only two cohomology groups. Its extension class records the first infinitesimal neighborhood of the embedding together with the object. It is not generally a split triangle.

\begin{proposition}\label{prop:normal-obstruction}
The triangle \eqref{eq:self-intersection} gives
\begin{equation}\label{eq:normal-deformation}\IHESaliaslabel{eqn: VirDecomp2}{eq:normal-deformation}
 \begin{aligned}
 \RH_X(E,E)&\longrightarrow\RH_Y(i_*E,i_*E)\\
 &\longrightarrow\RH_X(E,E\otimes N)[-1]
 \longrightarrow\RH_X(E,E)[1].
 \end{aligned}
\end{equation}
For a sheaf $E$, the resulting long exact sequence contains
\begin{equation}\label{eq:normal-LES}
 0\longrightarrow\ext^1_X(E,E)\longrightarrow\ext^1_Y(i_*E,i_*E)
 \longrightarrow\Hom_X(E,E\otimes N)
 \xrightarrow{\partial_E}\ext^2_X(E,E).
\end{equation}
Up to the fixed convention for connecting morphisms, $\alpha_E$ is the composite of the Atiyah class of $E$ with the conormal extension class of $X\subset Y$. In a fiber embedding $X_b\subset\cX$, $\partial_E$ evaluates the Atiyah--Kodaira--Spencer obstruction to moving $E$ with the fiber.
\end{proposition}
\begin{proof}
We first identify the three terms and the connecting arrow. Put
$K_E=E\otimes N^\vee$ and $C_E=\LLf i^*i_*E$. The divisor
self-intersection triangle is
\[
 K_E[1]\longrightarrow C_E\longrightarrow E
                  \xrightarrow{\alpha_E}K_E[2].
\]
Applying the contravariant functor $\RH_X(-,E)$ reverses its arrows.
Derived adjunction and the dualizability of $N$ give the diagram
\begin{equation}\label{eq:normal-adjunction-diagram}
\begin{tikzcd}[column sep=large,row sep=large]
 \RH_X(E,E)\arrow[r]\arrow[d,equal]
 &\RH_X(C_E,E)\arrow[r]\arrow[d,"{\sim}"']
 &\RH_X(K_E[1],E)\arrow[d,"{\sim}"']\\
 \RH_X(E,E)\arrow[r]
 &\RH_Y(i_*E,i_*E)\arrow[r]
 &\RH_X(E,E\otimes N)[-1].
\end{tikzcd}
\end{equation}
The rows are distinguished triangles, with the last arrow to
$\RH_X(E,E)[1]$ understood. The first arrow is pushforward of
endomorphisms, and the middle vertical equivalence is the
adjunction $\LLf i^*\dashv i_*$. This constructs
\eqref{eq:normal-deformation} without choosing a splitting.

An element $a\in\Hom_X(E,E\otimes N)$ determines
$\widetilde a=a\otimes\id_{N^\vee}:K_E\to E$.
The connecting homomorphism in the long exact sequence is
\begin{equation}\label{eq:normal-boundary-explicit}
 \partial_E(a)=\widetilde a[2]\circ\alpha_E
                         \in\ext^2_X(E,E),
\end{equation}
with the connecting-arrow convention of the displayed triangle.
For a sheaf $E$, negative self-Ext groups and
$\ext^{-1}_X(E,E\otimes N)$ vanish. The degree-one part of the
long exact sequence is therefore exactly \eqref{eq:normal-LES}.
It says that a prescribed normal parameter $a$ lifts to a
first-order ambient deformation if and only if
$\partial_E(a)=0$. When such a lift exists, all its lifts form a
torsor under $\ext^1_X(E,E)$. This is a statement about
first-order deformation classes; higher Ext groups remain in
\eqref{eq:normal-deformation}.

We next identify $\alpha_E$ geometrically. The conormal triangle
and the Atiyah class fit into
\begin{equation}\label{eq:At-KS-triangle-diagram}
\begin{tikzcd}[column sep=large,row sep=large]
 E\arrow[r,"{\At(E)}"]\arrow[dr,"{\alpha_E}"']
 &E\otimes\mathbb L_X[1]
               \arrow[d,"{\id_E\otimes\kappa_i[1]}"]\\
 &E\otimes N^\vee[2],
\end{tikzcd}
\qquad
\begin{gathered}
 N^\vee\longrightarrow\LLf i^*\mathbb L_Y
          \longrightarrow\mathbb L_X,\\
 \kappa_i:\mathbb L_X\longrightarrow N^\vee[1].
\end{gathered}
\end{equation}
The equality of the diagonal with the composite is the
Atiyah--Kodaira--Spencer obstruction formula \cite{HT}.
It is also the extension class of the self-intersection triangle.
To recall the chain meaning of this identification, consider
$X^{(1)}\subset Y$ cut out by the square of the ideal of $X$.
The ideal of $X$ in $X^{(1)}$ is $N^\vee$ and has square zero.
On an affine cover choose a bounded locally free model for $E$
and lift its differential to $X^{(1)}$. The square of a lifted
differential is an $N^\vee$-valued map of degree two. Associativity
shows that it is a cocycle. Changing the lift by an
$N^\vee$-valued map $h$ of degree one changes its square by
$dh+hd$, since $h^2=0$. On overlaps one records, in addition,
the lifted transition maps and their failure of the cocycle
identity. Together these are a degree-two cocycle in a
\v{C}ech model of $\cRHom(E,E\otimes N^\vee)$.
Using locally free resolutions of the cotangent triangle, the
conormal class records the ambient square-zero extension, and
the Atiyah class records the local connection terms and their
commutator with the differential of $E$. Their Yoneda product
is this degree-two obstruction cocycle, giving
\eqref{eq:At-KS-triangle-diagram}. The use of the cotangent
complex keeps this formula valid when the divisor is singular. Thus the
formula compares the actual embedding and the actual perfect
object, including the overlap data.

Finally let $X=X_b$ be a fiber of $\pi:\mathcal X\to C$ and
$v\in T_{C,b}$. Then $N^\vee=\O_X\otimes T_{C,b}^\vee$.
Contracting $\kappa_i$ with $v$ gives the Kodaira--Spencer
class $\kappa(v):\mathbb L_X\to\O_X[1]$ of the first-order
fiber deformation. Take the normal parameter
$a_v=\id_E\otimes v$. The diagrams above give
\begin{equation}\label{eq:fiber-obstruction-diagram}
\begin{tikzcd}[column sep=large,row sep=large]
 E\arrow[r,"{\At(E)}"]\arrow[d,equal]
 &E\otimes\mathbb L_X[1]
                  \arrow[r,"{\id_E\otimes\kappa_i[1]}"]
                  \arrow[d,"{\id_E\otimes\kappa(v)[1]}"']
 &E\otimes T_{C,b}^\vee[2]\arrow[d,"{\mathrm{ev}_v}"]\\
 E\arrow[r,"{\partial_E(a_v)}"']&E[2]\arrow[r,equal]&E[2].
\end{tikzcd}
\end{equation}
Hence $\partial_E(a_v)=\At(E)\cdot\kappa(v)$, in the stated
composition order. This proves the claimed interpretation of
the normal obstruction for the geometric family. The same
constructions applied to a universal perfect complex are
compatible with base change, giving the universal triangles in
Appendix~\ref{app:checks}.
\end{proof}

\begin{corollary}\label{cor:no-splitting}
A decomposition
\[
 \RH_Y(i_*E,i_*E)\simeq
 \RH_X(E,E)\oplus\RH_X(E,E\otimes N)[-1]
\]
is not a consequence of $N$ being trivial or of $X$ being Calabi--Yau. Even a chosen splitting as complexes does not determine a splitting of the $L_\infty$ operations. The triangle and its higher products must be retained unless a compatible splitting has been proved.
\end{corollary}
\begin{proof}
The connecting morphism in \eqref{eq:normal-deformation} can be nonzero. Moreover the deformation differential in \eqref{eq:CE} contains all transferred higher brackets, which are not specified by a direct-sum decomposition of the underlying cohomology.
\end{proof}

For the fixed embedding of a quintic in $\P^4$, the formal
neighborhood cannot in general be replaced by the normal bundle.
The tangent-sequence calculation in~\cite[Proposition 4.20]{CompanionII}
exhibits the obstruction. The normal triangle above is therefore
used with its connecting morphism, without an assumed splitting.

\subsubsection{Cyclic two-term reductions and ordinary potentials}
Assume the following cyclic model for the reduced formal deformation problem.

\begin{hypothesis}\label{hyp:cyclic-model}
A chosen formal reduced moduli problem is governed by a minimal cyclic $L_\infty$ algebra $L$ with only
\[
 L^1=U,\qquad L^2=U^{\vee}.
\]
The pairing identifies these terms, and the tensors
\[
 c_{n+1}(x_0,\ldots,x_n)
 :=\langle\ell_n(x_1,\ldots,x_n),x_0\rangle
\]
are symmetric for $n\ge2$, in the suspended convention for degree-one inputs. The chosen reduced moduli problem and its cyclic pairing are part of this hypothesis.
\end{hypothesis}
On the divisor loci of \cite[Theorem~7.9]{CompanionII}, the
reduced critical presentation has $L^1=T_{[s]}M_m$,
$L^2=T_{[s]}^*M_m$ and $\ell_n=0$ for every $n$. Its evaluation pairing
and zero cyclic tensors verify this hypothesis; the simultaneous
moduli and orientation checks are in
\cite[Proposition~B.2]{CompanionII}.

\begin{proposition}\label{prop:cyclic-potential}\IHESaliaslabel{lem: A of X}{prop:cyclic-potential}
Under Hypothesis~\ref{hyp:cyclic-model}, define
\begin{equation}\label{eq:formal-potential}
 f(x)=\sum_{n\ge2}\frac{1}{(n+1)!}
 \langle\ell_n(x,\ldots,x),x\rangle
 \in\widehat\Sym(U^{\vee}).
\end{equation}
Then $df$ is the Maurer--Cartan obstruction map, paired with tangent vectors, and
\[
 \widehat\CE^{\bullet}(L)
 \simeq\bigl(\widehat\Sym(U^{\vee})\otimes\Lambda^{\bullet}U,
 \iota_{df}\bigr),
\]
where a generator of $\Lambda^1U$ has cohomological degree $-1$. In particular the formal reduced problem is the formal derived critical locus of $f$.
\end{proposition}
\begin{proof}
Differentiate the homogeneous summand of order $n+1$ in a direction $v$. Its $n+1$ terms are equal by cyclic symmetry, so
\[
 df(x)(v)=\sum_{n\ge2}\frac{1}{n!}
 \langle\ell_n(x,\ldots,x),v\rangle.
\]
This is exactly the Maurer--Cartan series $\sum\ell_n(x^n)/n!$ under $L^2=U^{\vee}$. Degree considerations leave no higher outputs beyond $L^2$, so the Chevalley--Eilenberg differential is zero on the coordinate functions and sends the degree $-1$ generators to these components of $df$. This is the Koszul algebra claimed.
\end{proof}

\begin{remark}
Formula~\eqref{eq:formal-potential} belongs to the complete local ring $\widehat\Sym(U^\vee)$. It constructs a formal critical presentation. For a quasi-smooth $(-1)$-shifted symplectic derived scheme, the Darboux theorem gives local algebraic critical charts \cite{BBJ}. The Artin-stack version describes the associated atlas and reduced Darboux model \cite{BenBassatBravBussiJoyce2015ShiftedDarboux}. We use the formal construction for deformation complexes and the indicated geometric charts for the matrix-factorization calculation.
\end{remark}

\subsubsection{The normal variables and the two degrees of potential}
For $i:X\hookrightarrow Y$ with normal line bundle $N$, put
\[
 U=\ext_X^1(E,E),\qquad
 W_1=\Hom_X(E,E\otimes N),\qquad
 W_2=\ext_X^1(E,E\otimes N).
\]
The exact triangle~\eqref{eq:normal-deformation} gives
\[
 U\longrightarrow\ext_Y^1(i_*E,i_*E)
 \longrightarrow W_1\xrightarrow{\partial_E}\ext_X^2(E,E).
\]
For a fiber embedding, $N=\O_{X_b}\otimes T_{C,b}$. If $E$ is
simple, then $W_1=\CC\,\id_E\otimes T_{C,b}$ and
$W_2=U\otimes T_{C,b}$. The element
$\id_E\otimes\partial_t$ moves the support in the base direction,
and its obstruction is $\At(E)\cdot\kappa(\partial_t)$.
There is no assumption that $W_1\simeq W_2$.

On a specified reduced model, normal mixed operations give a formal
series $\Phi(x):W_1\to W_2$. For degree-zero variables
$a\in W_1$ and $\eta\in W_2^\vee$, its cyclic extension is the
ordinary function
\begin{equation}\label{eq:I-additive-potential}
 F(x,a,\eta)=f(x)+\langle\eta,\Phi(x)a\rangle.
\end{equation}
Its derivatives are
\[
 \partial_\eta F=\Phi(x)a,\qquad
 \partial_aF=\Phi(x)^\vee\eta,\qquad
 d_xF=df+\langle\eta,(d\Phi)a\rangle.
\]
Thus the cyclic extension is additive. Identification of these
operations with the actual ambient deformation algebra also requires
the connecting morphism and all the higher operations of that
algebra. The companion supplies the explicit transfer formulas in
its Sections~4.5.5--4.5.8~\cite{CompanionII}.

The foliation quotient instead has a graded algebra of the form
\[
 B=(R[\theta_1,\ldots,\theta_r],\delta),\qquad
 |\theta_j|=-1,\quad\delta\theta_j=s_j,\qquad
 \nu=\sum_j\tau_j\theta_j,\quad\sum_j\tau_js_j=0.
\]
Here $R$ has degree zero. A factorization $\nu=fg$ with
$g=\sum_jb_j\theta_j$ is exactly the collection of coefficient
identities $\tau_j=fb_j$. The degree-zero variable $\eta$ in
\eqref{eq:I-additive-potential} is not this degree-$-1$ generator.
For a primitive $g$, choose $c_j$ with $\sum_jc_jb_j=1$ and put
$\kappa=\sum_jc_j\partial_{\theta_j}$. The identities
\begin{equation}\label{eq:I-extraction}
 \kappa(g)=1,\qquad f=\kappa(\nu),\qquad
 \partial_tf=\kappa(\partial_t\nu)-f\kappa(\partial_tg)
\end{equation}
follow from the graded Leibniz rule. They apply when the stated
factorization and comparison with the actual critical presentation
have been supplied.

A universal product is not among our conjectures. For example, in a
fiber-compatible presentation with $\tau_j|_{t=0}=\partial_jf_0$, a
regular product with $f|_{t=0}=f_0\ne0$ would require
$\partial_jf_0=f_0b_j|_{t=0}$. If $f_0$ has order $m\ge2$ at a
critical point, some derivative has order $m-1$, whereas a multiple
of $f_0$ has order at least $m$. The companion proves this obstruction,
and its specified real counterpart, in Propositions~4.29 and~4.36
\cite{CompanionII}. Accordingly the remaining comparison problem is
formulated in terms of compatible cyclic and bar complexes, not by
assuming a factorization which fails this test.
\section{Gauss--Manin comparison and its conjectural extension}
\label{sec:cyclic}
\IHESaliaslabel{ssec: Matrix factorizations and GM-connection}{sec:cyclic}
The geometric constructions have produced the special-fiber derived
intersection and the fourfold deformation problem. To obtain DT
cohomology, one must also choose the critical structure and its
orientation. We first fix these data and the local cyclic complex,
then distinguish regular specialization from a logarithmic limit.
The global connection and the geometric comparisons in the resolved
quintic are proved in the companion~\cite{CompanionII}.

\subsection{Oriented critical structures and the cohomological invariant}
\label{sec:I-DT}
Let $M_b$ be a finite-type algebraic space obtained by classical scalar
rigidification of a chosen simple-perfect-complex locus. Scalar
rigidification removes the central $\mathbb G_m$ from isomorphism
sheaves; it does not define a two-term derived enhancement. The
algebraic-space construction is that of~\cite{Lieblich2006}.
For the cohomological invariant, specify compatible critical charts
\[
 R_\alpha\hookrightarrow U_\alpha,\qquad
 R_\alpha=\operatorname{Crit}(f_\alpha),\quad
 f_\alpha|_{R_\alpha^{\mathrm{red}}}=0,
\]
where $U_\alpha$ is smooth. Their classes modulo the squares of the
defining ideals determine a d-critical structure
$s_b\in H^0(M_b,\mathcal S^0_{M_b})$~\cite{Joyce}.

\begin{hypothesis}\label{hyp:DT}
Choose a line bundle $\mathcal O_b^{\mathrm{or}}$ on
$M_b^{\mathrm{red}}$ and an isomorphism
\[
 (\mathcal O_b^{\mathrm{or}})^{\otimes2}\xrightarrow{\sim}K_{M_b,s_b}.
\]
In a family the local orientations and comparison isomorphisms are
required to agree on intersections and satisfy the cocycle condition.
\end{hypothesis}
The oriented vanishing-cycle construction~\cite{BBDJS} defines
$\mathcal P_b$ and
\[
 \mathrm{DT}_{\mathrm{coh}}^j(M_b)
     =\mathbb H_c^j(M_b^{\mathrm{an}},\mathcal P_b).
\]
A concrete instance is the smooth section-complex locus, with zero
d-critical structure and orientation $K_{M_b}$: its virtual canonical
bundle is $K_{M_b}^{\otimes2}$ and $\mathcal P_b=\CC_{M_b}[\dim M_b]$.
This uses no square root of the ordinary canonical bundle. The
companion verifies the geometric moduli identification for the
resolved quintic in Theorem~7.12~\cite{CompanionII}.

\subsection{The ordinary cyclic complex and its base derivative}
For a smooth scheme $U$ and an ordinary function $f$, a matrix
factorization is a two-periodic locally free object with odd
differential $d_E$ satisfying $d_E^2=f\,\id_E$, in the coherent
localization used in~\cite{Efimov2018CyclicMF}. Its periodic cyclic
complex is compared, under the hypotheses of that theorem, with
\begin{equation}\label{eq:I-twisted-DR}
 (\Omega_U^\bullet((u)),D_f),\qquad D_f=u\,d-df\wedge.
\end{equation}
This is an ordinary degree-zero curvature. No matrix-factorization
category for the degree-$-1$ potential $\nu$ is used here.

Let $V\subset C$ have coordinate $t$ and let
$F\in\Gamma(U\times V,\O_{U\times V})$ be a compatible ordinary
potential for the selected relative deformation problem. Set
\[
 D_F=u\,d_U-d_UF\wedge,\qquad
 \nabla_{\partial_t}=\partial_t-u^{-1}F_t.
\]
The commutator is the sum of
\[
 [\partial_t,D_F]=-d_UF_t\wedge,
 \qquad[-u^{-1}F_t,D_F]=d_UF_t\wedge,
\]
so $[\nabla_{\partial_t},D_F]=0$. This constructs the chain operator
on that presentation. If $F$ has been extracted from a primitive
graded product, equation~\eqref{eq:I-extraction} computes the same
coefficient $F_t$. It does not assert that such a product exists.

For a nontrivial smooth ambient family $q:\mathscr U\to V$, choose
local lifts $v_\alpha$ of $\partial_t$. On overlaps put
$w_{\alpha\beta}=v_\beta-v_\alpha$ and
$h_{\alpha\beta}=u^{-1}\iota_{w_{\alpha\beta}}$. Cartan's formula gives
\[
 [D_F,h_{\alpha\beta}]
 =\mathcal L_{w_{\alpha\beta}}-u^{-1}w_{\alpha\beta}(F)
 =\nabla_\beta-\nabla_\alpha.
\]
The global Cech correction, its homotopies and its independence of
lifts are constructed in~\cite[Theorem 5.7]{CompanionII}. Applying it
to the original moduli problem requires the ordinary critical charts
to be identified with that problem, and the chart identifications to
intertwine these operators up to the specified homotopies.

\subsection{The finite regular comparison datum}\label{sec:I-finite-datum}
Put $K=\CC((u))$ and $R=\Gamma(V,\O_C)\otimes_\CC K$. The general
regular comparison uses a two-periodic complex $P$ of finite
projective $R$-modules with a $K$-connection $\nabla$ such that
\begin{equation}\label{eq:I-finite-connection}
 \nabla(ap)=da\otimes p+a\nabla p,\qquad
 [\nabla,d_P]=0,
\end{equation}
together with geometric quasi-isomorphisms
\begin{equation}\label{eq:I-fiber-identifications}
\begin{aligned}
 P\otimes_R^{\mathbf L}K_b
   &\xrightarrow{\sim}\mathcal K_b^{\mathrm{DT}} &&(b\ne0),\\
 P\otimes_R^{\mathbf L}K_0
   &\xrightarrow{\sim}R\Gamma_c(M_0,\mathcal B_{0,\mathrm{DT}}).
\end{aligned}
\end{equation}
Here $\mathcal K_b^{\mathrm{DT}}$ denotes the normalized periodic
realization, with the critical-chart suspension included. The second
map is an identification with the actual oriented central bar
complex, not merely a vector-space isomorphism chosen after taking
cohomology. Its construction belongs to the geometric comparison.

A finite regular datum~\eqref{eq:I-finite-connection} has locally free
cohomology and commutes with derived base change. The finite-complex
proof appears in~\cite[Theorem 5.10]{CompanionII}. Over $K[[t]]$, a
connection matrix $A(t)$ has the unique formal fundamental solution
$G'=-A G$, $G(0)=1$. This compares the central fiber with the
\emph{formal} generic fiber after extending scalars to $K((t))$.
An evaluation of this formal series at a nonzero complex point is
not part of the construction. Analytic transport between closed
fibers requires an analytic realization and a chosen path.

For the complete divisor loci of Conjectures~\ref{conj:I-completion}
and~\ref{conj:I-regular-comparison}, the proposed finite complex has
one generator $e_a$ for each $0\le a\le d_m$, in parity $d_m$, with
\[
 d_P=0,\qquad \nabla e_a=0,\qquad
 e_a\longmapsto I\mathcal I^{-1}(u^az_a).
\]
The $z_a$ are Cech--de Rham representatives of $H^a$, $I$ is the
bar inclusion and $\mathcal I$ is contraction with top forms.
The companion constructs these maps on the actual resolved quintic
locus~\cite[Theorem 7.12, equations (7.37)--(7.40)]{CompanionII}.
Their constant horizontal basis gives comparison of closed fibers
as well as formal fibers in this example.

\subsection{Logarithmic comparison and nearby cycles}
A pole-free extension is a restriction on the monodromy of the
natural connection. If a holomorphic connection is defined on a
disk, its local horizontal frame is single-valued, so its restriction
to the punctured disk has trivial monodromy. Consequently nontrivial
unipotent monodromy excludes the regular datum above.

The sheaves $j_{D*}I_{x/D}$ in a hyperplane linear system provide an
example within the moduli problem. On a smooth quintic their moduli
locus $Y_b$ is the projective-bundle incidence variety over $X_b$.
With $\xi=c_1(\O_Y(1))$, the horizontal projective-bundle maps give
\[
 H^k(Y_b,\CC)\simeq\bigoplus_{a=0}^3
                  \xi^a p^*H^{k-2a}(X_b,\CC).
\]
The companion computes $\operatorname{rank}(\log T_{X_b})=21$ on
$H^3(X_b)$ and hence rank $84$ on the odd periodic DT cohomology of
$Y_b$~\cite[Theorem 7.14]{CompanionII}. This rules out extending
\eqref{eq:I-fiber-identifications} with a pole-free connection for
that sheaf locus. It does not contradict the constant divisor-section
comparison.

\begin{problem}\label{prob:I-logarithmic}
For a specified bounded family of oriented DT moduli on $C^*$,
construct a finite logarithmic cyclic realization, with a connection
\[
 \nabla:\mathcal H\longrightarrow
        \mathcal H\otimes\Omega_C^1(\log\{0\}),
\]
compatible with its geometric Gauss--Manin system. Identify its
limiting fiber and residue with the nearby-cycle realization of the
oriented DT complex. Construct, when it exists, a specialization map
from a defined central derived-intersection DT complex to this
nearby-cycle complex, and determine its cone.
\end{problem}

This problem requires a relative constructible DT complex before its
nearby cycles can be taken. In the smooth proper incidence example,
that complex on the punctured family is the constant sheaf shifted
by six. Its semistable projective-bundle completion provides the
logarithmic relative de Rham model, with residue
$-(2\pi i)^{-1}\log T$; the proof is
\cite[Proposition 7.15]{CompanionII}. The limiting fiber is nearby
cohomology. A singular point of $X_0$ need not define a perfect point
ideal, so this construction does not assert a moduli interpretation
of every point of the singular incidence variety.

For a relative constructible complex $\mathcal P$, any comparison
with the ordinary central fiber is measured by the specialization
triangle
\[
 i^*\mathcal P\longrightarrow R\psi\mathcal P
   \longrightarrow\operatorname{Cone}(i^*\mathcal P\to R\psi\mathcal P)
   \xrightarrow{+1}.
\]
Vanishing of the cohomology of this cone is a separate requirement
for an isomorphism. Thus a logarithmic extension and a regular
central-fiber identity are two different conclusions. The derived
restriction and bar constructions in the next section specify the
central object to which such a comparison would apply.
\section{Computation of DT cohomology over the special fiber}
\label{sec:quantization}\IHESaliaslabel{def-quant}{sec:quantization}
The ordinary two-sided bar resolution records the derived intersection.
The cohomological DT complex also requires its orientation and a
compatible differential. We give the quantized construction on smooth
Lagrangian presentations and its bounded Spencer realization on critical
charts. The finite filtration controls reduction modulo $\hbar^N$;
projectivity alone is not used to assert degeneration at the second page.

\subsection{Deformation quantization approach}\label{ssec: Deformation quantization approach}
\subsubsection{The quantization comparison and its coefficients}
Let $Z$ be a holomorphic symplectic manifold of dimension $2n$, and let $L_1,L_2\subset Z$ be smooth closed Lagrangian submanifolds equipped with square roots $K_{L_i}^{1/2}$. Write $(\phi_{L_1\cap L_2},T_{L_1\cap L_2})$ for the DT perverse sheaf and its monodromy. The holomorphic quantization background is \cite{DS,KS2}. The canonical formal deformation-quantization algebroid is denoted $\widehat{\mathcal W}_Z$. We use the right module $\mathcal Q_1$ and left module $\mathcal Q_2$ associated to the two oriented Lagrangians, with zero additional monodromy parameters and the module and duality normalizations of \cite{GS}.

The following comparison theorem is due to Gunningham and Safronov \cite[Theorem A and Theorem 7.1]{GS}.
\begin{theorem}\label{thm:GS}\IHESaliaslabel{thm: GS}{thm:GS}
In this setting there is an isomorphism of differential perverse sheaves
\begin{equation}\label{eq:GS}
 \mathcal Q_1\otL_{\widehat{\mathcal W}_Z}\mathcal Q_2
 \simeq
 \operatorname{RH}^{-1}\bigl(\phi_{L_1\cap L_2},T_{L_1\cap L_2}\bigr).
\end{equation}
Here the right-hand side uses the formal Riemann--Hilbert correspondence in the quantization variable $\hbar$; it is a $\CC((\hbar))$-linear object with its regular $\hbar$-connection. This is \cite[Theorem A, Theorem 7.1]{GS}.
\end{theorem}
The coefficient change, connection, and monodromy in \eqref{eq:GS} are part of the assertion. A derived tensor product must use a right and a left module, or an explicitly dualized alternative; two unspecified left modules do not define that product. Also, an algebroid is not automatically a globally untwisted sheaf of algebras. These distinctions are important for the algebraic models below.

For the integral calculation we make the quantization algebra and its $\CC[[\hbar]]$-lattices explicit in the following hypothesis. Their reduction modulo $\hbar$ gives the classical modules whose derived intersection appears in the first page of the spectral sequence.

\subsubsection{The DT complex computed by the bar construction}
Keep the oriented Lagrangians and quantization modules of Theorem~\ref{thm:GS}, and put $M_0=L_1\cap L_2$ for the classical support. Resolve the right and left modules by flat models in a compatible quantization presentation. The sheafified two-sided bar complex, with the completions required by that presentation, defines
\begin{equation}\label{eq:DT-quantized-bar}
 \mathcal B^{\hbar}_{0,\mathrm{DT}}
  :=\operatorname{Tot}\operatorname{Bar}
       (\mathcal Q_1,\widehat{\mathcal W}_Z,\mathcal Q_2)
  \simeq\mathcal Q_1\otL_{\widehat{\mathcal W}_Z}\mathcal Q_2.
\end{equation}
The terms have the right-module action at the first face, the left-module action at the last face, and multiplication in the intermediate faces. The differential is written explicitly in \eqref{eq:bar}. With integral lattices one first constructs this complex modulo $\hbar^N$, uses the derived inverse limit, and then inverts $\hbar$, as specified in Hypothesis~\ref{hyp:DQ}.

By Theorem~\ref{thm:GS}, \eqref{eq:DT-quantized-bar} represents the differential realization of the oriented DT sheaf on this intersection, with the Riemann--Hilbert and shift conventions of that theorem. Define its periodic version by
\begin{equation}\label{eq:DT-bar-periodic}
 \mathcal B_{0,\mathrm{DT}}
  :=\operatorname{Per}\bigl(
       \mathcal B^{\hbar}_{0,\mathrm{DT}}
       \otimes_{\CC((\hbar)),\,\hbar\mapsto-u}K\bigr),
 \qquad K=\CC((u)).
\end{equation}
Here $\operatorname{Per}$ places the even and odd degrees of a bounded representative in the two parity terms and retains its differential. The shift fixed by the DT normalization is included before this operation. The coefficient convention $\hbar\mapsto-u$ matches the differential convention used in Section~\ref{sec:cyclic}: on the critical chart of \cite[Theorem D]{GS}, the differential $d+\hbar^{-1}df\wedge$ becomes $d-u^{-1}df\wedge$, and multiplication by $u^p$ on $p$-forms identifies it with $u\,d-df\wedge$. The cohomology calculated on the special fiber is
\begin{equation}\label{eq:DT-bar-groups}
 \mathbb H_c^{\bar i}(M_0,\mathcal B_{0,\mathrm{DT}}),
       \qquad\bar i=\overline0,\overline1.
\end{equation}
This is the central complex and these are the central cohomology groups in \cite[Hypothesis~5.11]{CompanionII}. The formal connection in $\hbar$, which records the local DT monodromy in Theorem~\ref{thm:GS}, is distinct from the Gauss--Manin connection along $C$ that compares the geometric fibers.

For a critical chart, Proposition~\ref{prop:local-BV} supplies the local mixed differential $-\delta_f+u\Delta$ and identifies its cohomology with the ordinary-potential twisted de Rham calculation. Thus the DT calculation uses the bar complex together with its quantization or BV differential and orientation, not just the list of Tor sheaves of the classical intersection.

\subsubsection{Filtrations}
To calculate the quantized tensor product, we reduce it modulo successive
powers of $\hbar$. The finite filtration below gives a convergent
spectral sequence at each order; its reduction is the ordinary derived
intersection. The required complete representative and reduction map
are specified before the theorem.

\begin{hypothesis}\label{hyp:DQ}
On a chosen open part of $Z$, or in an algebroid presentation with compatible descent, fix a flat, $\hbar$-adically complete sheaf of associative $\CC[[\hbar]]$-algebras $\mathcal A_\hbar$ with $\mathcal A_\hbar/\hbar\simeq\O_Z$. Fix a right module $\mathcal N_\hbar$ and a left module $\mathcal M_\hbar$, both flat over $\CC[[\hbar]]$, reducing to $\mathcal N_0$ and $\mathcal M_0$. Assume the completed derived tensor product has a bounded complex model
\[
 C^\bullet\simeq\mathcal N_\hbar\widehat\otimes^{\mathbf L}_{\mathcal A_\hbar}\mathcal M_\hbar
\]
whose terms are $\hbar$-torsion-free and complete, and for which derived reduction gives
\begin{equation}\label{eq:DQ-basechange}
 C^\bullet\otL_{\CC[[\hbar]]}\CC\simeq
 \mathcal N_0\otL_{\O_Z}\mathcal M_0.
\end{equation}
These are model hypotheses, to be checked for the quantization modules in an application. For example, one may establish them using suitable flat resolutions and completed tensor products. An isomorphism after inverting $\hbar$ alone does not supply them.
\end{hypothesis}
The complete lattices and the reduction map in this hypothesis are
constructed in Proposition~\ref{prop:r14-complete-Spencer} for the
zero section and the graph of $df$ on a smooth critical presentation.
For the divisor moduli $M_m$, the global zero-critical presentation
uses the Rees algebra of $\mathscr D_{M_m}$ and the lattices
$\O_{M_m}[[\hbar]]$, $\omega_{M_m}[[\hbar]]$.
\cite[Proposition~B.3]{CompanionII} verifies the global hypotheses
and gives the coordinate differentials. It verifies torsion-free
terms and derived reduction, not an additional assertion that the
cohomology of the integral complex is $\hbar$-torsion-free.

\subsubsection{The complete lattices on a critical presentation of the moduli locus}
\label{sssec:r14-DQ-verification}
We verify Hypothesis~\ref{hyp:DQ} in the zero-section and graph
presentation associated to a critical chart of the moduli problem.
Let $U$ be a smooth affine coordinate chart of dimension $n$ and
let $f\in\Gamma(U,\O_U)$ be its ordinary potential. The
Lagrangians in this presentation are
\[
 L_0=\{p_1=\cdots=p_n=0\},\qquad
 L_f=\{p_i=\partial_{x_i}f\}\quad\text{in }T^*U.
\]
Their derived intersection is $\dCrit_U(f)$. If the chart is a
chart of the selected locus in
$\mathfrak X_1\times^h_{\mathfrak S}\mathfrak X_2$, it represents
that moduli problem. These local zero-section and graph
Lagrangians are not being identified with the full restriction
morphisms $\mathfrak X_i\to\mathfrak S$.

\begin{proposition}\label{prop:r14-complete-Spencer}
On this chart, write $B=\Gamma(U,\O_U)$ and consider the
$\hbar$-adically completed algebra with ordered symbols
\[
 A_\hbar=B[p_1,\ldots,p_n][[\hbar]],\qquad
 [p_i,a]=\hbar\partial_i a,\quad[p_i,p_j]=0.
\]
The expression for the underlying module means that each
coefficient of $\hbar^j$ is polynomial in the $p_i$.
This is the local Rees algebra of differential operators,
completed in $\hbar$. Its $\hbar$-adically completed localizations
on cotangent coordinate opens give the corresponding quantization
sheaf. Define a left module and a right module by
\begin{equation}\label{eq:r14-quantization-actions}
 \begin{aligned}
 M_{f,\hbar}&=B[[\hbar]],&
 p_i\cdot m&=\hbar\partial_i m+f_i m,\\
 N_{0,\hbar}&=\Omega_B^n[[\hbar]],&
 (a\,\mathrm{vol})\cdot p_i&=-\hbar(\partial_i a)\,\mathrm{vol},
 \qquad f_i=\partial_i f.
 \end{aligned}
\end{equation}
Functions act by multiplication. The completed derived tensor
product has the bounded complete representative
\begin{equation}\label{eq:r14-Spencer-tensor}
 C^{-q}=\Omega_B^n[[\hbar]]\otimes_B\Lambda^qT_B,
 \quad0\leq q\leq n,\qquad
 d_C=-\sum_i(\hbar\partial_i+f_i)\partial_{e_i},
\end{equation}
where $e_i=\partial_{x_i}$ and $\partial_{e_i}$ is left exterior
contraction on the symbols. In particular, its terms are
$\hbar$-torsion-free and complete, and
\begin{equation}\label{eq:r14-DQ-derived-reduction}
 C\otimes^{\mathbf L}_{\CC[[\hbar]]}\CC
 \simeq
 \Omega_B^n\otimes^{\mathbf L}_{B[p_1,\ldots,p_n]}B_f,
 \quad p_i|_{\Omega_B^n}=0,\quad p_i|_{B_f}=f_i.
\end{equation}
Thus these modules satisfy Hypothesis~\ref{hyp:DQ}, with
$\mathcal N_0=\omega_U$ on $L_0$ and
$\mathcal M_0=\O_U$ on $L_f$, in this quantization presentation.
No regular-sequence assumption on the derivatives $f_i$ is
needed.
\end{proposition}
\begin{proof}
The relations define an associative ordered-symbol algebra: on
symbols its multiplication is
\[
 a\star b=\sum_{\alpha\in\mathbb N^n}
       \frac{\hbar^{|\alpha|}}{\alpha!}
             (\partial_p^\alpha a)(\partial_x^\alpha b).
\]
At each $\hbar$ order the coefficient is a finite sum. The
commuting differentiations show associativity by expanding the
two iterated sums, or equivalently by composing normally ordered
differential operators. This gives the asserted PBW expansion
and flatness over $\CC[[\hbar]]$.

On $M_{f,\hbar}$ one has
\[
 [\hbar\partial_i+f_i,\hbar\partial_j+f_j]
       =\hbar(f_{ij}-f_{ji})=0,\qquad
 [\hbar\partial_i+f_i,a]=\hbar\partial_i a.
\]
For the right module, direct evaluation gives
\[
 ((a\,\mathrm{vol})p_i)b-((ab\,\mathrm{vol})p_i)
                      =\hbar a\partial_i b\,\mathrm{vol}.
\]
These identities verify both actions in
\eqref{eq:r14-quantization-actions}. The underlying modules are
power-series modules and have no $\hbar$-torsion, hence are flat
over the discrete valuation ring $\CC[[\hbar]]$.

Set $T_i=p_i-f_i$. Since $[T_i,T_j]=0$, the complex of free
left $A_\hbar$-modules
\[
 P^{-q}=A_\hbar\otimes_B\Lambda^qT_B,\qquad
 d_P(a\otimes e_{i_1}\wedge\cdots\wedge e_{i_q})
  =\sum_{r=1}^q(-1)^{r-1}aT_{i_r}\otimes e_{I\setminus i_r}
\]
satisfies $d_P^2=0$. Its augmentation sends $a\in P^0$ to
$a\cdot1\in M_{f,\hbar}$. Modulo $\hbar$ this is the Koszul
resolution on
\[
 p_1-f_1,\ldots,p_n-f_n\quad\text{in }B[p_1,\ldots,p_n].
\]
Translation $p_i\mapsto p_i+f_i$ identifies this sequence with
the coordinate variables $p_i$. It is therefore a regular
sequence and its quotient is $B$, without any condition on
$df$.

For completeness, exactness lifts to the completed resolution.
In the augmented complex, let $z$ be a cocycle. Its reduction
modulo $\hbar$ is $d_0y_0$ for some $y_0$. Lift $y_0$ to $P$
and write $z-d_Py_0=\hbar z_1$. Torsion-freeness implies
$d_Pz_1=0$. Repeating gives $z_j-d_Py_j=\hbar z_{j+1}$ and
\[
 z=d_P\Bigl(\sum_{j\ge0}\hbar^j y_j\Bigr).
\]
The sum converges termwise since the modules are complete.
Thus $P\to M_{f,\hbar}$ is a bounded free resolution in the
completed category. Its reductions modulo $\hbar^N$ give the
corresponding resolutions at every finite order.

Tensor the resolution on the left with the right module
$N_{0,\hbar}$. The tensor relation sends right multiplication
by $T_i$ on $A_\hbar$ to
\[
 (a\,\mathrm{vol})T_i
             =-(\hbar\partial_i a+f_i a)\,\mathrm{vol}.
\]
This is exactly \eqref{eq:r14-Spencer-tensor}. Boundedness and
termwise completeness allow the completed tensor product to be
computed by the inverse limit of these finite-order complexes.
Their transition maps are surjective in each term.

Each term of $C$ is flat over $\CC[[\hbar]]$, so its derived
reduction is its ordinary reduction. The latter has differential
$-\sum_i f_i\partial_{e_i}$. On the other hand, tensoring the
classical regular-sequence resolution just constructed with
$\Omega_B^n$ on the zero section gives the same differential.
This proves \eqref{eq:r14-DQ-derived-reduction} as an identity
of the displayed complexes and establishes all of the model
conditions in Hypothesis~\ref{hyp:DQ}.

The relation with the DT differential can also be checked
without invoking a filtered convergence theorem. Define
\[
 J_q(a\,\mathrm{vol}\otimes e_{i_1}\wedge\cdots\wedge e_{i_q})
  =(-1)^{nq+q(q+1)/2}
       a\,\iota_{e_{i_q}}\cdots\iota_{e_{i_1}}\mathrm{vol}.
\]
The graded Leibniz rule for contraction gives
\[
 J_{q-1}d_C=(-1)^n(\hbar d+df\wedge)J_q.
\]
The sign $(-1)^n$ is the differential sign in a shift by $[n]$.
Thus $C$ is the shifted twisted de Rham complex
$(\Omega_B^\bullet[[\hbar]],\hbar d+df\wedge)[n]$.
After $\hbar=-u$, multiplying $p$-forms by $(-1)^p$ identifies
its differential with $u\,d-df\wedge$, with the same shift.
This is the ordinary-potential differential used for the
fiberwise DT complex. The local vanishing-cycle interpretation
is the twisted de Rham comparison \cite[Theorem D]{GS}.
\end{proof}

The preceding construction applies in particular on every critical
chart of $M=\P^4$ in \cite[Section~7]{CompanionII}, where
$n=4$ and $f=0$. Its complete representative and classical
reduction are
\begin{equation}\label{eq:r14-hyperplane-DQ-lattice}
 C^{-q}=\omega_U[[\hbar]]\otimes\Lambda^qT_U,
 \quad d_C=-\hbar\sum_i\partial_i\partial_{e_i},\qquad
 C/\hbar=(\omega_U\otimes\Lambda^\bullet T_U,0).
\end{equation}
There are five terms in degrees $-4,\ldots,0$, of ranks
$1,4,6,4,1$ over $\O_U[[\hbar]]$. After contraction they are
$\Omega_U^\bullet[[\hbar]][4]$ with differential $\hbar d$.
Exactly the same construction works in dimensions $14$ and $34$
for the degree-two and degree-three divisor loci of
\cite[Theorem~7.9]{CompanionII}.

These complexes also glue for the zero-potential projective
moduli spaces. On $y=y(x)$, the operators change by
\[
 p^y_j=\sum_i\frac{\partial x_i}{\partial y_j}p^x_i,
 \qquad dy_1\wedge\cdots\wedge dy_n
         =\det\left(\frac{\partial y}{\partial x}\right)
                                 dx_1\wedge\cdots\wedge dx_n.
\]
The left action on $\O_M[[\hbar]]$ is $\hbar v$, and the right
action on $\omega_M[[\hbar]]$ is $-\hbar\mathcal L_v$.
Consequently they are the usual globally defined left and right
Rees differential-operator modules. In an arbitrary vector-field
frame the Spencer differential is
\[
\begin{aligned}
 d(a\otimes v_1\wedge\cdots\wedge v_q)
 &=\sum_i(-1)^{i-1}a(\hbar v_i)\otimes v_{\widehat i}\\
 &\quad+\hbar\sum_{i<j}(-1)^{i+j}
       a\otimes[v_i,v_j]\wedge v_{\widehat i\widehat j}.
\end{aligned}
\]
This formula verifies compatibility under the displayed
coordinate transformations. Contraction with top forms
identifies its tensor complex globally with
$(\Omega_M^\bullet[[\hbar]],\hbar d)[n]$. Hence the comparison
modulo every $\hbar^N$ and the derived reduction commute with
all overlap maps. It is an actual complete lattice for the zero
critical presentation used in the divisor-sheaf calculation.

The right module here reduces to $\omega_M$, and the left one
to $\O_M$. Their product of line bundles is the orientation
$K_M$ of the zero d-critical locus, whose virtual canonical
bundle is $K_M^{\otimes2}$. This asymmetric choice does not
assume a square root of $K_M$. For example,
$K_{\P^4}=\O_{\P^4}(-5)$ has no square root, whereas $K_{\P^4}$
is a square root of $K_{\P^4,0}=K_{\P^4}^{\otimes2}$.
Thus the calculation above verifies Hypothesis~\ref{hyp:DQ}
for these explicitly chosen Rees modules. When applying
Theorem~\ref{thm:GS} to a different pair of geometric
Lagrangians, its separate half-canonical-bundle and quantization
module hypotheses must be checked for that pair. For the present
zero critical presentation the DT identification follows
directly from the global de Rham complex just computed.

\begin{theorem}\label{thm:DQ-spectral}\IHESaliaslabel{spectral}{thm:DQ-spectral}
Under Hypothesis~\ref{hyp:DQ}, for each integer $N\geq1$ there is a convergent spectral sequence of sheaves
\begin{equation}\label{eq:DQ-spectral}
 E_1^{p,q}=\cH^{p+q}\bigl(\mathcal N_0\otL_{\O_Z}\mathcal M_0\bigr)\,\hbar^p
 \ \Longrightarrow\ \cH^{p+q}(C^\bullet/\hbar^N C^\bullet),
 \qquad 0\leq p<N.
\end{equation}
Its differentials have bidegree $(r,1-r)$. In particular, the terms on the first page can be expressed using the ordinary sheaves $\Tor_j^{\O_Z}(\mathcal N_0,\mathcal M_0)$ in cohomological degree $-j$. The first differential is the Bockstein determined by the first-order deformation of the whole tensor-product complex. No degeneration at $E_2$ is asserted without a further hypothesis.
\end{theorem}
\begin{proof}
Put $F^p(C/\hbar^N)=\hbar^pC/\hbar^NC$ for $0\leq p\leq N$, with $F^0=C/\hbar^N$ and $F^N=0$. Since each term of $C$ is $\hbar$-torsion-free, multiplication by $\hbar^p$ induces an isomorphism
\[
 \gr_F^p(C/\hbar^N)\simeq(C/\hbar C)\,\hbar^p.
\]
All differentials preserve the filtration. The spectral sequence of this finite filtered complex therefore has the stated $E_1$-page, by \eqref{eq:DQ-basechange}, and converges to the associated graded of its cohomology. Boundedness and the finite filtration eliminate convergence issues at this stage.

For the Bockstein description, let $z_0$ be a cocycle modulo $\hbar$, and lift it to $z$ modulo $\hbar^2$. Then $dz$ is divisible by $\hbar$, and the first differential is represented by $(\hbar^{-1}dz)\bmod\hbar$. Changing the lift or representative changes this by a boundary. This construction involves the deformations of the algebra and both modules, not just a selected action on one of them.
\end{proof}
If compatible splittings are available, write locally
\[
 d=d_0+\hbar d_1+\hbar^2d_2+\cdots.
\]
Then $d_1$ induces the first differential on $H(C_0,d_0)$. The coefficient identities in $d^2=0$ imply that its square is zero on that cohomology. Higher differentials may depend on $d_2,d_3,\ldots$ and on successive lifting choices. Passing from the finite-order spectral sequences to a statement about $C$ itself requires a completeness and inverse-limit argument, for example a verified Mittag--Leffler condition on the relevant cohomology systems. It is not part of finite-filtration convergence.

\begin{proposition}\label{prop:DQ-formality}
Suppose, in addition, $C$ is filtered quasi-isomorphic to
\[
 (H^\bullet[[\hbar]],\hbar\delta),\qquad \delta^2=0,
\]
where $H^\bullet$ has zero differential before the displayed deformation and the quasi-isomorphism induces an isomorphism on associated-graded cohomology. Then all the spectral sequences \eqref{eq:DQ-spectral} degenerate at $E_2$.
\end{proposition}
\begin{proof}
For this model the first differential is $\delta$, except where it exits the finite filtration. Every class on the second page that is not in the last filtration step is represented by $\hbar^p z$ with $\delta z=0$; it is already a cocycle in the filtered complex. A class in the last step is automatically a cocycle modulo $\hbar^N$. Thus all second-page classes have cocycle representatives and all higher differentials vanish. A filtered quasi-isomorphism as stated identifies the spectral sequences from the first page onward.
\end{proof}
The filtered quasi-isomorphism in Proposition~\ref{prop:DQ-formality} is the additional hypothesis for second-page degeneration.

\subsubsection{The two-sided bar resolution}
Let $A$ be a unital associative algebra over a field, $N$ a right $A$-module, and $M$ a left $A$-module. The two-sided bar model for $N\otL_A M$ has terms
\begin{equation}\label{eq:bar}
 B_n(N,A,M)=N\otimes A^{\otimes n}\otimes M\qquad(n\geq0),
\end{equation}
placed in cohomological degree $-n$. Its differential is
\begin{align*}
 d(n_0\otimes a_1\otimes\cdots\otimes a_n\otimes m)
 &=n_0a_1\otimes a_2\otimes\cdots\otimes a_n\otimes m\\
 &\quad+\sum_{i=1}^{n-1}(-1)^i n_0\otimes\cdots\otimes a_i a_{i+1}\otimes\cdots\otimes m\\
 &\quad+(-1)^n n_0\otimes a_1\otimes\cdots\otimes a_{n-1}\otimes a_n m.
\end{align*}
The simplicial identities for multiplication and the two module actions imply $d^2=0$. In the dg setting one adds the internal differential and the Koszul signs. Over a general base ring, use appropriate flat replacements before totalization. The relative tensor product is over $A$, not over the shifted complex $A[1]$. The shift in a reduced bar construction is a suspension of its augmentation ideal, not a new unital algebra structure on $A[1]$.

\subsubsection{An explicit local BV computation}
An ordinary critical chart admits a concrete mixed complex which supplies the desired local differential without assuming a global homotopy BV theorem. Let $U\subset\A^n$ have coordinates $x_1,\ldots,x_n$, let $f\in\O(U)$, and put
\[
 K_f=\O(U)[\theta_1,\ldots,\theta_n],\qquad |\theta_i|=-1,
 \qquad \delta_f=\sum_i f_i\partial_{\theta_i},\qquad
 \Delta=\sum_i\partial_{x_i}\partial_{\theta_i}.
\]
The variables $\theta_i$ anticommute and their derivatives are left odd derivatives. The operator $\Delta$ has degree $+1$ and order at most two.

\begin{proposition}\label{prop:local-BV}
These operators satisfy
\begin{equation}\label{eq:BV-identities}
 \delta_f^2=\Delta^2=0,\qquad \delta_f\Delta+\Delta\delta_f=0.
\end{equation}
Consequently $(K_f((u)),-\delta_f+u\Delta)$ is a $2$-periodic complex. The second-order operator $\Delta$ generates the usual Schouten bracket, with the sign convention $\{\theta_i,x_j\}=\delta_{ij}$. A signed contraction by $dx_1\wedge\cdots\wedge dx_n$ identifies this complex with
\begin{equation}\label{eq:BV-DR}
 \bigl(\Omega_U^\bullet((u)),-df\wedge+u\,d_{\mathrm{dR}}\bigr)[n].
\end{equation}
The shift is a cohomological shift, and is retained after two-periodization:
for a complex $(C,d_C)$ we use $d_{C[n]}=(-1)^n d_C$.
\end{proposition}
\begin{proof}
Odd derivatives anticommute. The coefficients $f_i$ commute, and the derivatives $\partial_{x_i}$ commute, proving the first two square-zero identities. In the anticommutator the terms with two derivatives acting on the argument cancel in pairs. The remaining term is
\[
 \sum_{i,j} f_{ij}\,\partial_{\theta_j}\partial_{\theta_i},
\]
which vanishes because the Hessian is symmetric and the odd derivatives are antisymmetric. This proves \eqref{eq:BV-identities}.

For homogeneous $a,b$, use the convention
\[
 \{a,b\}=(-1)^{|a|+1}\bigl(\Delta(ab)-(\Delta a)b-(-1)^{|a|}a\Delta b\bigr).
\]
Evaluation on the generators gives $\{\theta_i,x_j\}=\delta_{ij}$, and brackets between two coordinate functions or two constant $\theta$'s vanish. The biderivation extension is the Schouten bracket, so its graded Jacobi identity follows, alternatively from the second-order and square-zero identities for $\Delta$.

Let $\omega=dx_1\wedge\cdots\wedge dx_n$. For a decomposable $p$-vector $\xi=v_1\wedge\cdots\wedge v_p$, write $\iota_\xi=\iota_{v_p}\cdots\iota_{v_1}$ and define
\[
 I_p(\xi)=(-1)^{np+p(p-1)/2}\iota_\xi\omega.
\]
A direct calculation on the ordered coordinate monomials gives
\[
 I_{p-1}(\delta_f\xi)=(-1)^n df\wedge I_p(\xi),\qquad
 I_{p-1}(\Delta\xi)=(-1)^n d_{\mathrm{dR}}I_p(\xi).
\]
For the second formula allow function coefficients of $\xi$ and differentiate them.
The factor $(-1)^{np}$ supplies the differential sign of the shift:
\[
 I(-\delta_f+u\Delta)=(-1)^n(-df\wedge+u\,d_{\mathrm{dR}})I.
\]
Exterior degree $p$ has cohomological degree $-p$, while its image
has degree $(n-p)-n=-p$ in the shifted target. Thus this is a
degree-zero chain isomorphism to \eqref{eq:BV-DR}. In particular,
\[
 H^{\bar i}(K_f((u)),-\delta_f+u\Delta)
   \simeq H^{\overline{i+n}}(\Omega^\bullet((u)),u\,d-df\wedge).
\]
The parity shift is trivial precisely when $n$ is even.
\end{proof}
The divergence operator depends on the chosen volume, and this coordinate computation is a local statement. It is compatible with the ordinary critical-chart calculation in \cite[Section~5.4]{CompanionII}; it is not a proof of global categorical descent or of a BV module on every derived Lagrangian intersection. Baranovsky--Ginzburg construct related structures on coisotropic intersections from specified deformation data \cite{BG,BF}; such input must be supplied rather than omitted.

The contraction with a local volume form identifies the BV model with
twisted de Rham forms. To use this identification on overlapping
critical charts, we now calculate its transformation under a change
of volume form.

\subsubsection{Dependence of the BV differential on a volume form}
The local operator $\Delta$ is a divergence operator. On an oriented smooth chart it can be defined by transporting the de Rham differential across contraction with a volume form. If the volume is changed from $\omega$ to $g\omega$ with $g$ invertible, the new operator satisfies
\begin{equation}\label{eq:BV-volume-change}
 \Delta_{g\omega}(\xi)=\Delta_\omega(\xi)
       +g^{-1}\iota_{dg}\xi,
\end{equation}
with the signed contraction convention of Proposition~\ref{prop:local-BV}. Indeed, the new contraction map is $g$ times the old one, and applying $d$ to $gI(\xi)$ gives the two terms. The extra contraction has order one. Thus the induced Schouten bracket is unchanged, while the actual mixed differential is conjugated by the change in volume.

The orientation bundles retain these transition functions. In particular,
the change-of-volume formula specifies how the local BV complexes
are identified on overlaps; the half-canonical data in
Theorem~\ref{thm:GS} enter this identification.

\subsection{The degeneration series and the companion comparison}
For an effective perfect complex $E_0$ on $X_0$, the normalization
triangle gives
\[
 [E_0]=[j_{1*}E_1]+[j_{2*}E_2]-[j_{S*}E_S]
                    \quad\text{in }G_0(X_0).
\]
The boundary contribution is therefore retained in any class-indexed
sum. The general coefficientwise comparison would identify a defined
smooth-fiber periodic DT complex with the oriented bar complex for
the corresponding central matching locus. Neither the structure
sheaf bar resolution alone nor multiplication of two independent
component invariants supplies this identification.

For the section-complex case of
Conjecture~\ref{conj:I-regular-comparison}, the expected series, with
one extending twist fixed, is
\[
 \sum_{m\ge1}q^m\sum_j
       \dim\mathrm{DT}_{\mathrm{coh}}^j(M_{m,b})z^j
   =\sum_{m\ge1}q^m z^{-d_m}\sum_{a=0}^{d_m}z^{2a}.
\]
Each coefficient is a finite Laurent polynomial. In the resolved
quintic this identity, including the central bar coefficient, is
proved in~\cite[Corollary 8.2]{CompanionII}. Its periodic coefficient
has rank $r_m$ in parity $d_m$, not always in even parity. For the
point-ideal sheaves, the limiting construction instead retains the
logarithmic monodromy and nearby cohomology specified in
Problem~\ref{prob:I-logarithmic}.
\clearpage
\appendix
\section{Relative Atiyah class}\label{app:checks}
This appendix records the universal maps used in the deformation
comparison of Section~\ref{sec:local}. The Atiyah class gives the
cotangent morphism; the divisor triangle and Kodaira--Spencer class
give its normal connecting map. The trace and base-line-bundle
calculations identify these maps with the dualities of
Section~\ref{sec:CY}.

\subsection{Universal Atiyah classes and the cotangent complex}\label{ssec: RelativeAtiyah}
Let $Y$ be a smooth proper $d$-fold, let $\mathfrak M$ be a locally geometric finite-presentation open in $\RPerf(Y)$, and let $\mathcal E$ be its universal perfect complex on $Y\times\mathfrak M$. Denote projection to the moduli by $q$, and projection to $Y$ by $p_Y$. The Atiyah class relative to $Y$ is
\begin{equation}\label{eq:universal-At}
 \At_{/Y}(\mathcal E):\mathcal E\longrightarrow
 \mathcal E\otimes q^*\ttL_{\mathfrak M}[1].
\end{equation}
It is the component of the universal Atiyah class that measures variation of the object in the moduli direction. The complex $\ttL_{\mathfrak M}$ here is the cotangent complex of the actual derived moduli problem, including stabilizers and higher obstruction directions.

Since $\mathcal E$ is dualizable, \eqref{eq:universal-At} can be paired with evaluation to obtain a morphism from its endomorphism complex to $q^*\ttL_{\mathfrak M}[1]$. After relative duality along $q$, it gives
\begin{equation}\label{eq:universal-cotangent-map}\IHESaliaslabel{eqn: At_PAbs}{eq:universal-cotangent-map}
 \R q_*\cRHom(\mathcal E,\mathcal E\otimes p_Y^*\omega_Y)[d-1]
 \longrightarrow\ttL_{\mathfrak M}.
\end{equation}
No trivialization of $\omega_Y$ is used to define this map.

\begin{proposition}\label{prop:universal-cotangent}\IHESaliaslabel{lem: Atiyah and -2 structure}{prop:universal-cotangent}
Put $A=\R q_*\cRHom(\mathcal E,\mathcal E)$. On the full derived perfect-object moduli locus,
\begin{equation}\label{eq:Atiyah-cotangent}
 \T_{\mathfrak M}\simeq A[1],\qquad
 \ttL_{\mathfrak M}\simeq A^\vee[-1]
 \simeq\R q_*\cRHom(\mathcal E,\mathcal E\otimes p_Y^*\omega_Y)[d-1].
\end{equation}
The morphism \eqref{eq:universal-cotangent-map} is this cotangent identification.
\end{proposition}
\begin{proof}
Deformations of a perfect complex across a square-zero extension are governed by its derived endomorphisms. At a family, the corresponding deformation functor has tangent $A[1]$; this is the moduli-of-objects tangent calculation of \cite{TV,BD2}. The universal Atiyah class is the universal derivation for that deformation functor, so its adjoint gives the dual tangent identification. Relative Serre duality gives
\[
 A^\vee\simeq\R q_*\cRHom(\mathcal E,\mathcal E\otimes p_Y^*\omega_Y)[d].
\]
Shifting by $[-1]$ proves the last expression and its agreement with \eqref{eq:universal-cotangent-map}. The construction uses evaluation and trace, so it is compatible with restriction to a base chart and with the derived base changes under discussion.
\end{proof}

For a classical truncation $M=t_0\mathfrak M$, composing the restriction of this cotangent complex with the natural morphism to $\ttL_M$ supplies the derived enhancement's obstruction-theory map. It is a two-term perfect obstruction theory only when the indicated amplitude condition is satisfied. For a smooth fourfold its natural duality expression has shift $[3]$, but its amplitude need not be $[-1,0]$.

\subsection{Relative families and the base line bundle}
Now let $\pi:\cX\to C$ be as in Definition~\ref{def:good}, and let $\mathfrak M=\RPerf(\cX/C)$ on the chosen locus. The universal perfect complex lies on $\cX_{\mathfrak M}=\cX\times_C^h\mathfrak M$, and $q:\cX_{\mathfrak M}\to\mathfrak M$ is a proper perfect morphism of relative dimension three. The relative Atiyah class over $\cX$ has target $\mathcal E\otimes q^*\ttL_{\mathfrak M/C}[1]$. The same duality construction gives
\begin{equation}\label{eq:relative-At-map}
 \R q_*\cRHom(\mathcal E,\mathcal E\otimes\omega_{q})[2]
 \xrightarrow{\sim}\ttL_{\mathfrak M/C}.
\end{equation}
Relative duality in dimension three gives the displayed shift $[2]$.

Without a relative volume form, $\omega_q$ is the pullback of the line bundle $\lambda$ from Proposition~\ref{prop:base-line}. Thus relative duality gives $A^\vee\simeq A\otimes p^*\lambda[3]$, where $p:\mathfrak M\to C$. This is the twisted pairing in \eqref{eq:twisted-relative}. Trivializing $\lambda$ is exactly what removes the twist; resolving $A$ by vector bundles cannot remove a nontrivial base line bundle.

\subsubsection{The absolute tangent of the relative moduli family}
Cotangent transitivity gives the tangent triangle
\begin{equation}\label{eq:relative-absolute-tangent}
 \T_{\mathfrak M/C}\longrightarrow\T_{\mathfrak M}
 \longrightarrow p^*T_C\xrightarrow{+1}.
\end{equation}
At a point $(b,E)$, the connecting map is a morphism
\[
 T_{C,b}\longrightarrow\ext^2_{X_b}(E,E)
\]
on the relevant cohomology, together with its derived enhancement. It sends a first-order base direction to the obstruction to deforming $E$ along that direction. The class is the composition of the Atiyah class of $E$ with the Kodaira--Spencer class of the family; this is the deformation-obstruction interpretation in \cite{HT}.

The kernel of this connecting map consists of the base directions along which the first-order deformation of $E$ exists. A nonzero value in $\ext^2_{X_b}(E,E)$ obstructs that deformation. This gives the precise relation between the deformation of the supporting variety and the deformation of the object.

\subsection{The universal normal triangle on a family of objects}
Suppose a family of perfect complexes $\mathcal E$ is given on $X\times T$ and let $i_T:X\times T\hookrightarrow Y\times T$. Applying the normal triangle relative to $T$ and then derived pushforward gives
\begin{equation}\label{eq:universal-normal}\IHESaliaslabel{eqn: ObsDecomp}{eq:universal-normal}\IHESaliaslabel{eqn: VirDecomp3}{eq:universal-normal}
 \begin{aligned}
 A_X&\longrightarrow A_Y\longrightarrow A_N[-1]\longrightarrow A_X[1],\\
 A_X&=\R q_{X*}\cRHom(\mathcal E,\mathcal E),\\
 A_Y&=\R q_{Y*}\cRHom(i_{T*}\mathcal E,i_{T*}\mathcal E),\\
 A_N&=\R q_{X*}\cRHom(\mathcal E,\mathcal E\otimes N).
 \end{aligned}
\end{equation}
Properness and perfectness in the stated setting make these perfect complexes and ensure the needed derived base changes. The connecting map is the universal version of the Atiyah-class obstruction. In the relative fiber-graph setting the normal line bundle is pulled back from $T_C$, rather than from $\omega_Y^{-1}|_X$ for an anticanonical ambient hypersurface.

Dualizing \eqref{eq:universal-normal} reverses its arrows and retains the connecting morphism. Relative Serre duality identifies the dual terms with the indicated shifted tensor twists. A chosen nullhomotopy of the connecting morphism gives a splitting at the level of complexes, while the unsplit triangle records the general comparison.

\subsubsection{The connecting morphism in the normal deformation triangle}
Consider the long exact sequence in \eqref{eq:normal-LES}. The quotient of ambient first-order deformations by fiberwise first-order deformations is the \emph{kernel} of
\[
 \Hom_X(E,E\otimes N)\longrightarrow\ext^2_X(E,E),
\]
not necessarily all of $\Hom_X(E,E\otimes N)$. Replacing the triangle by a direct sum sets this obstruction map to zero. It therefore changes the first-order deformation problem before one even reaches higher $L_\infty$ products.

For an anticanonical hypersurface $X\subset Y$, the bundle $N$ is generally positive. For a fiber of a smooth-curve family, it is a constant line bundle on $X$, with its base tangent factor. The two normal terms have different meanings and can have very different dimensions. They cannot both be denoted by $\omega_Y^{-1}|_X$ without changing the geometry.

\raggedbottom

\end{document}